\documentclass[fleqn,11pt,twoside]{article}
\usepackage{amsthm,amsthm,amssymb, color, xcolor,epsfig, graphics, subfigure}
\usepackage{amsmath, graphicx, latexsym, lscape}

\makeatletter
\newcommand{\copyrightnote}[2]{{\renewcommand{\thefootnote}{}
 \footnotetext{\small\it
\begin{flushleft}
 \copyright \ #1   #2  
\end{flushleft}}}}

\newcommand{\Name}[1]{\begin{flushleft}
                       \LARGE \bf #1
                       \end{flushleft}\vspace{-3mm}}

\newcommand{\Author}[1]{\begin{flushleft}
                       \it #1 \end{flushleft}}

\newcommand{\Address}[1]{\begin{flushleft}
                       \it #1 \end{flushleft}}

\newcommand{\Date}[1]{\begin{flushleft}
                      \small  \it #1 \end{flushleft}}

\newcommand{\evenhead}{Author \ name}
\newcommand{\oddhead}{Article \ name}

\renewcommand{\@evenhead}{
\hspace*{-3pt}\raisebox{-15pt}[\headheight][0pt]{\vbox{\hbox to \textwidth
{\thepage \hfil \evenhead}\vskip4pt \hrule}}}
\renewcommand{\@oddhead}{
\hspace*{-3pt}\raisebox{-15pt}[\headheight][0pt]{\vbox{\hbox to \textwidth
{\oddhead \hfil \thepage}\vskip4pt\hrule}}}
\renewcommand{\@evenfoot}{}
\renewcommand{\@oddfoot}{}

\long\def\@makecaption#1#2{%
  \vskip\abovecaptionskip
  \sbox\@tempboxa{\small \textbf{#1.}\ \ #2}%
  \ifdim \wd\@tempboxa >\hsize
    {\small \textbf{#1.}\ \ #2}\par
  \else
    \global \@minipagefalse
    \hb@xt@\hsize{\hfil\box\@tempboxa\hfil}%
  \fi
  \vskip\belowcaptionskip}

\newcommand{\JNMPnumberwithin}[3][\arabic]{%
  \@ifundefined{c@#2}{\@nocounterr{#2}}{%
    \@ifundefined{c@#3}{\@nocnterr{#3}}{%
      \@addtoreset{#2}{#3}%
      \@xp\xdef\csname the#2\endcsname{%
        \@xp\@nx\csname the#3\endcsname .\@nx#1{#2}}}}%
}

\renewenvironment{proof}[1][\proofname]{\par
  \normalfont
  \topsep6\p@\@plus6\p@ \trivlist
  \item[\hskip\labelsep\textbf{%
    #1\@addpunct{.}}]\ignorespaces
}{%
  \qed\endtrivlist
}

\newcommand{\resetfootnoterule} {
  \renewcommand\footnoterule{%
  \kern-3\p@
  \hrule\@width.4\columnwidth
  \kern2.6\p@}
}

\renewcommand{\footnoterule}{}

\makeatother

\newtheorem{theor}{Theorem}
\newtheorem{claim}[theor]{Claim}

\theoremstyle{definition}

\newtheorem{proposition}[theor]{Proposition}
\newtheorem{lemma}[theor]{Lemma}
\newtheorem{define}{Definition}

\newtheorem{problem}{Problem}
\newtheorem{open}[problem]{Open problem}
\newtheorem{example}{Example}
\newtheorem{implement}{Implementation}
\newtheorem{method}{Method}
\theoremstyle{remark}
\newtheorem{rem}{Remark}

\newtheorem{encoding}{Encoding}

\newcommand{\BBR}{\mathbb{R}}

\newcommand{\BBQ}{\mathbb{Q}}

\newcommand{\cP}{\mathcal{P}}

\newcommand{\gothg}{\mathfrak{g}}

\DeclareMathOperator{\Star}{Star}
\DeclareMathOperator{\Aut}{Aut}

\newcommand{\lshad}{[\![}
\newcommand{\rshad}{]\!]}

\newcommand{\schouten}[1]{\lshad {#1} \rshad}

\newcommand{\by}[1]{\textrm{{#1}}}
\newcommand{\jour}[1]{\textit{{#1}}}
\newcommand{\vol}[1]{\textbf{{#1}}}
\newcommand{\book}[1]{\textit{{#1}}}

\begin{document}

\renewcommand{\evenhead}{ {\LARGE\textcolor{blue!10!black!40!green}{{\sf \ \ \ ]ocnmp[}}}\strut\hfill 
R Buring and A V Kiselev
}
\renewcommand{\oddhead}{ {\LARGE\textcolor{blue!10!black!40!green}{{\sf ]ocnmp[}}}\ \ \ \ \ 
Kontsevich's deformation quantization of affine Poisson brackets%
}

\thispagestyle{empty}
\newcommand{\FistPageHead}[3]{
\begin{flushleft}
\raisebox{8mm}[0pt][0pt]
{\footnotesize \sf
\parbox{150mm}{{Open Communications in Nonlinear Mathematical Physics}\ \ {\LARGE\textcolor{blue!10!black!40!green}{]ocnmp[}}
\ \ Special Issue 2, 2024\ \  pp
#2\hfill {\sc #3}}}\vspace{-13mm}
\end{flushleft}}

\FistPageHead{1}{\pageref{firstpage}--\pageref{lastpage}}{ \ \ }

\strut\hfill

\strut\hfill

\copyrightnote{The author(s). Distributed under a Creative Commons Attribution 4.0 International License}

\begin{center}

{\bf {\large Proceedings of the OCNMP-2024 Conference:\\ 

\smallskip

Bad Ems, 23--29 June 2024}}
\end{center}

\smallskip

\Name{
Kontsevich's star\/-\/product up to order~7 for\\[3pt] 
affine Poisson brackets: 
where are the Riemann\\[3pt] 
zeta values\,?%
}

\Author{
R.~Buring$^{\,1}$ and 
A.\,V.\,Kiselev$^{\,2}$}

\Address{$^{1}$ 
Centre INRIA de Saclay \^Ile\/-\/de\/-\/France, B\^at.~Alan Turing, 1~rue Honor\'e\ d'Estienne d'Orves, 
F-91120 Palaiseau, France\\[2mm]
$^{2}$ 
Ber\-nou\-lli Institute for Mathematics, Computer Science and Artificial Intelligence, University of Groningen, P.O.~Box 407, 9700~AK Groningen, The Netherlands}

\Date{Received August 30, 2024; Accepted September 22, 2024}

\setcounter{equation}{0}

\begin{abstract}
\noindent 
The Kontsevich star\/-\/product admits a well-defined restriction to the class of affine --\,in particular, linear\,-- Poisson brackets; its graph 
expansion consists only of Kontsevich's graphs with in-degree $\leqslant 1$ for aerial 
vertices.
We obtain the formula $\star_{\text{aff}}\text{ mod }\bar{o}(\hbar^7)$ with harmonic propagators for 
the graph weights (over $n\leqslant 7$ aerial vertices);
we verify 
that all these weights satisfy the cyclic weight relations by Shoikhet--Felder--Willwacher, that they match the computations using the \textsf{kontsevint} software by Panzer, 
and 
the resulting affine star\/-\/product is associative modulo $\bar{o}(\hbar^7)$.

We discover that the Riemann zeta value $\zeta(3)^2/\pi^6$, which enters the harmonic graph weights (up to rationals), actually 
disappears from the analytic formula of $\star_{\text{aff}}\text{ mod }\bar{o}(\hbar^7)$ \emph{because} all the $\mathbb{Q}$-linear combinations of Kontsevich graphs near $\zeta(3)^2/\pi^6$ represent 
differential consequences of the Jacobi identity for the affine Poisson bracket, hence their contribution vanishes.
We thus derive a ready\/-\/to\/-\/use shorter formula $\star_{\text{aff}}^{\text{red}}$
mod~$\bar{o}(\hbar^7)$ with only rational coefficients.
\end{abstract}

\label{firstpage}


\section{Introduction}
Kontsevich's $\star$-product formula from \cite{MK97} works for all Poisson structures on finite\/-\/di\-men\-sio\-nal affine manifolds.
In particular, Poisson bracket's coefficients can be polynomial of arbitrary degree (which is well defined by the initial assumption that the Poisson manifold is affine).
Examples of polynomial-coefficient Poisson brackets are natural and abundant: such are the linear Kirillov--Kostant structures on the duals $\mathfrak{g}^\ast$ of Lie algebras $\mathfrak{g}$, or affine Poisson brackets for mechanical systems in external fields, or Sklyanin's quadratic Poisson bracket and the log-symplectic bracket $\{x,y\} = \tfrac{1}{2}xy$ for which $x \star y = \exp(\hbar)\, y \star x$ is conjectured~\cite[\S 1.30]{MKLefschetzLectures}, or, in fact, any bi-vector $\mathcal{P}$ with polynomial coefficients on $\mathbb{R}^2$; starting from dimension three, on $\mathbb{R}^{d \geqslant 3} \ni \boldsymbol{x}$ one has the Nambu\/-\/determinant Poisson brackets,
\begin{equation}\label{EqNambuPBr}
	\{f,g\}(\boldsymbol{x}) = \varrho(\boldsymbol{x}) \cdot \det\Bigl 
	(
	{\partial(a_1,\ldots,a_{d-2},f,g)} \bigr/ {\partial(x^1,\ldots,x^d)} \Bigr
	),
\end{equation}
for the derived Poisson bi-vectors $\mathcal{P} = \schouten{ \ldots \schouten{\varrho(\boldsymbol{x})\ \partial_{x^1} \wedge \cdots \wedge \partial_{x^d}, a_1}, \ldots, a_{d-2}}$.
On $\mathbb{R}^3$ with the inverse density $\varrho$ and Casimir $a_1=a$, \emph{affine} Nambu-determinant Poisson brackets are produced by taking either $\deg(\varrho)=1$ and $\deg(a)=1$ or $\deg(\varrho)=0$ and $\deg(a)=2$.

The problem of construction of (equivalents to) Kontsevich's star\/-\/product(s) $\star_{\text{aff}}$ for linear or affine Poisson brackets was studied in the works by 
Gutt~\cite{Gutt83}, Kathotia~\cite{Kathotia}, Ben Amar~\cite{BenAmar2003,BenAmar2007},
Dito~\cite{DitoGuttKontsevich}, and others.
The upper bound $\deg(\mathcal{P}^{ij}) \leqslant 1$ upon the degree of Poisson bracket coefficients permits an explicit construction of the \emph{affine} star\/-\/product expansion $\star_{\text{aff}}\text{ mod }\bar{o}(\hbar^k)$ for affine 
Poisson brackets.\footnote{\label{FootPeano}%
Peano's omicron notation means that all the terms at~$\hbar^0$,\ $\hbar^1$,\ $\ldots$,\ 
$\hbar^k$ are known but higher order terms at~$\hbar^{>k}$ stay `unknown'.}
The graph expansion of $\star_{\text{aff}}$ runs over the Kontsevich graphs (by their definition, built of wedges) with the bound $\leqslant 1$ for the in-degree of aerial vertices, which are inhabited by copies of the linear or affine Poisson structure.
By starting with two types of Kontsevich's graphs, namely the Bernoulli and loop graphs at every number $n$ of aerial vertices (see Definition~\ref{DefBernoulliLoopGraph} on p.~\pageref{DefBernoulliLoopGraph} below), one can calculate inductively the Kontsevich weights of all the other relevant graphs in that order $n$.
The formulas --\,for us to verify, see Discussion on p.~\pageref{SecDiscussion} in what follows\,-- were available from the paper~\cite{BenAmar2003}.

While some analytic formulas were known for calculation of particular graph weights or of the entire affine star\/-\/product $\star_{\text{aff}}\text{ mod }\bar{o}(\hbar^k)$, 
these formulas always refer to a set of conventions and normalization postulates which are however not necessarily pronounced. 
In practice, every such reasoning needs verification, to let us be certain of absence of error.
The certificate of associativity up to $\bar{o}(\hbar^k)$ for Kontsevich's affine star\/-\/product $\star_{\text{aff}}\text{ mod }\bar{o}(\hbar^k)$ is 
what we 
do 
in the follow\/-\/up paper~\cite{factor23}: we have $k=7$.

We develop several methods to control the construction of affine star\/-\/product with harmonic propagators.
The postulate of associativity is primary.
We notice that for affine Poisson brackets and affine star\/-\/products, the Kontsevich graph language restricts to graphs with aerial vertex in-degree $\leqslant 1$ \emph{both} in the star\/-\/product formula and in the associator, so that graphs with at least one higher in-degree aerial vertex can be discarded.%
\footnote{\label{FootInDegBoundInLeibniz}%
In the Leibniz graphs, which encode differential consequences of the Jacobi identity, the in-degree of wedge tops is equally bounded by $\leqslant 1$ but the in-degree of the trident vertex, standing for the Jacobiator $\tfrac{1}{2}\schouten{\mathcal{P}, \mathcal{P}}$ of the bi-vector $\mathcal{P}$, must have the bound $\leqslant 2$. This higher in-degree is reworked into the previously stated bound $\leqslant 1$ when the Jacobiator is expanded into its Kontsevich's graphs and the in-coming arrows spread by the Leibniz rule over the three pairs of internal vertices in the Jacobiator.}

Secondly, we observe that the cyclic weight relations by Shoikhet--Felder--Willwacher (from \cite[Appendix~E]{FelderWillwacher}) form a \emph{triangular} linear algebraic system with respect to the in-degree of aerial vertices. 
Consequently, these linear systems are well defined for the graph weights in $k$th order expansions $\star_{\text{aff}}\text{ mod }\bar{o}(\hbar^k)$ of the \emph{affine} star\/-\/product.
The cyclic weight relations now constrain the Kontsevich graph weights in $\star_{\text{aff}}$ by high rank systems\footnote{\label{FootMaxAttained}%
Moreover, as soon as we postulate the maximal in\/-\/degree of aerial vertices of the Kontsevich graphs in the star\/-\/product, the linear algebraic system of cyclic weight relations for the weights of graphs over $n\geqslant 2$ aerial vertices is again \emph{triangular} with respect to the number of aerial vertices where the maximal in\/-\/degree is actually attained.
The mechanism of this filtration by the vertex number is the same as the overall filtration by the maximal in\/-\/degree. Indeed, neither the in\/-\/degrees of aerial vertices nor the number of such vertices with the maximal in\/-\/degree can increase when arrows are re\/-\/directed from aerial vertices to sinks (and the sinks are cyclically permuted).
This reasoning --\,with a caveat from footnote~\ref{FootInDegBoundInLeibniz}\,-- also applies to the maximal in\/-\/degree vertex number filtration of the system of cyclic weight relations for the weights of Leibniz graphs that express the associator for~$\star$ through the Jacobi identity (cf.~\cite{factor23}).}
for every $k \geqslant 2$. 

Thirdly, we discover that the number of Kontsevich's graphs suitable for $\star_{\text{aff}}$, due to the in-degree bound $\leqslant 1$ for aerial vertices, is exponentially less than the number of Kontsevich's graphs without such bound in the same expansion order $\hbar^n$ of the full star\/-\/product $\star\text{ mod }\bar{o}(\hbar^k)$: see Table~\ref{TabKontsevich} and footnote~\ref{FootDefZeroNeutral} below.

\begin{table}[htb]
\caption{The count of admissible Kontsevich graphs in the $\star$-product.}\label{TabKontsevich}
{\small
\begin{center}
\begin{tabular}{l r r r r r r r}
\hline
	$n$ & $1$ & $2$ & $3$ & $4$ & $5$ & $6$ & $7$ \\ \hline
	\# Kontsevich graphs, generated & $1$ & $6$ & $44$ & $475$ & $6874$ & $126750$ & $2814225$ \\
	\# Kontsevich generated, nonzero & $1$ & $6$ & $38$ & $445$ & $6488$ & $122521$ & $2744336$ \\
	\# Kontsevich generated, nonzero, diff.\ order${}>0$ & $1$ & $4$ & $30$ & $331$ & $4907$ & $91694$ & $2053511$ \\
	\# Kontsevich generated, nonzero, diff.\ order${}>0$, & $1$ & $4$ & $30$ & $330$ & $4893$ & $91489$ & $2049704$ \\
	\quad connected \\
	$\quad\bullet$ of them, with in-degree(aerial vertices) $\leqslant 2$ & 1 & 4 & 30 & 265 & 2801 & 33690 & 451927 \\
	$\qquad\bullet$ of them, prime & 1 & 3 & 24 & 215 & 2327 & 28649 & 391958 \\
	$\quad\bullet$ of them, with in-degree(aerial vertices) $\leqslant 1$ & 1 & 4 & 14 & 51 & 161 & 542 & 1723 \\
	$\qquad\bullet$ of them, prime & 1 & 3 & 8 & 23 & 59 & 171 & 477 \\
	\# Kontsevich graphs (coeff${} \neq 0$ in $\star$ at $\hbar^n$) & $1$ & $4$ & $13$ & $247$ & $2356$ & $66041$ & ? \\
	$\quad\bullet$ of them, with in-degree(aerial vertices) $\leqslant 1$ & 1 & 4 & 6 & 35 & 84 & 334 & 958 \\
	\# Cyclic weight relations & $1$ & $4$ & $30$ & $331$ & $4907$ & $91694$ & $2053511$ \\ 
	Corank of linear algebraic system  & $0$ & $1$ & $11$ & $103$ & $1561$ & ? & ? \\ 
	\hline
\end{tabular}
\end{center}

}
\end{table}

This is true both before and after we know the Kontsevich graph weights (about a half of which equal zero\footnote{\label{FootDefZeroNeutral}%
By definition, a Formality graph is called \emph{zero} if it admits an automorphism (over fixed sink vertices) which makes the graph equal minus itself (so that this graph's formula is equal to minus itself, hence zero).
In contrast with the above, for a given choice of propagator (we take harmonic ones as in~\cite{MK97}), the Kontsevich weight of a certain nonzero connected Formality graph that sends arrows to all of its sinks can still `incidentally' equal zero. We say that such Formality graphs are \emph{neutral}; in the $\star$-\/product expansion, their coefficients${}=0$ (see also Definition~\ref{DefEyeOnGround}, Example~\ref{ExEyeOnGround}, and Lemmas~\ref{LemmaWeyeOnGroundVanishes},\ \ref{LemmaWdisconnectedVanishes},\ \ref{LemmaWintegrandVanishes} in~\S\ref{SecContraints} below).
Thirdly, \emph{prime} Formality graphs remain connected if all of the sinks and all of the edges to the sinks are erased 
(see~\cite[Lemma~5]{cpp} with the multiplicative formula, Eq.~(7) therein, for the weights of \emph{composite}, i.e.\ non\/-\/prime graphs).} 
at every order $n$ in $\star$ for $4 \leqslant n \leqslant 7$).
Table~\ref{TabKontsevich} reveals how small the problem is for $\star_{\text{aff}}\text{ mod }\bar{o}(\hbar^k)$, compared with the task of finding the full star\/-\/product $\star\text{ mod }\bar{o}(\hbar^k)$ with harmonic propagators.

We combine the many known constraints upon the Kontsevich graph weights in $\star_{\text{aff}}$, and we make the postulate of associativity work by restricting the associator of $\star_{\text{aff}}$ onto particular (classes of) affine Poisson brackets: e.g., onto generic affine bi-vectors $\mathcal{P} = 
(\alpha x + \beta y + \gamma)\partial_x \wedge \partial_y$ on $\mathbb{R}^2$ with Cartesian coordinates $x$ and $y$, or on the Nambu-determinant Poisson brackets~\eqref{EqNambuPBr} on $\mathbb{R}^3$ with $\deg(\varrho, a) = (1,1)$ or $\deg(\varrho, a) = (0,2)$, etc.

To calculate the few remaining master\/-\/parameters giving us 
all the graph weights at~$\hbar^n$, we use the \textsf{kontsevint} software by Panzer (freely accompanying the publication~\cite{BPP} by Banks\/--\/Panzer\/--\/Pym).
In the end, we obtain the formula $\star_{\text{aff}}\text{ mod }\bar{o}(\hbar^7)$, by construction covering the subcase of linear Poisson brackets on the duals $\mathfrak{g}^\ast$ of Lie algebras.
This seventh order expansion $\star_{\text{aff}}\text{ mod }\bar{o}(\hbar^7)$, see Appendix~\ref{AppStarAffineOriginalEncoding}, contains the Kontsevich graph weights from the integral formula with harmonic propagators (as in \cite{MK97}).
Moreover, this is the authentic \emph{gauge} of the affine star\/-\/product: indeed, the two-cycle digraph with wedges $(\mathsf{0},\mathsf{3})$,\ $(\mathsf{1},\mathsf{2})$ on two sinks $\mathsf{0},\mathsf{1}$ and two aerial vertices $\mathsf{2},\mathsf{3}$ is not gauged out.

\begin{center}
\rule{6em}{0.7pt}
\end{center}

\noindent%
The Riemann zeta values are known to appear in the graph weights of star\/-\/products.
Kontsevich discovered 
this in \cite{Operads1999};
Felder and Willwacher~\cite{FelderWillwacher} found a graph at~$\hbar^7$ in~$\star$ such that its weight contains $\zeta(3)^2/\pi^6$ up to rationals;
Banks, Panzer, and Pym~\cite{BPP} studied in depth which (multiple) Riemann zeta values show up in the Kontsevich graph weights, and they wrote a computer program
\textsf{kontsevint} (accompanying~\cite{BPP}) which expresses the weights in terms of the (un)known Riemann zeta values
(that is in terms of rationals and inert expressions like~$\zeta(3)$,\ $\zeta(5)$,\ $\zeta(7)$,\ $\ldots$, divided by suitable powers of~$\pi$).
The seventh order expansion of the affine star\/-\/product, which we report in Appendix~\ref{AppStarAffineOriginalEncoding}, contains $\zeta(3)^2/\pi^6$ in the coefficients of many graphs, in particular the one considered by Felder and Willwacher.
On the other hand, Ben Amar proved in~\cite
{BenAmar2003} that for polynomial functions, the bi-differential operator of Kontsevich's (affine) star\/-\/product for Poisson brackets on $\mathfrak{g}^\ast$ can be calculated by using only graphs with rational weights.
In other words, the (ir)rational values of Riemann zeta must be effectively absent from the formula of the affine star\/-\/product $\star_{\text{aff}}$.
Yet 
by the above, the (\emph{ir})rational Riemann zeta values \emph{are} nominally present in the affine star\/-\/product coefficients.
We bring together these two claims by assimilating the entire $\zeta(3)^2/\pi^6$-part of the genuine Kontsevich affine star\/-\/product $\star_{\text{aff}}\text{ mod }\bar{o}(\hbar^7)$ with harmonic propagators into differential consequences of the Jacobi identity;
in the graph expansion
, these differential consequences are realized by Leibniz graphs with one trident vertex for the Jacobiator and possibly with other wedge tops for copies of the Poisson bi\/-\/vector.

Let us emphasize that not only the $\zeta(3)^2/\pi^6$-part but more terms can be reduced at $\hbar^{\leqslant 7}$ in the affine star\/-\/product; already at $\hbar^3$ some coefficients are changed when one graph is eliminated.
Thus we reduce the graph expansion from Appendix~\ref{AppStarAffineOriginalEncoding} to a 
shorter, easily usable formula $\star_{\text{aff}}^{\text{red}}\text{ mod }\bar{o}(\hbar^7)$ in Appendix~\ref{AppStarAffineFormula} based on the graph expansion in Appendix~\ref{AppStarAffineReducedEncoding}.
The resulting analytic expression is valid for any affine or linear Poisson bracket.%
\footnote{The formula is provided in the authentic gauge: no gauge transform was applied to the affine star\/-\/product after the removal of identically zero part. The removed analytic expressions equal a linear differential operator acting on the Jacobiator of the Poisson bivector.}
This 
analytic formula $\star_{\text{aff}}^{\text{red}}\text{ mod }\bar{o}(\hbar^7)$ can be used at once in applications of deformation quantization in physics and other domains of mathematics such as Lie theory. 

\begin{rem}\label{RemReferee1BasesArtifacts}
The first referee points out that already in Kontsevich's original paper~\cite{MK97} 
it was proven that for linear Poisson brackets on~$\gothg^*$, 
the resulting non\/-\/commutative unital associative algebra is canonically isomorphic to the universal enveloping algebra of the associated Lie algebra~$\gothg$, 
whose standard presentation is defined over the field of rationals~$\BBQ$.
In this sense it is known that the structure of the resulting algebra ultimately involves no 
nontrivial multiple zeta values (MZVs), 
so any MZVs that appear in the $\star$-\/product are artifacts of the specific choice of basis for the (Lie) algebra.  
It remains to be explored 
whether this choice of basis has any physical sense; 
this depends on how strictly one interprets Bohr's correspondence principle
in deformation quantisation.
\end{rem}

\begin{rem}\label{RemZetaPersistsLogSympl}
The Riemann zeta value $\zeta(3)^2 / \pi^6$ \emph{cannot be} fully eliminated from the Kontsevich $\star$-\/product (with harmonic propagators, in its authentic gauge). Indeed, a counterexample is provided by the log\/-\/symplectic bracket
$\{x,y\}=\tfrac{1}{2}xy$ on~$\BBR^2$, see~\cite{BPP}: 
the transcendental number $\zeta(3)$ does show up at~$\hbar^6$ in the $\star$-\/product from~\cite{MK97} evaluated at this specific, quadratic\/-\/nonlinear Poisson bracket.
\end{rem}

The third question, which we answer in~\cite{factor23}, 
is: How does the associativity work for the full star\/-\/product\,? 
We establish that the mechanism of star\/-\/product associativity from~\cite{MK97} starts working from order $7$ onward differently from the orders $0$, $\ldots$,~$6$.
Namely, at low expansion orders the Leibniz graph realization of the associator, that is, an expression of the associator in terms of the Jacobi identity and its differential consequences, was obtained instantly from the Kontsevich graphs showing up 
in the associator after collecting similar terms.
We observe that at order $7$, this instant realization is no longer the case for the affine star\/-\/product; 
to achieve a sufficient set of Leibniz graphs, the first layer of neighbor Leibniz graphs must be taken into account.
(See~\cite{factor23} 
and~\cite{JPCS17,cpp} for the iterative scheme.)
This effect persists at order $7$ for the full star\/-\/product with generic Poisson structures.
For the reduced affine star\/-\/product, the need of more layers for Leibniz graphs is even stronger.

In conclusion, the mechanism of Kontsevich's Formality theorem does require the use of Leibniz graphs which are \emph{not} instantly recognized from the associator, for the star\/-\/product to be associative.
The ``invisible'' Riemann zeta values act as the placeholders for Kontsevich's graphs in the star\/-\/product (and then for those Kontsevich's graphs in the associator which are used to build the Leibniz graphs certifying the associativity).
\begin{center}
\rule{6em}{0.7pt}
\end{center}

Our main results are summarized in the following set of propositions.


\textbf{1.}\quad
We verify the associativity up to $\bar{o}(\hbar^6)$ for the full Kontsevich star\/-\/product, known modulo $\bar{o}(\hbar^6)$ from Banks--Panzer--Pym \cite{BPP} for arbitrary (non)linear Poisson brackets and arbitrary arguments, by realizing (every homogeneous tri-differential component of) the associator as \emph{a} sum of Leibniz graphs from the $0$th layer, that is the Leibniz graphs produced at once by contracting edges in the Kontsevich graphs from the associator 
(see Proposition~4 in~\cite{factor23}).

\textbf{2.}\quad
The Kontsevich $\star$-product admits a restriction to the class of Poisson brackets with \emph{affine} coefficients on finite-dimensional affine manifolds (e.g., such are the Kirillov\/--\/Kostant Poisson brackets on the duals of Lie algebras: their coefficients are strictly linear without constant terms).
In section~\ref{SecStarAffine}, see Proposition~\ref{PropStarAffine} on p.~\pageref{PropStarAffine} below,
we obtain the expansion $\star_{\text{aff}}$ mod $\bar{o}(\hbar^7)$ of \emph{affine} Kontsevich's star\/-\/product: in all the Kontsevich graphs in it the in-degrees of aerial vertices are bounded by $\leqslant 1$.
The graph expansion $\star_{\text{aff}}$ mod $\bar{o}(\hbar^7)$ is contained in 
Appendix~\ref{AppStarAffineOriginalEncoding};
at $\hbar^6$ and $\hbar^7$,
the coefficients of many Kontsevich graphs in $\star_{\text{aff}}$ mod $\bar{o}(\hbar^7)$ contain $\zeta(3)^2/\pi^6$ times a nonzero rational factor, plus a rational part.

\textbf{3.}\quad
We discover that in both the orders $\hbar^6$ and $\hbar^7$ in $\star_{\text{aff}}$ with the harmonic graph weights, the entire coefficient of $\zeta(3)^2/\pi^6$, itself a $\mathbb{Q}$-linear combination of Kontsevich graphs, assimilates into a linear combination of Leibniz graphs (doing so at once, without need of the 1st layer of Leibniz graphs).
In consequence, whenever the affine star\/-\/product $\star_{\text{aff}}$ mod $\bar{o}(\hbar^7)$ is restricted to an arbitrary affine Poisson structure, the constant $\zeta(3)^2/\pi^6$ does not appear at all in the resulting analytic expression.
In fact, more terms in $\star_{\text{aff}}$ mod $\bar{o}(\hbar^7)$ can be absorbed into Leibniz graphs.
Namely, in section~\ref{SecStarAffineReduced} we obtain and in Appendix~\ref{AppStarAffineReducedEncoding} we list the affine Kontsevich graph expansion $\star_{\text{aff}}^\text{red}$ mod $\bar{o}(\hbar^7)$, at once excluding the part which is now known to vanish identically: 
there remain only $326$ 
nonzero rational coefficients of Kontsevich graphs (at all orders, up to $\hbar^7$),
in contrast with the original graph expansion of the Kontsevich star\/-\/product $\star_{\text{aff}}$ mod $\bar{o}(\hbar^7)$ in which the Kontsevich integral formula yields $1423$ nonzero $\mathbb{Q}$-linear combinations of $1$ and $\zeta(3)^2/\pi^6$ for the Kontsevich graph coefficients.%
\footnote{In the above reduction of $\star_{\text{aff}}$ mod $\bar{o}(\hbar^7)$ none of the Kontsevich graph weights was altered or were redefined; the reduction of the number of terms which effectively contribute to the star\/-\/product of functions and to its associator is due to a revealed property of those graphs and their Kontsevich weights.}

\textbf{4.}\quad
We verify the associativity up to $\bar{o}(\hbar^7$) of the affine Kontsevich star\/-\/product, known modulo $\bar{o}(\hbar^7)$ for arbitrary affine Poisson brackets and arbitrary arguments,
by 
realizing (every homogeneous tri-differential component of) the associator as \emph{a} sum of Leibniz graphs.
We establish that for the tri-differential orders $\{(3,3,2)$, $(2,3,3)$, $(3,2,3)$, $(2,4,2)\}$, the $0$th layer of Leibniz graphs is not enough for any such factorization to exist, yet solutions appear in presence of the $1$st layer of Leibniz graphs in each of the four exceptional cases; see Proposition~5 in~\cite{factor23},
and see the \textit{Proof scheme \textup{(}for the reduced affine star\/-\/product} $\star_{\text{aff}}^{\text{red}}$ mod $\bar{o}(\hbar^7)$) on p.~6 in~\cite{factor23}
specifically about the properties of the associator for the reduced affine star\/-\/product $\star_{\textup{aff}}^{\textup{red}}$ mod $\bar{o}(\hbar^7)$ with only rational coefficients.
We finally deduce that the indispensability of the first layer of Leibniz graphs is such that it carries on to any factorization of the associator at~$\hbar^7$ for the full star\/-\/product modulo $\bar{o}(\hbar^7)$; see Proposition~6 in~\cite{factor23}.

\smallskip
This paper is structured as follows.
In section~\ref{SecStarAffine} we list practical methods to constrain the coefficients of Kontsevich's graphs in (affine) star\/-\/product expansion and we obtain the graph encoding of the seventh order expansion $\star_{\text{aff}}\text{ mod }\bar{o}(\hbar^7)$ of the affine star\/-\/product with harmonic propagators.
In section~\ref{SecStarAffineReduced} we reduce this affine star\/-\/product 
by making use of the Jacobi identity, so that for each Poisson bracket with linear or affine coefficients, the resulting formula of $\star_{\text{aff}}^{\text{red}}\text{ mod }\bar{o}(\hbar^7)$ remains the same as it was before reduction; the much-shortened graph encoding and the ready-to-use 
analytic formula $\star_{\text{aff}}^{\text{red}}\text{ mod }\bar{o}(\hbar^7)$ are given in Appendices~\ref{AppStarAffineReducedEncoding} and~\ref{AppStarAffineFormula} respectively.
(In~\cite{factor23} 
we discover that Kontsevich's mechanism of associativity for the full or affine star\/-\/product works differently at orders~$0$,\ $\ldots$,\ $6$ against orders seven and higher.
The difference is this: at low orders the vanishing of the associator thanks to the Jacobi identity is immediately recognizable, whereas starting from order~7, it can take several steps to accumulate all the needed Leibniz graphs (with the Jacobiator) to represent higher order terms of the associator for the star\/-\/product.)
The paper concludes with a reference to the earlier 
work of Ben Amar about the Kontsevich star\/-\/products for linear Poisson structures; we compare the old predicted values of weights for some 
graphs such as the Bernoulli graphs with the true values now obtained and cross-verified by a combination of methods.%
\footnote{NB: There are two opposite sign conventions for the first Bernoulli number.}

\section{The Kontsevich star\/-\/product with harmonic propagators for affine Poisson brackets: expansion $\star_{\text{aff}}$ mod $\bar{o}(\hbar^7)$}
\label{SecStarAffine}

\begin{define}[{cf.~\cite[Def.\,3]{cpp}}]\label{DefMKweight}
The formula for the harmonic weights~$W_\Gamma \in \mathbb{R}$ is given in~\cite[\S6.2]{cpp}
; it is \[ W_\Gamma = \Bigg(\prod_{k=1}^n \frac{1}{\#\!\Star(k)!}\Bigg) \cdot \frac{1}{(2\pi)^{2n+m-2}} \int_{\bar{C}^+_{n,m}} \bigwedge \limits_{e \in E_\Gamma} d\phi_e, \]
where $\#\Star(k)$ is the number of edges \emph{star}ting from vertex $k$, $d\varphi_e$ is the ``harmonic angle'' differential $1$-form associated to the edge $e$, and the integration domain $\bar{C}^+_{n,m}$ is the connected component of $\bar{C}_{n,m}$ which is the closure of configurations where points $q_j$, $1 \leqslant j \leqslant m$ on $\mathbb{R}$ are placed in increasing order: $q_1 < \cdots < q_m$.
For convenience, let us also define\footnote{\label{FootWhatBPPcompute}%
It is the values $w_\Gamma$ instead of $W_\Gamma$ which are calculated by
the \textsf{kontsevint} software
~\cite{BPP}
.}
\[ w_\Gamma = \bigg(\prod_{k=1}^n \#\!\Star(k)! \bigg) \cdot  W_\Gamma. \]
The convenience is that by summing over labelled graphs $\Gamma$, we actually sum over the equivalence classes $[\Gamma]$ (i.e. over unlabeled graphs) with multiplicities $(w_\Gamma/W_\Gamma)\cdot n! / \#\!\Aut(\Gamma)$.
The division by the volume $\#\!\Aut(\Gamma)$ of the symmetry group eliminates the repetitions of graphs which differ only by a labeling of vertices but, modulo such, do not differ by the labeling of ordered edge tuples (issued from the vertices which are matched by a symmetry).
\end{define}

To find the weights of Kontsevich's graphs in the affine star\/-\/product, we use the same strategy as in~\cite{cpp} now implemented in \textsf{gcaops}~\cite{gcaops}.
We extend the set of constraints upon the graph coefficients by several useful lemmas, and simultaneously, we reduce the problem by filtering sums of graphs by maximal in\/-\/degree of their aerial vertices, also during insertion of a Kontsevich graph into a sink of another Kontsevich graph when the associator is formed.

\subsection{Sources of constraints for the Kontsevich graph weights}\label{SecContraints}

Let us summarize the methods to constrain weights of Kontsevich graphs in the affine star\/-\/product modulo $\bar{o}(\hbar^k)$.


\begin{lemma}
Under a relabeling of aerial vertices, the weight $w(\Gamma)$ of a Kontsevich graph $\Gamma$ does not change.
(Hence, a choice of labeling can be arbitrary.)
Under a swap of ordering $L \rightleftarrows R$ of outgoing edges at an aerial vertex of a Kontsevich graphs, its weight changes sign.
(Hence, the weight of a \emph{zero} graph vanishes, see \cite[Example~5]{cpp}.)
\end{lemma}

\begin{lemma}
For the normalization of Kontsevich graph weights as in \cite{cpp}, where we set $w(\Gamma) = (\prod_{k=1}^n \# \operatorname{Star}(v_k)!)\, W_\Gamma$ with respect to the Kontsevich formula~$W_\Gamma$ for graphs on $n$ aerial vertices $v_1, \ldots, v_n$, 
the weight of a composite Kontsevich graph $\Gamma = \Gamma_1 \bar{\times} \Gamma_2$ on $n = n_1 + n_2$ aerial vertices is (by the graph weights' multiplicativity) equal to the product of the weights: $w(\Gamma) = w(\Gamma_1) \cdot w(\Gamma_2)$.
\end{lemma}

The multiplicativity of weights yields (inhomogenenous) linear relations for the weights of composite graphs at order $n$ as soon as the weights of graphs at lower orders $k < n$ are known.


\begin{lemma}
Upon flipping a Kontsevich graph on $n$ aerial vertices, i.e.\ interchanging the two sinks, its weight is multiplied by~$(-1)^n$.
\end{lemma}

\begin{lemma}
The Shoikhet--Felder--Willwacher cyclic weight relations from \cite[Appendix~E]{FelderWillwacher}
restrict to the subset of \emph{affine} graphs,
which are selected in the set of all Kontsevich graphs by the bound, in-degree $\leqslant 1$, for aerial vertices. 
\end{lemma}

\begin{proof}
Re-directing edges --\,now, to ground vertices\,-- does not increase the in-degree of aerial vertices,
neither does so a cyclic permutation of the ground vertices.
Hence the linear algebraic system of Shoikhet\/--\/Felder\/--\/Willwacher cyclic weight relations is triangular, filtered by the bounds for aerial vertex in-degrees.
\end{proof}

\begin{define}\label{DefEyeOnGround}
A Kontsevich graph with an ``eye-on-ground'' is a Kontsevich graph containing a $2$-cycle between aerial vertices such that both vertices in the $2$-cycle are connected to the same ground vertex.
\end{define}

\begin{example}[see {\cite[Example 26]{cpp}}]
\label{ExEyeOnGround}
Here is an example of a Kontsevich graph with an eye-on-ground:
$\text{\raisebox{-12pt}{
\unitlength=0.70mm
\linethickness{0.4pt}
\begin{picture}(15.00,23.67)
\put(2.00,5.00){\circle*{1.33}}
\put(13.00,5.00){\circle*{1.33}}
\put(-3.33,13.00){\circle*{1.33}}
\put(7.67,13.00){\circle*{1.33}}
\put(7.67,13.00){\vector(-2,-3){5.33}}
\put(-3.33,13.00){\vector(2,-3){5.33}}
\bezier{64}(-3.33,13.00)(1.67,7.00)(7.50,13.00)
\bezier{64}(7.67,13.00)(1.67,18.33)(-2.50,13.00)
\put(6.5,12.00){\vector(1,1){0.67}}
\put(-2.00,14.00){\vector(-1,-1){0.67}}
\put(13.00,17.00){\circle*{1.33}}
\put(13.00,17.00){\vector(0,-1){12.33}}
\put(13.00,17.00){\vector(-3,-2){5.33}}
\end{picture}
}}$
From \cite[Example 26]{cpp} we know that its weight \texttt{w\_3\_8} vanishes.
\end{example}

\begin{lemma}\label{LemmaWeyeOnGroundVanishes}
The weight of a Kontsevich graph on two ground vertices with an ``eye on ground'' vanishes.
\end{lemma}

For example, we directly verify this for all Kontsevich graphs on $n \leqslant 6$ aerial vertices by using Panzer's \textsf{kontsevint} software.

\begin{proof}
This is readily seen from the count of dimensions.
Consider the ``eye-on-ground'' subgraph with two aerial vertices, one sink, and four edges.
The number of angles between an outgoing edge and $i\infty$ equals four, with as many differentials of these four angles.
By using the affine group, move the sink to $0$ on the real line.
The two aerial vertices remain inside the hyperbolic upper half\/-\/plane.
Rescale the entire picture such that (at least) one aerial vertex is at height $y=1$.%
\footnote{Or fix some of the remaining parameters otherwise, e.g., place an aerial vertex on a unit circle.}
The remaining parameters are: the $x$-coordinate of the $(y=1)$-vertex and Cartesian coordinates $(x,y)$ of the other vertex.
But the four angles for the outgoing edges within this subgraph, hence their differentials, are completely determined by the (differentials of the) three parameters which we have just counted.
The four-form expressed using three differentials vanishes, hence the integrand in Kontsevich's weight formula for the entire graph also vanishes, and so does the graph weight.
\end{proof}

\begin{lemma}\label{LemmaWdisconnectedVanishes}
The weight of a disconnected graph vanishes.
\end{lemma}

Note that such graphs can consist of two components, each standing on one sink, but also e.g., of one component standing on two sinks and a purely aerial component.%
\footnote{Purely aerial components of graphs in Kontsevich's calculus correspond to the vacuum parts of Feynman diagrams.}

\begin{proof}
Again, this is seen by the count of dimensions.
Using action by the affine group elements, bring the leftmost sink to the origin $0 \in \mathbb{R}$ and rescale the graph picture in the upper half-plane $\mathbb{H}^2$ such that the $y$-coordinate of some aerial vertex is any given positive number.
Now, in the other component(s) of the graph, i.e.\ in the subgraph(s) not connected with the sink in the origin, the angles between edges and $i\infty$ --\,hence their differentials\,-- stay invariant under local (to not hit a vertex of any other component) horizontal translations along $\mathbb{R} = \partial\mathbb{H}^2$.
The $x$-coordinate $x_0$ of a vertex in the movable component, parametrizing the entire component's location, can be taken as one of the coordinates on the space $C_{n,m}$ of vertex configurations.
But, as said, for all the edges in this detached subgraph, the differentials in the respective angles in Kontsevich's formula can only depend on the pairwise differences of the $x$-coordinates (and on the not yet fixed $y$-coordinates) of vertices; neither the constant $x_0$ nor its differential occurs in the integrand.
Consequently, the differentials of fewer parameters on the configuration space $C_{n,m}$ show up in the top-degree form with the differentials of all the edge angles.
By the pigeonhole principle, 
at least one differential of a parameter on $C_{n,m}$ is repeated twice, hence the integrand is identically zero, and so is the entire graph weight.
\end{proof}

\begin{rem}
Connected \emph{affine} Kontsevich graphs standing on both sinks cannot contain an ``eye-on-ground'', due to the in-degree bound $\leqslant 1$ for either aerial vertex in the $2$-cycle (see Example~\ref{ExEyeOnGround}).
The ``eye-on-ground`` subpattern essentially works for Kontsevich graphs which are not affine, e.g., for quadratic Poisson brackets (see Remark~\ref{RemLogSymplStar} below).
\end{rem}

\begin{lemma}\label{LemmaWintegrandVanishes}
The graph weight vanishes if the integrand vanishes in Kontsevich's integral formula.%
\footnote{More graphs can have weight equal to zero sporadically, that is, when the Kontsevich formula with a nonzero integrand integrates to zero for an agreed choice of the propagator (here, harmonic). Examples of such \emph{neutral} graphs are found in \cite[Figure~2 and Example~26]{cpp}.}
\end{lemma}

The zero graphs, which can be affine, and the connected ``eye-on-ground'' graphs, which cannot be affine,
are natural examples for the integrand to vanish
(see footnote~\ref{FootDefZeroNeutral} on p.~\pageref{FootDefZeroNeutral} for relevant definitions).

\begin{rem}\label{RemLogSymplStar}
Kontsevich conjectures in~\cite[\S 1.30]{MKLefschetzLectures} that for the quadratic log-\/symplectic Poisson bracket $\{x,y\} = \tfrac{1}{2}xy$ on~$\mathbb{C}^2$, its ``quadratic'' star\/-\/product --\,with in\/-\/degrees~$\leqslant 2$ for all aerial vertices of Kontsevich graphs in~$\star$\,-- satisfies the identity $x \star y = \exp(\hbar) y \star x$ at all orders of expansion in~$\hbar$.
When one constructs (a quadratic or higher order) Kontsevich's star\/-\/product with harmonic propagators, there are two options.
The first option is to verify Kontsevich's conjecture as soon as every 
next order of expansion is reached in the quadratic star\/-\/product.
The other option is to \emph{postulate} this conjecture and thus constrain the Kontsevich graph weights by one more relation at each order in~$\hbar$, so that the rank of the linear algebraic system is higher, the system is solved faster, and there remain fewer master\/-\/parameters.
(But then, one must indicate that the Kontsevich graph weights were obtained modulo the unproven conjecture.)%
\footnote{To the best of our knowledge, E.~Panzer (2019) produced the seventh order formulas of~$x \star y$ and~$y \star x$ for the log-\/symplectic structure, that is, by using Kontsevich's graphs with aerial vertex in\/-\/degrees~$\leqslant 2$ and with either sink's in\/-\/degree~$\equiv 1$; up to~$\hbar^7$, Kontsevich's conjecture holds true.
It is clear that while knowing the quadratic star\/-\/product of two coordinate functions, one does not yet have the quadratic star\/-\/product of two arbitrary arguments, nor can one certify its associativity. Therefore, the seventh order expansion of Kontsevich's star\/-\/product for the log-\/symplectic Poisson bracket in particular and for arbitrary quadratic Poisson structures remains an open problem.}
\end{rem}

Finally, the postulate of associativity for the star\/-\/product puts restrictions on the graph weights appearing in~it.

\begin{method}
\label{Method2D}
The Kontsevich affine star\/-\/product $\star_{\text{aff}}$ is associative for all Poisson structures $P = (\alpha x + \beta y + \gamma)\,\partial_x \wedge \partial_y$ on $\mathbb{R}^2$ with coordinates $x,y$.
In particular, the differential operator $A(f,g,h) = (f \star g) \star h - f \star (g \star h) \text{ mod }\bar{o}(\hbar^7)$ acting on $f \otimes g \otimes h \in C^\infty(\mathbb{R}^2)^{\otimes 3}$ has coefficients which vanish as polynomials in $x,y,\alpha,\beta,\gamma$.
\end{method}

In practice, Method~\ref{Method2D} allowed us to determine two graph weights at $\hbar^7$ in $\star_{\text{aff}}$ after the application of all the aforementioned other methods (when the rank of the linear algebraic system was already high).
Method~\ref{Method2D} is not so powerful but fast.

\begin{method}
\label{Method3D}
The Kontsevich affine star\/-\/product $\star_{\text{aff}}$ is associative for all Nambu--Poisson structures $P = \schouten{\varrho\,\partial_x \wedge \partial_y \wedge \partial_z, a}$ on $\mathbb{R}^3$, where $(\varrho, a)$ are polynomials of degrees $\deg(\varrho, a) = (1,1)$ or $(0, 2)$.
In particular, the differential operator $A(f,g,h) = (f \star g) \star h - f \star (g \star h) \text{ mod }\bar{o}(\hbar^7)$ acting on $f \otimes g \otimes h \in C^\infty(\mathbb{R}^3)^{\otimes 3}$ has coefficients which vanish as polynomials in $x,y,z$ and in the coefficients of polynomials $\varrho$ and $a$.
\end{method}

In contrast with Method~\ref{Method2D}, Method~\ref{Method3D} usually increases the rank a lot yet it is much slower.

\begin{rem}
When constructing star\/-\/products for polynomial Poisson brackets of a given (prescribed) degree, we can restrict the star\/-\/product ansatz and its associator to natural classes of polynomial Poisson brackets of that particular degree: e.g., by using the Nambu\/-\/determinant Poisson structures on~$\mathbb{R}^3$ with suitable degree polynomials $\varrho$,~$a$.
Likewise, when constructing the full star\/-\/product for generic Poisson brackets, one can increase the rank of the linear algebraic system upon the graph weights by using Method~3 from \cite{cpp}, that is, by restricting the $\star$-\/product ansatz and its associator to generic Nambu\/-\/determinant Poisson structures on~$\mathbb{R}^3$ and then dealing with differential polynomials in~$\varrho$ and~$a$.
\end{rem}

The master\/-\/parameters that remain 
are now calculated directly by using Panzer's program \textsf{kontsevint} for computation of Kontsevich's weights of graphs (see \cite{BPP}).%
\footnote{The paper \cite{BPP} contained a typo in the value of the weight for the affine graph on $n=7$ aerial vertices which Felder and Willwacher considered in \cite{FelderWillwacher} in the context of Riemann zeta. We detect that the correct weight of that graph is equal to $\frac{13}{2903040} - \frac{1}{256}\zeta(3)^2/\pi^6$; see the erratum to \cite{BPP}.\label{FootWeightCorrection}}

Independently, the graph weights can be calculated directly by applying the contour integral technique to the Kontsevich formula, see \cite[Appendix A.1]{cpp}; this method becomes increasingly difficult for Kontsevich graphs on many vertices; some intermediate steps are computer-assisted: integrals are evaluated symbolically or numerically.


So far, we summarized specific methods to obtain 
the values of weights for non-specific Kontsevich's graphs in the (affine, etc.) star\/-\/product ansatz.
Still, for the Bernoulli and loop graphs showing up in the affine star\/-\/product for linear or affine Poisson brackets, Ben Amar developed a technique for calculation of their exact weight values (see~\cite
{BenAmar2003} and Discussion in this paper).
The weights for these two classes of Kontsevich graphs serve a base in the inductive calculation of the affine star\/-\/product of two arbitrary polynomials.




\begin{rem}
The Formality theorem in~\cite{MK97} (see also~\cite{Ascona96}) guarantees that the (affine or full) star\/-\/product is associative because its associator can be realized as the Kontsevich graph expansion for a sum of Leibniz graphs, each with the Jacobiator in a trident vertex (cf.\ footnote~\ref{FootInDegBoundInLeibniz} on p.~\pageref{FootInDegBoundInLeibniz}).
The weights of Leibniz graphs in the authentic factorization (which the Formality theorem suggests) themselves are determined by the Kontsevich integral formula,
and they too satisfy the cyclic weight relations, now adapted to the Leibniz graphs on $3$ sinks.
These canonical Leibniz graph weight values and cyclic weight relations upon them can in hindsight constrain the coefficients of Kontsevich's graphs in the (affine) star\/-\/product ansatz.
We say that the above would be \emph{the strong solution} to the associator factorization problem.
However, 
it suffices for us to certify the associativity of $\star_{\text{aff}}\text{ mod }\bar{o}(\hbar^7)$ by realizing each tri-differential order part 
of its associator as \emph{a} sum of Leibniz graphs, i.e.\ regardless of their Kontsevich's weights and of their constraint by the cyclic weight relations.
In this way, we report in~\cite{factor23} a solution of the \emph{weak} factorization problem for the verification of affine star\/-\/product associativity. 
\end{rem}

\subsection{The expansion $\star_{\text{aff}}\text{ mod }\bar{o}(\hbar^7)$ of Kontsevich's star\/-\/product for affine Poisson brackets: Riemann $\zeta(3)^2/\pi^6$ still present}

For the class of Poisson brackets with \emph{affine} coefficients (whose higher derivatives vanish identically), e.g., the Kirillov\/--\/Kostant linear brackets, we advance to the seventh order in the expansion of the Kontsevich $\star$-product.

Indeed, the restriction of Kontsevich's $\star$-product to the spaces of affine Poisson brackets is well defined: all the Kontsevich graphs in $\star_{\text{aff}}$ mod $\bar{o}(\hbar^n)$ and in its associator up to $\bar{o}(\hbar^n)$ only have aerial vertices with in-degree $\leqslant 1$.
The linear algebraic system of the Shoikhet--Felder--Willwacher cyclic weight relations is by construction triangular with respect to the weights of Kontsevich graphs (in $\star$) with an overall bound for the in-degrees of aerial vertices.
(The linear system of cyclic weight relations is also triangular with respect to the in-degrees of aerial vertices in the Leibniz graphs which can be used to express the associator via differential consequences of the Jacobi identity.)

\begin{proposition}
\label{PropStarAffine}
The weighted graph encoding for the analytic formula of Kontsevich's affine star\/-\/product $\star_{\textup{aff}}$ \textup{mod} $\bar{o}(\hbar^7)$ --\,in particular, for all the Kirillov\/--\/Kostant Poisson brackets, linear on the duals~$\mathfrak{g}^*$ of finite\/-\/dimensional Lie algebras\,-- is given in Appendix~\ref{AppStarAffineEncoding}.
There are $1423$ nonzero Kontsevich weights of affine Kontsevich graphs in $\star_{\textup{aff}}$ \textup{mod} $\bar{o}(\hbar^7)$ near the graphs 
at all orders $\leqslant 7$ overall.
The multiple zeta value $\zeta(3)^2/\pi^6$ starts appearing in the weights at $n \geqslant 6$ vertices.%
\footnote{The Kontsevich weight of the Felder\/--\/Willwacher affine graph from~\cite{FelderWillwacher} 
equals $\frac{13}{2903040} - \frac{1}{256}\zeta(3)^2/\pi^6$, thus now correcting a typo in the \textup{\textsf{kontsevint}} program description~\cite{BPP} by Banks--Panzer--Pym.}
\end{proposition}

\begin{proof}[Proof scheme]
The ansatz for $\star_{\textup{aff}}$ \textup{mod} $\bar{o}(\hbar^n)$ contains, at $\hbar^7$, $1731$ affine Kontsevich graphs with in-degree $\leqslant 1$ of aerial vertices; their Kontsevich weights are constrained by the elementary properties \textup{(}such as mirror reflections, whence \emph{basic} graphs\textup{)}, by the weights multiplicativity \textup{(}whence the \emph{prime} graphs\textup{)}, by the vanishing statements for the Kontsevich graphs which are disconnected over the sinks and for the Kontsevich graphs which contain a triangle subgraph standing on a sink, and by the cyclic weight relations: the corank of the merged linear algebraic system upon the $1731$ unknowns equals~$76$.
The values of the $76$ master parameters\footnote{\label{FootRestrictTo2Daffine}%
The restriction of the associator for $\star_{\textup{aff}}$ \textup{mod} $\bar{o}(\hbar^7)$ to a generic affine Poisson bracket $P = (ax + by + c)\partial_x \wedge \partial_y$ on $\mathbb{R}^2$ decreases the corank down to $74$.
This gain --\,by only two parameters\,-- was only discovered \emph{after} the 76 master parameters were known.
The second referee inquires whether using higher dimensional linear and affine Poisson brackets on~$\BBR^{\geqslant 3}$ would produce more constraints --- compared with just two new equations from the underlying~$\BBR^2$. We learn 
that at a given order of~$\hbar$ (here $k=7$), by increasing the dimension of the underlying Poisson manifold~$\BBR^d$ one can increase the rank (beyond what one has from~$\BBR^2$) but this improvement, as $d$~grows, soon saturates because new constraints upon the graph weights repeat each other.}
\textup{(}themselves the weights of certain affine Kontsevich graphs on $n=7$ aerial vertices in the affine star\/-\/product $\star_{\textup{aff}}$ \textup{mod} $\bar{o}(\hbar^7)$\textup{)} have been computed using the \textup{\textsf{kontsevint}} program by Banks\/--\/Panzer\/--\/Pym; these values are listed in \cite[Part I, \S 3.7, Cell 54]{BuringDisser}.
\end{proof}

\begin{rem}\label{RemUseKontsevintTooMuch}
The second referee argues that it would be better to not compute only precisely the corank number of weights, but more of them.
(In fact, this is 
what we did, in view of footnote~\ref{FootRestrictTo2Daffine}: the true corank was $74 < 76$.)
Every extra integral computed gives us a more overdetermined system of equations used.
The main counter\/-\/argument against doing that too much in practice is the time needed to run the excessive evaluations of non\/-\/master graph weights.

Our core strategy in~\cite{cpp} and in the present paper is to combine all of the known constraints upon the weights of Kontsevich graphs in the $\star$-\/product (and upon the Kontsevich weights of Leibniz graphs that express the associator) in order to increase the rank of the algebraic system and thus reduce the remaining number of `time\/-\/expensive' master parameters. By realizing this strategy, we experimentally detect how precisely much does the corank drop due to a combination of the sets of constraints of different nature.
\end{rem}

\begin{open}\label{OpenRelMZV}
It has been proven in~\cite{BPP} that all graphs' coefficients in Kontsevich's universal formula of the star\/-\/product with harmonic propagators are rational linear combinations of multiple zeta values (MZVs). The property of $\star$-\/product's associativity for every specific Poisson bracket and the postulate of $\star$-\/product's associativity for arbitrary Poisson brackets (so that the Formality theorem from~\cite{MK97} guarantees the expression of the associator in terms of the Leibniz graphs) constrain the MZVs at hand by linear relations over~$\BBQ$.
Which (families of) linear relations between the MZVs in~\cite{KanekoYamamotoRelMZV}, from which any other relation conjecturally follows, are then reproduced by the associativity of Kontsevich's $\star$-\/product\,?
\end{open}

\begin{open}
Which other sub-patterns, besides the ``eye-on-ground'' lemma, guarantee the vanishing of the integrand of the Kontsevich graph weight\,?
\end{open}

\section{The reduced affine star\/-\/product $\star_{\text{aff}}^{\text{red}}$ mod $\bar{o}(\hbar^7)$: why there is no Riemann zeta $\zeta(3)^2/\pi^6$ in the analytic formula}
\label{SecStarAffineReduced}
\noindent%
Apparent from Appendix~\ref{AppStarAffineFormula} with the analytic formula of $\star_{\text{aff}}\text{ mod }\bar{o}(\hbar^7)$ is the absence of $\zeta(3)^2/\pi^6$, which is nominally present in the graph expansion of Kontsevich's star\/-\/product with harmonic propagators (see Felder\/--\/Willwacher~\cite{FelderWillwacher}).

\begin{claim}
All the instances of $\zeta(3)^2/\pi^6$ disappear from $\star_{\textup{aff}}\textup{ mod }\bar{o}(\hbar^7)$ because the $\mathbb{Q}$-linear combinations of Kontsevich graphs near all $\zeta(3)^2/\pi^6$ --\,and moreover, many other $\mathbb{Q}$-linear combinations of Kontsevich graphs in $\star_{\textup{aff}}\textup{ mod }\bar{o}(\hbar^7)$\,-- assimilate into the $\mathbb{Q}$-linear combinations of \emph{Leibniz} graphs, so that the analytic formulae which these (or any) Leibniz graphs encode vanish identically due to the Jacobi identity for the affine Poisson bracket(s) at hand. 
\end{claim}

Let us explain how the much-reduced graph encoding $\star_{\text{aff}}^{\text{red}}\text{ mod }\bar{o}(\hbar^7)$ is obtained (see Appendix~\ref{AppStarAffineReducedEncoding}); we emphasize that for every given affine Poisson bracket, the analytic formula encoded by the graph expansion of $\star_{\text{aff}}^{\text{red}}\text{ mod }\bar{o}(\hbar^7)$ is exactly the same as 
the formula which was produced from the original graph expansion $\star_{\text{aff}}\text{ mod }\bar{o}(\hbar^7)$, see Appendix~\ref{AppStarAffineOriginalEncoding} for the affine star\/-\/product.

\begin{proof}[Proof scheme]
To eliminate as much as possible from the affine star\/-\/product we pose an \emph{unsolvable} problem: let us try to represent the entire graph encoding of the  $\star_{\text{aff}}\text{ mod }\bar{o}(\hbar^7)$ as a linear combination of Leibniz graphs.
(The impossible existence of a solution to this problem would mean that the affine star\/-\/product is identically zero up to order 7 for any affine Poisson bracket!)
Having generated all the relevant Leibniz graphs on 2 sinks and with up to $6$ aerial vertices (of which one vertex has out-degree three and the others (if any) two), and by taking a sum $L$ of these Leibniz graphs with undetermined $\mathbb{Q}[\zeta(3)^2/\pi^6]$-coefficients, we construct the linear algebraic system \[\star_{\text{aff}} - L = 0\text{ mod }\bar{o}(\hbar^7).\]
This linear system is now solved \emph{approximately} by using the method \texttt{solve\_right} from \textsf{SageMath}.\footnote{It can be expected that the method runs an iterative minimization of the quadratic potential $\langle \boldsymbol{x}, A\boldsymbol{x} \rangle$ obtained from a linear system $A\boldsymbol{x}=0$. The method stops at the potential's bottom $x_0$ of the values of Leibniz graph coefficients. The computation is run over the extension $\mathbb{Q}[\zeta(3)^2/\pi^6]$ of the field of rationals.}


The key idea is that by adding or subtracting any linear combination of Leibniz graphs, we do not alter the analytic formula which is encoded by the Kontsevich graph expansion of the star\/-\/product.
By definition, we put $\star_{\text{aff}}^{\text{red}} \mathrel{{:}{=}} \star_{\text{aff}} - L_0(x_0)$, where we subtract the found solution $L_0$ with the Leibniz graph coefficients $x_0$ in it.
We establish that the entire $\mathbb{Q}$-linear combination of Kontsevich's graphs near $\zeta(3)^2/\pi^6$ in the original graph expansion $\star_{\text{aff}}\text{ mod }\bar{o}(\hbar^7)$ is absorbed 
into the sum $L_0$ of Leibniz graphs.
Moreover, not only the $\zeta(3)^2/\pi^6$-part of the affine star\/-\/product but also much of its part with rational coefficients is absorbed further into the Leibniz graphs combination over $\mathbb{Q}[\zeta(3)^2/\pi^6]$.
In the reduced affine star\/-\/product $\star_{\text{aff}}^{\text{red}}\text{ mod }\bar{o}(\hbar^7)$, there remain only $326$ nonzero rational coefficients of Kontsevich graphs (at all orders, up to $\hbar^7$): this reduced graph encoding is listed in Appendix~\ref{AppStarAffineReducedEncoding}.
We remember 
that for any given affine Poisson bracket, the analytic formula of the affine star\/-\/product seventh order expansion produced from the originally found graph expansion $\star_{\text{aff}}\text{ mod }\bar{o}(\hbar^7)$ is now exactly the same as the analytic formula obtained 
from the reduced expansion $\star_{\text{aff}}^{\text{red}}\text{ mod }\bar{o}(\hbar^7)$; 
the formula is given in Appendix~\ref{AppStarAffineFormula}.
By a straightforward calculation, we verify that the difference of two analytic formulas amounts to a linear differential operator acting on the Jacobiator of the affine Poisson structure.
\end{proof}


\begin{rem}
\label{RemNoIrrationalContradiction}
For linear star\/-\/products Ben Amar~\cite
{BenAmar2003} predicted that all the coefficients of Kontsevich's graphs in the expansion are rational numbers.
On the other hand, the work of Felder\/--\/Willwacher~\cite{FelderWillwacher} suggests --\,already on the level of affine Poisson brackets\,-- and the work of Banks\/--\/Panzer\/--\/Pym~\cite{BPP} confirms 
that the Kontsevich graph coefficients in a linear star\/-\/product take values in the extensions of~$\mathbb{Q}$ by suitable multiple zeta values.
Our finding that affine star\/-\/product can be reduced by using Leibniz graphs brings together both statements: there is no contradiction between them. 
\end{rem}

\begin{open}Why does the entire $\zeta(3)^2/\pi^6$-part of $\star_{\text{aff}}\text{ mod }\bar{o}(\hbar^7)$ assimilate into expressions which vanish due to the Jacobi identity\,?
What is the mechanism of this effect within the proof of Formality theorem?
Likewise, what is the ``rational'' part of the (affine) star\/-\/product which does not contribute to the analytic formula in the same way\,?
Which multiple 
zeta values will not ever show up in the resulting analytic formula of the affine star\/-\/product\,?
\end{open}

\section{Conclusion}
\noindent%
We have studied the following questions:\\[0.5pt]
\mbox{ }$\bullet$\quad  
   What is, and what can be the formula of Kontsevich's star\/-\/product up to order~$7$ for affine --\,in particular, for linear on $\mathfrak{g}^*$\,-- Poisson brackets\,?\\
\mbox{ }$\bullet$\quad
   Where are the Riemann zeta values --- are they really present in the coefficients of Kontsevich's affine star\/-\/product\,?\\
\mbox{ }$\bullet$\quad  
   How does the associativity mechanism of Kontsevich's star\/-\/product work for generic or affine Poisson structures: namely, at order~$7$ of the expansion, is it the same or different from such mechanism at lower orders\,?

\smallskip
Now, this paper contains a reduced, 
ready-to-use formula of seventh order expansion for Kontsevich's star\/-\/product built for Poisson brackets with linear or affine coefficients (see Appendix~\ref{AppStarAffineFormula}). We certify in~\cite{factor23}
that this reduced affine star\/-\/product expansion $\star_{\text{aff}}^{\text{red}}\text{ mod }\bar{o}(\hbar^7)$ is associative up to $\bar{o}(\hbar^7)$.
For every affine Poisson bracket, our formula expresses Kontsevich's original star\/-\/product~\cite{MK97} with harmonic propagators, in its authentic gauge.

\begin{open}
Make the reduced affine star\/-\/product formula shorter (i.e.\ smaller in size) by using gauge transformations.
\end{open}

\section{Discussion}
\label{SecDiscussion}


\noindent%
By combining various 
methods to constrain the weights of Kontsevich's graphs in a star\/-\/product,
we make every such method work towards the verification of other methods.
This not only reduces a risk of our error, but also allows us to spot errors in the earlier published work.
For instance, we detect that the true value of the Kontsevich weight for the Felder--Willwacher graph \cite{FelderWillwacher} is $w(\Gamma) = \frac{13}{2903040} - \frac{1}{256}\zeta(3)^2/\pi^6$;
this graph shows up at $\hbar^7$ in the affine star\/-\/product, and an expression with a sign mismatch was reported in \cite{BPP}.
(The \textsf{kontsevint} software by Panzer now confirms the true value, which we obtain.) 

Likewise, we recall that for linear or affine Poisson brackets, the respective $\star$-product can be built iteratively, by starting with two explicit formulas for the weights of Bernoulli and loop graphs, 
using the method from~\cite
{BenAmar2003} by Ben Amar.

\begin{define}\label{DefBernoulliLoopGraph}
The \emph{Bernoulli graph} $\Gamma_n$ on two sinks $0,1$ and $n$ aerial vertices $2,\ldots,n+1$ is a Kontsevich directed graph built of wedges: $(2,0), (2,1)$ and $(k+3, k+2), (k+3, 1)$ for $0 \leqslant k < n-1$.\\[0.5pt]
\mbox{ }$\bullet$\quad
The \emph{loop graph} $\Gamma_n'$ on two sinks $0,1$ and $n > 1$ aerial vertices $2,\ldots,n+1$ is a Kontsevich directed graph built of wedges: $(2,0),(2,1+n)$ and $(k+3,k+2), (k+3,1)$ for $0\leqslant k < n-1$.
\end{define}

Knowing the Kontsevich weights of Bernoulli and loop graphs one can calculate the star\/-\/product $x^i \star g$ of a coordinate function $x^i$ and an arbitrary 
second argument~$g$.

\begin{lemma}[{\cite[Proposition 4.4.1]{Kathotia}}, {\cite[Corollary 6.3]{BenAmar2007}}]
The Kontsevich weight of Bernoulli graph is $w(\Gamma_n) = B_n / (n!)^2$ and the weight of loop graph is $w(\Gamma_n') = \tfrac{1}{2} B_n / (n!)^2$, where $B_n$ is the $n$th Bernoulli number and the convention in Ben Amar's work~\cite{BenAmar2007} must be $B_1 = +1/2$.
\end{lemma}

In Table~\ref{TabAffineWeight} we report the newly found values of the Bernoulli and loop graph weights (for $1 \leqslant n \leqslant 7$ and $2 \leqslant n \leqslant 7$ respectively).
\begin{table}[htb]
\caption{The weights of Bernoulli and loop graphs.}\label{TabAffineWeight}
\begin{center}	
\begin{tabular}{l | l | l | l | l | l | l | l}
$n$ & 1 & 2 & 3 & 4 & 5 & 6 & 7 \\ \hline
$w(\Gamma_n) = B_n / (n!)^2$ & $1/2$ & $1/24$ & $0$ & $-1/17280$ & $0$ & $1/21772800$ & $0$ \\
$w(\Gamma_n') = \tfrac{1}{2} B_n / (n!)^2$ & none & $1/48$ & $0$ & $-1/34560$ & $0$ & $1/43545600$ & $0$
\end{tabular}
\end{center}
\end{table}


All the produced values are identically equal to Ben Amar's 
theoretical prediction, provided that $B_1 = +1/2$ (which is the opposite of the current \textsf{SageMath} convention $-1/2$).


\subsection*{Acknowledgements}
The authors thank the anonymous referees for important comments and criticisms.
The research of R.B.\ was 
supported by project~$5020$ at the Institute of Mathematics,
Johannes Gutenberg\/--\/Uni\-ver\-si\-t\"at Mainz\footnote{This paper is extracted in part from the Ph.D.\ dissertation of R.~Buring (\textit{summa cum laude}, 13~October 2022, JGU Mainz).
RB thanks the dissertation committee for evaluating the thesis.
} 
and by CRC-326 grant GAUS ``Geometry and Arithmetic of Uniformized Structures''.
The travel of A.
K.\ was partially supported by project 135110 at the Bernoulli Institute, University of Groningen.
RB and AVK are grateful to the organizers of workshop ``Lie groupoids, Poisson geometry, and differentiable stacks'' in Banff, Canada, 5--10 June, 2022.
RB and AVK are grateful to the organizers of \textsc{Group34} colloquium on group theoretical methods in Physics on 18--22 July 2022 in Strasbourg for a warm atmosphere during the meeting.
RB thanks Erik Panzer for helpful discussions, for an invitation to the Isaac Newton Institute for Mathematical Sciences in Cambridge, and especially for the availability of \textsf{kontsevint} program 
for calculation of Kontsevich graph weights.
RB thanks Marco Gualtieri 
and Henrique Bursztyn for discussions in Banff. 
AVK thanks the organizers of international conference `Open Communications in Nonlinear Mathematical Physics' held during 23--29~June 2024 in Bad Ems, Germany.

\appendix

\newpage

\section{The Kontsevich graph encoding of the affine star\/-\/product $\star_{\text{aff}}\text{ mod }\bar{o}(\hbar^7)$}
\label{AppStarAffineEncoding}

\begin{implement}[Kontsevich graph encoding from {\cite[Implementation 1]{cpp}}]\label{DefEncoding}
The format to store a signed graph $s\Gamma$ with $\Gamma \in G_{m,n}$ is the integer number $m>0$, the integer $n \geqslant 0$, the sign $s$, followed by the (possibly empty, when $n=0$) list of $n$~ordered pairs of targets for edges issued from the internal vertices $m$,\ $\ldots$,\ $m+n-1$, respectively.
The full format is then ($m$,\ $n$, $s$; list of ordered pairs); in plain text we also write \verb"m n s   <list of ordered pairs>".
\end{implement}

\subsection{$\star_{\text{aff}}\text{ mod }\bar{o}(\hbar^7)$}
\label{AppStarAffineOriginalEncoding}

\begin{encoding}
\label{EncStarAffine7}
In the format described in Implementation~\ref{DefEncoding}, 
every graph encoding is followed by this graph's coefficient:

{\tiny

}
\end{encoding}

\twocolumn
\subsection{$\star_{\text{aff}}^{\text{red}}\text{ mod }\bar{o}(\hbar^7)$}
\label{AppStarAffineReducedEncoding}

\begin{encoding}
\label{EncStarAffine7Reduced}
In the format described in Implementation~\ref{DefEncoding},
now weights $\in\mathbb{Q}$:
{\tiny
\begin{verbatim}
h^0:
2 0 1       1
h^1:
2 1 1   0 1    1
h^2:
2 2 1   1 3 0 2    -1/6
2 2 1   1 3 0 1    -1/3
2 2 1   0 3 0 1    1/3
2 2 1   0 1 0 1    1/2
h^3:
2 3 1   1 3 0 2 0 1    -1/6
2 3 1   1 4 0 2 0 1    -1/3
2 3 1   1 4 0 1 0 1    -1/3
2 3 1   0 4 0 1 0 1    1/3
2 3 1   0 1 0 1 0 1    1/6
h^4:
2 4 1   1 5 1 4 0 3 0 2    1/72
2 4 1   1 3 1 4 0 5 0 2    -2/45
2 4 1   1 5 1 4 0 2 0 3    11/180
2 4 1   1 5 1 4 0 3 0 1    1/18
2 4 1   1 3 1 5 0 2 0 1    -2/45
2 4 1   1 3 1 4 0 5 0 1    2/15
2 4 1   1 4 0 5 0 2 0 1    -1/18
2 4 1   1 3 0 4 0 5 0 1    2/45
2 4 1   1 5 0 4 0 2 0 1    -2/15
2 4 1   1 3 0 2 0 1 0 1    -1/12
2 4 1   1 5 0 2 0 1 0 1    -1/3
2 4 1   1 5 0 4 0 1 0 1    -1/9
2 4 1   1 5 0 1 0 1 0 1    -1/6
2 4 1   0 5 0 1 0 1 0 1    1/6
2 4 1   0 1 0 1 0 1 0 1    1/24
2 4 1   1 5 1 4 0 1 0 1    1/18
2 4 1   0 5 0 4 0 1 0 1    1/18
2 4 1   1 4 1 5 1 3 0 1    1/45
2 4 1   0 4 0 5 0 3 0 1    -1/45
2 4 1   1 4 1 5 1 3 0 2    1/90
2 4 1   1 3 0 5 0 2 0 4    1/90
h^5:
2 5 1   1 4 1 2 0 5 0 3 0 1    1/60
2 5 1   3 4 1 6 0 5 0 1 0 1    2/15
2 5 1   1 6 1 2 0 5 0 3 0 1    2/45
2 5 1   1 3 1 4 0 5 0 6 0 1    2/45
2 5 1   1 5 1 4 0 3 0 6 0 1    1/18
2 5 1   3 4 1 6 0 2 0 1 0 1    11/360
2 5 1   1 5 1 4 0 3 0 2 0 1    1/72
2 5 1   1 3 1 6 0 2 0 1 0 1    -2/45
2 5 1   1 3 1 4 0 6 0 1 0 1    2/15
2 5 1   1 4 1 5 0 6 0 1 0 1    1/9
2 5 1   1 4 1 6 0 2 0 1 0 1    1/18
2 5 1   1 3 0 4 0 6 0 1 0 1    2/45
2 5 1   1 6 0 4 0 2 0 1 0 1    -2/15
2 5 1   1 5 0 6 0 2 0 1 0 1    -1/9
2 5 1   1 4 0 6 0 2 0 1 0 1    -1/18
2 5 1   1 6 0 2 0 1 0 1 0 1    -1/6
2 5 1   1 6 0 5 0 1 0 1 0 1    -1/9
2 5 1   1 3 0 2 0 1 0 1 0 1    -1/36
2 5 1   1 4 1 5 1 3 0 2 0 1    1/90
2 5 1   1 5 1 4 1 6 0 3 0 1    2/45
2 5 1   1 4 1 6 1 5 0 3 0 1    -1/45
2 5 1   1 3 0 5 0 2 0 4 0 1    1/90
2 5 1   1 4 0 5 0 6 0 2 0 1    -1/45
2 5 1   1 4 0 2 0 5 0 6 0 1    2/45
2 5 1   0 1 0 1 0 1 0 1 0 1    1/120
2 5 1   1 6 0 1 0 1 0 1 0 1    -1/18
2 5 1   0 6 0 1 0 1 0 1 0 1    1/18
2 5 1   1 6 1 5 0 1 0 1 0 1    1/18
2 5 1   0 6 0 5 0 1 0 1 0 1    1/18
2 5 1   1 4 1 6 1 3 0 1 0 1    1/45
2 5 1   0 4 0 6 0 3 0 1 0 1    -1/45
h^6:
2 6 1   1 4 1 7 1 3 0 6 0 2 0 5    -2/567
2 6 1   3 5 1 7 1 6 0 4 0 2 0 1    -31/3780
2 6 1   1 7 1 4 1 5 0 6 0 3 0 2    -1/360
2 6 1   3 5 1 2 1 6 0 4 0 7 0 1    -7/1620
2 6 1   4 5 1 6 1 2 0 7 0 3 0 1    -11/2160
2 6 1   1 7 1 6 1 5 0 4 0 3 0 2    -1/1296
2 6 1   1 4 1 6 1 7 0 3 0 2 0 1    4/315
2 6 1   4 6 1 7 1 5 0 3 0 1 0 1    -2/105
2 6 1   1 6 1 4 1 7 0 3 0 5 0 1    -2/315
2 6 1   1 4 1 7 1 5 0 6 0 2 0 1    -1/180
2 6 1   1 4 1 7 1 5 0 6 0 3 0 1    -4/105
2 6 1   1 4 1 6 1 3 0 7 0 2 0 1    1/270
2 6 1   1 4 1 6 1 3 0 7 0 5 0 1    2/135
2 6 1   4 6 1 5 1 7 0 3 0 1 0 1    -1/45
2 6 1   3 7 2 6 1 5 0 1 0 1 0 1    -11/1080
2 6 1   1 6 1 4 1 7 0 3 0 2 0 1    -2/135
2 6 1   1 7 1 6 1 5 0 4 0 3 0 1    -1/216
2 6 1   1 5 1 4 0 6 0 3 0 7 0 1    -4/315
2 6 1   5 6 1 4 0 7 0 3 0 1 0 1    -2/105
2 6 1   1 5 1 2 0 3 0 6 0 7 0 1    2/315
2 6 1   1 5 1 2 0 7 0 6 0 3 0 1    1/180
2 6 1   1 5 1 2 0 6 0 7 0 3 0 1    4/105
2 6 1   1 7 1 4 0 6 0 3 0 5 0 1    -1/270
2 6 1   1 7 1 2 0 6 0 3 0 5 0 1    -2/135
2 6 1   5 6 1 4 0 3 0 7 0 1 0 1    -1/45
2 6 1   3 7 2 6 0 5 0 1 0 1 0 1    11/1080
2 6 1   1 5 1 4 0 3 0 6 0 7 0 1    2/135
2 6 1   1 6 1 5 0 7 0 3 0 2 0 1    1/216
2 6 1   1 4 1 2 0 5 0 3 0 1 0 1    1/120
2 6 1   3 4 1 7 0 6 0 1 0 1 0 1    2/45
2 6 1   1 5 1 4 0 7 0 3 0 1 0 1    2/45
2 6 1   1 7 1 5 0 3 0 2 0 1 0 1    2/45
2 6 1   1 5 1 6 0 3 0 7 0 1 0 1    1/18
2 6 1   3 7 1 6 0 5 0 1 0 1 0 1    2/45
2 6 1   4 7 1 6 0 5 0 1 0 1 0 1    -2/45
2 6 1   1 6 1 2 0 7 0 3 0 1 0 1    4/135
2 6 1   1 7 1 4 0 5 0 6 0 1 0 1    4/135
2 6 1   3 7 2 6 0 1 0 1 0 1 0 1    11/720
2 6 1   1 7 1 4 0 3 0 2 0 1 0 1    1/18
2 6 1   1 6 1 5 0 7 0 3 0 1 0 1    1/54
2 6 1   1 5 1 4 0 3 0 2 0 1 0 1    1/144
2 6 1   1 5 1 7 1 3 1 4 0 2 0 6    -1/378
2 6 1   3 6 1 5 1 2 1 4 0 7 0 1    -31/11340
2 6 1   1 7 1 4 1 5 1 6 0 3 0 2    -1/540
2 6 1   3 5 2 4 1 6 1 7 0 1 0 1    1/567
2 6 1   1 4 1 2 0 7 0 3 0 5 0 6    -1/378
2 6 1   3 4 1 7 0 6 0 2 0 5 0 1    -31/11340
2 6 1   1 4 1 5 0 6 0 3 0 7 0 2    -1/540
2 6 1   3 5 2 4 0 6 0 7 0 1 0 1    1/567
2 6 1   1 7 0 6 0 1 0 1 0 1 0 1    -1/18
2 6 1   1 3 0 2 0 1 0 1 0 1 0 1    -1/144
2 6 1   1 3 0 7 0 1 0 1 0 1 0 1    -1/18
2 6 1   3 7 1 6 0 1 0 1 0 1 0 1    1/15
2 6 1   1 3 1 7 0 2 0 1 0 1 0 1    2/45
2 6 1   1 4 1 6 0 7 0 1 0 1 0 1    1/9
2 6 1   1 7 1 6 0 5 0 1 0 1 0 1    1/54
2 6 1   1 4 1 7 0 2 0 1 0 1 0 1    1/36
2 6 1   3 7 0 6 0 1 0 1 0 1 0 1    1/15
2 6 1   1 3 0 4 0 7 0 1 0 1 0 1    -2/45
2 6 1   1 7 0 2 0 6 0 1 0 1 0 1    -1/9
2 6 1   1 7 0 6 0 5 0 1 0 1 0 1    -1/54
2 6 1   1 3 0 2 0 7 0 1 0 1 0 1    -1/36
2 6 1   1 4 1 5 1 3 0 2 0 1 0 1    1/180
2 6 1   4 5 1 6 1 7 0 1 0 1 0 1    -2/45
2 6 1   1 7 1 4 1 6 0 3 0 1 0 1    -4/135
2 6 1   1 4 1 7 1 3 0 6 0 1 0 1    1/135
2 6 1   1 5 1 7 1 6 0 2 0 1 0 1    -1/108
2 6 1   1 5 1 4 1 7 0 3 0 1 0 1    2/45
2 6 1   1 4 1 7 1 5 0 3 0 1 0 1    -1/45
2 6 1   1 3 0 5 0 2 0 4 0 1 0 1    1/180
2 6 1   4 7 0 6 0 5 0 1 0 1 0 1    2/45
2 6 1   1 4 0 7 0 5 0 6 0 1 0 1    -4/135
2 6 1   1 6 0 5 0 7 0 4 0 1 0 1    1/135
2 6 1   1 5 0 7 0 6 0 2 0 1 0 1    -1/108
2 6 1   1 4 0 5 0 7 0 2 0 1 0 1    -1/45
2 6 1   1 4 0 2 0 5 0 7 0 1 0 1    2/45
2 6 1   1 7 1 4 1 5 1 6 0 3 0 1    -1/270
2 6 1   1 4 1 6 1 5 1 7 0 3 0 1    -1/270
2 6 1   3 6 1 5 1 7 1 4 0 1 0 1    -2/315
2 6 1   1 5 1 4 1 7 1 6 0 3 0 1    -4/315
2 6 1   1 4 0 7 0 5 0 6 0 2 0 1    1/270
2 6 1   1 4 0 5 0 2 0 6 0 7 0 1    1/270
2 6 1   3 6 0 5 0 7 0 4 0 1 0 1    -2/315
2 6 1   1 4 0 6 0 5 0 7 0 2 0 1    4/315
2 6 1   0 1 0 1 0 1 0 1 0 1 0 1    1/720
2 6 1   1 7 0 1 0 1 0 1 0 1 0 1    -1/72
2 6 1   0 7 0 1 0 1 0 1 0 1 0 1    1/72
2 6 1   1 7 1 6 0 1 0 1 0 1 0 1    1/36
2 6 1   0 7 0 6 0 1 0 1 0 1 0 1    1/36
2 6 1   1 4 1 7 1 3 0 1 0 1 0 1    1/90
2 6 1   1 7 1 6 1 5 0 1 0 1 0 1    -1/162
2 6 1   0 4 0 7 0 3 0 1 0 1 0 1    -1/90
2 6 1   0 7 0 6 0 5 0 1 0 1 0 1    1/162
2 6 1   1 6 1 7 1 5 1 3 1 4 0 1    -2/945
2 6 1   0 6 0 7 0 5 0 3 0 4 0 1    2/945
2 6 1   1 7 1 4 1 5 1 6 0 1 0 1    -1/135
2 6 1   0 7 0 4 0 5 0 6 0 1 0 1    -1/135
2 6 1   1 6 1 7 1 5 1 3 1 4 0 2    -1/945
2 6 1   1 3 0 7 0 2 0 6 0 4 0 5    -1/945

h^7:
2 7 1   1 6 1 5 1 3 0 7 0 8 0 4 0 1    -1/180
2 7 1   1 3 1 5 1 2 0 7 0 4 0 6 0 1    53/11340
2 7 1   1 5 1 2 1 6 0 7 0 3 0 4 0 1    -31/3780
2 7 1   1 3 1 5 1 6 0 4 0 7 0 8 0 1    2/135
2 7 1   1 3 1 6 1 2 0 7 0 8 0 4 0 1    -8/945
2 7 1   1 3 1 5 1 8 0 7 0 4 0 6 0 1    -29/945
2 7 1   3 6 1 4 1 5 0 7 0 8 0 1 0 1    -46/945
2 7 1   3 5 1 8 1 7 0 6 0 4 0 1 0 1    -32/945
2 7 1   3 6 1 7 1 5 0 2 0 8 0 1 0 1    8/315
2 7 1   3 6 1 4 1 5 0 8 0 2 0 1 0 1    -191/22680
2 7 1   3 6 1 2 1 7 0 4 0 8 0 1 0 1    -11/1080
2 7 1   1 6 1 2 1 5 0 4 0 7 0 3 0 1    -1/360
2 7 1   1 8 1 2 1 5 0 4 0 7 0 3 0 1    -1/135
2 7 1   1 3 1 5 1 6 0 7 0 4 0 8 0 1    -1/135
2 7 1   3 6 1 7 1 5 0 4 0 8 0 1 0 1    -1/45
2 7 1   3 8 2 7 1 5 0 4 0 1 0 1 0 1    -11/2160
2 7 1   1 8 1 7 1 6 0 2 0 4 0 3 0 1    -1/216
2 7 1   1 7 1 6 1 5 0 4 0 3 0 2 0 1    -1/1296
2 7 1   4 8 2 6 3 7 0 1 0 1 0 1 0 1    31/22680
2 7 1   3 5 1 4 1 7 0 6 0 8 0 1 0 1    -16/945
2 7 1   4 5 2 6 1 8 0 7 0 1 0 1 0 1    16/945
2 7 1   1 4 1 8 0 2 0 7 0 1 0 1 0 1    1/54
2 7 1   1 5 1 4 0 3 0 2 0 1 0 1 0 1    1/432
2 7 1   1 5 1 4 0 2 0 3 0 1 0 1 0 1    1/360
2 7 1   3 8 1 4 0 2 0 1 0 1 0 1 0 1    1/135
2 7 1   1 5 1 8 0 3 0 4 0 1 0 1 0 1    1/15
2 7 1   1 3 1 4 0 5 0 8 0 1 0 1 0 1    -1/45
2 7 1   1 5 1 8 0 3 0 7 0 1 0 1 0 1    1/18
2 7 1   1 7 1 8 0 5 0 3 0 1 0 1 0 1    4/135
2 7 1   4 8 1 7 0 6 0 1 0 1 0 1 0 1    -2/135
2 7 1   3 8 1 7 0 6 0 1 0 1 0 1 0 1    2/135
2 7 1   1 3 1 5 0 7 0 8 0 1 0 1 0 1    4/135
2 7 1   1 5 1 7 0 8 0 6 0 1 0 1 0 1    1/27
2 7 1   1 5 1 4 0 3 0 8 0 1 0 1 0 1    1/36
2 7 1   1 8 1 2 1 6 0 4 0 3 0 1 0 1    4/315
2 7 1   3 8 1 5 1 7 0 4 0 1 0 1 0 1    -2/105
2 7 1   1 3 1 5 1 8 0 6 0 4 0 1 0 1    -2/105
2 7 1   1 3 1 4 1 5 0 6 0 8 0 1 0 1    8/945
2 7 1   1 6 1 5 1 3 0 8 0 4 0 1 0 1    -2/105
2 7 1   4 8 1 5 1 6 0 7 0 1 0 1 0 1    -2/45
2 7 1   1 7 1 5 1 2 0 8 0 4 0 1 0 1    -4/135
2 7 1   1 3 1 6 1 2 0 8 0 4 0 1 0 1    1/270
2 7 1   1 8 1 5 1 3 0 6 0 4 0 1 0 1    -1/180
2 7 1   1 7 1 8 1 3 0 6 0 4 0 1 0 1    -1/135
2 7 1   4 5 1 7 1 8 0 6 0 1 0 1 0 1    -1/45
2 7 1   1 7 1 5 1 6 0 4 0 8 0 1 0 1    -1/45
2 7 1   4 8 1 5 1 7 0 3 0 1 0 1 0 1    -1/45
2 7 1   3 8 2 7 1 6 0 1 0 1 0 1 0 1    -11/1080
2 7 1   1 8 1 2 1 5 0 4 0 3 0 1 0 1    -2/135
2 7 1   1 7 1 5 1 6 0 8 0 4 0 1 0 1    -1/54
2 7 1   1 8 1 6 1 5 0 4 0 3 0 1 0 1    -1/216
2 7 1   1 7 1 6 1 2 0 8 0 4 0 1 0 1    2/135
2 7 1   1 3 1 6 1 7 0 8 0 4 0 1 0 1    -1/135
2 7 1   1 5 1 4 0 6 0 3 0 8 0 1 0 1    -4/315
2 7 1   4 8 1 5 0 3 0 7 0 1 0 1 0 1    -2/105
2 7 1   1 5 1 2 0 6 0 8 0 3 0 1 0 1    2/105
2 7 1   1 8 1 2 0 6 0 3 0 5 0 1 0 1    -8/945
2 7 1   1 8 1 5 0 3 0 6 0 2 0 1 0 1    2/105
2 7 1   5 8 1 7 0 3 0 6 0 1 0 1 0 1    -2/45
2 7 1   1 5 1 8 0 3 0 6 0 7 0 1 0 1    4/135
2 7 1   1 8 1 4 0 6 0 3 0 5 0 1 0 1    -1/270
2 7 1   1 5 1 2 0 8 0 6 0 3 0 1 0 1    1/180
2 7 1   1 3 1 5 0 8 0 6 0 7 0 1 0 1    1/135
2 7 1   3 5 1 8 0 7 0 6 0 1 0 1 0 1    1/45
2 7 1   1 7 1 6 0 3 0 8 0 2 0 1 0 1    1/45
2 7 1   5 8 1 4 0 3 0 7 0 1 0 1 0 1    -1/45
2 7 1   3 8 2 7 0 6 0 1 0 1 0 1 0 1    11/1080
2 7 1   1 5 1 4 0 3 0 6 0 8 0 1 0 1    2/135
2 7 1   1 7 1 4 0 3 0 8 0 2 0 1 0 1    1/54
2 7 1   1 6 1 5 0 8 0 3 0 2 0 1 0 1    1/216
2 7 1   1 7 1 5 0 6 0 8 0 3 0 1 0 1    1/135
2 7 1   1 8 1 5 0 3 0 6 0 7 0 1 0 1    -2/135
2 7 1   1 4 1 8 1 5 1 7 0 3 0 2 0 1    -1/270
2 7 1   1 5 1 7 1 3 1 4 0 2 0 6 0 1    11/7560
2 7 1   1 7 1 4 1 5 1 8 0 3 0 6 0 1    -2/315
2 7 1   1 5 1 7 1 3 1 4 0 2 0 8 0 1    -4/315
2 7 1   3 4 1 8 1 5 1 6 0 7 0 1 0 1    -4/315
2 7 1   1 5 1 4 1 7 1 6 0 3 0 2 0 1    -31/7560
2 7 1   1 5 1 4 1 8 1 6 0 7 0 3 0 1    8/945
2 7 1   1 5 1 7 1 3 1 4 0 8 0 6 0 1    -2/945
2 7 1   1 7 1 4 1 5 1 6 0 3 0 2 0 1    -1/540
2 7 1   4 6 1 5 1 2 1 8 0 3 0 1 0 1    -37/15120
2 7 1   1 7 1 5 1 6 1 8 0 4 0 3 0 1    -1/135
2 7 1   1 7 1 5 1 8 1 6 0 4 0 2 0 1    1/270
2 7 1   1 5 1 4 0 6 0 8 0 7 0 3 0 1    -1/270
2 7 1   1 4 1 2 0 7 0 3 0 5 0 6 0 1    11/7560
2 7 1   1 5 1 2 0 3 0 6 0 7 0 8 0 1    -2/315
2 7 1   1 4 1 8 0 7 0 3 0 5 0 6 0 1    -4/315
2 7 1   4 5 1 7 0 8 0 6 0 3 0 1 0 1    4/315
2 7 1   1 5 1 4 0 7 0 6 0 3 0 2 0 1    -31/7560
2 7 1   1 5 1 2 0 6 0 7 0 3 0 8 0 1    8/945
2 7 1   1 8 1 2 0 7 0 3 0 5 0 6 0 1    -2/945
2 7 1   1 4 1 5 0 6 0 3 0 7 0 2 0 1    -1/540
2 7 1   3 5 1 4 0 6 0 2 0 8 0 1 0 1    -37/15120
2 7 1   1 6 1 5 0 7 0 8 0 2 0 3 0 1    1/270
2 7 1   1 6 1 5 0 3 0 7 0 2 0 8 0 1    -1/135
2 7 1   1 8 0 2 0 1 0 1 0 1 0 1 0 1    -1/72
2 7 1   1 8 0 7 0 1 0 1 0 1 0 1 0 1    -1/54
2 7 1   1 3 0 2 0 1 0 1 0 1 0 1 0 1    -1/720
2 7 1   3 8 1 7 0 1 0 1 0 1 0 1 0 1    1/45
2 7 1   1 3 1 8 0 2 0 1 0 1 0 1 0 1    2/135
2 7 1   1 4 1 7 0 8 0 1 0 1 0 1 0 1    1/18
2 7 1   1 8 1 7 0 6 0 1 0 1 0 1 0 1    1/54
2 7 1   1 4 1 8 0 2 0 1 0 1 0 1 0 1    1/108
2 7 1   3 8 0 7 0 1 0 1 0 1 0 1 0 1    1/45
2 7 1   1 3 0 4 0 8 0 1 0 1 0 1 0 1    -2/135
2 7 1   1 8 0 2 0 7 0 1 0 1 0 1 0 1    -1/18
2 7 1   1 8 0 7 0 6 0 1 0 1 0 1 0 1    -1/54
2 7 1   1 3 0 2 0 8 0 1 0 1 0 1 0 1    -1/108
2 7 1   1 4 1 5 1 3 0 2 0 1 0 1 0 1    1/540
2 7 1   4 6 1 7 1 8 0 1 0 1 0 1 0 1    -2/45
2 7 1   1 4 1 7 1 8 0 2 0 1 0 1 0 1    -4/135
2 7 1   1 3 1 4 1 7 0 8 0 1 0 1 0 1    1/135
2 7 1   1 5 1 8 1 7 0 6 0 1 0 1 0 1    -1/54
2 7 1   1 5 1 8 1 7 0 2 0 1 0 1 0 1    -1/108
2 7 1   1 5 1 4 1 8 0 3 0 1 0 1 0 1    1/45
2 7 1   1 3 1 5 1 8 0 4 0 1 0 1 0 1    -1/90
2 7 1   1 3 0 5 0 2 0 4 0 1 0 1 0 1    1/540
2 7 1   4 8 0 7 0 6 0 1 0 1 0 1 0 1    2/45
2 7 1   1 4 0 8 0 5 0 7 0 1 0 1 0 1    -4/135
2 7 1   1 7 0 5 0 8 0 4 0 1 0 1 0 1    1/135
2 7 1   1 6 0 8 0 7 0 2 0 1 0 1 0 1    -1/54
2 7 1   1 5 0 8 0 7 0 2 0 1 0 1 0 1    -1/108
2 7 1   1 4 0 5 0 8 0 2 0 1 0 1 0 1    -1/90
2 7 1   1 4 0 2 0 5 0 8 0 1 0 1 0 1    1/45
2 7 1   1 5 1 8 1 3 1 4 0 2 0 1 0 1    2/315
2 7 1   1 8 1 4 1 5 1 6 0 3 0 1 0 1    -1/270
2 7 1   1 4 1 6 1 5 1 7 0 8 0 1 0 1    -1/135
2 7 1   1 4 1 5 1 6 1 8 0 3 0 1 0 1    -2/105
2 7 1   1 4 1 6 1 5 1 8 0 3 0 1 0 1    -1/270
2 7 1   1 8 1 6 1 5 1 7 0 4 0 1 0 1    -2/135
2 7 1   1 5 1 7 1 8 1 6 0 4 0 1 0 1    1/135
2 7 1   4 8 1 2 1 5 1 7 0 1 0 1 0 1    2/315
2 7 1   1 3 0 6 0 8 0 4 0 5 0 1 0 1    -2/315
2 7 1   1 4 0 8 0 5 0 6 0 2 0 1 0 1    1/270
2 7 1   1 8 0 5 0 2 0 6 0 7 0 1 0 1    1/135
2 7 1   1 4 0 6 0 5 0 8 0 2 0 1 0 1    2/105
2 7 1   1 4 0 5 0 2 0 6 0 8 0 1 0 1    1/270
2 7 1   1 5 0 6 0 7 0 8 0 2 0 1 0 1    -1/135
2 7 1   1 5 0 8 0 2 0 6 0 7 0 1 0 1    2/135
2 7 1   4 8 0 2 0 5 0 7 0 1 0 1 0 1    2/315
2 7 1   1 6 1 7 1 5 1 3 1 4 0 2 0 1    -1/945
2 7 1   1 7 1 6 1 8 1 4 1 5 0 3 0 1    -4/945
2 7 1   3 4 1 8 1 5 1 6 1 7 0 1 0 1    2/945
2 7 1   1 6 1 4 1 8 1 7 1 5 0 3 0 1    4/945
2 7 1   1 8 1 6 1 7 1 4 1 5 0 2 0 1    -2/945
2 7 1   1 3 0 7 0 2 0 6 0 4 0 5 0 1    -1/945
2 7 1   1 4 0 2 0 7 0 8 0 5 0 6 0 1    -4/945
2 7 1   3 4 0 8 0 5 0 6 0 7 0 1 0 1    -2/945
2 7 1   1 3 0 8 0 7 0 2 0 5 0 6 0 1    -2/945
2 7 1   1 4 0 7 0 5 0 8 0 2 0 6 0 1    4/945
2 7 1   0 1 0 1 0 1 0 1 0 1 0 1 0 1    1/5040
2 7 1   1 8 0 1 0 1 0 1 0 1 0 1 0 1    -1/360
2 7 1   0 8 0 1 0 1 0 1 0 1 0 1 0 1    1/360
2 7 1   1 8 1 7 0 1 0 1 0 1 0 1 0 1    1/108
2 7 1   0 8 0 7 0 1 0 1 0 1 0 1 0 1    1/108
2 7 1   1 3 1 4 1 8 0 1 0 1 0 1 0 1    1/270
2 7 1   1 8 1 7 1 6 0 1 0 1 0 1 0 1    -1/162
2 7 1   0 3 0 4 0 8 0 1 0 1 0 1 0 1    -1/270
2 7 1   0 8 0 7 0 6 0 1 0 1 0 1 0 1    1/162
2 7 1   1 8 1 4 1 5 1 7 0 1 0 1 0 1    -1/135
2 7 1   0 8 0 4 0 5 0 7 0 1 0 1 0 1    -1/135
2 7 1   1 6 1 8 1 5 1 3 1 4 0 1 0 1    -2/945
2 7 1   0 6 0 8 0 5 0 3 0 4 0 1 0 1    2/945
\end{verbatim}
}
\end{encoding}

\onecolumn

\section{The analytic formula of (reduced) affine star\/-\/product $\star_{\text{aff}}^{\text{red}}\text{ mod }\bar{o}(\hbar^7)$ with harmonic propagators for affine Poisson structures}
\label{AppStarAffineFormula}

\noindent%
The analytic formula of seventh order expansion of Kontsevich's $\star_{\text{aff}}$ for affine Poisson brackets reads as follows
(this expression is obtained from the reduced encoding \emph{without} Riemann $\zeta(3)^2/\pi^6$ in Appendix~\ref{AppStarAffineReducedEncoding}; for every affine Poisson bracket this formula with only rational coefficients is evaluated to the same expression as one obtains from the graph encoding \emph{with} Riemann $\zeta(3)^2/\pi^6$ in Appendix~\ref{AppStarAffineOriginalEncoding}: the difference of two expressions amounts to a linear differential operator acting on the Jacobiator $\tfrac{1}{2}\schouten{P,P}$ of the Poisson bi-vector $P$):

{\footnotesize
\begin{multline*}
f \star_{\text{aff}}^{\text{red}} g = 
f g 
+\hbar\cP^{ij} \partial_{i} f \partial_{j} g 
+\hbar^{2}\big(
-\tfrac{1}{6} \partial_{{\ell}} \cP^{ij} \partial_{j} \cP^{k{\ell}} \partial_{i} f \partial_{k} g 
-\tfrac{1}{3} \partial_{{\ell}} \cP^{ij} \cP^{k{\ell}} \partial_{i} f \partial_{k} \partial_{j} g 
+\tfrac{1}{3} \partial_{{\ell}} \cP^{ij} \cP^{k{\ell}} \partial_{k} \partial_{i} f \partial_{j} g \\
+\tfrac{1}{2} \cP^{ij} \cP^{k{\ell}} \partial_{k} \partial_{i} f \partial_{{\ell}} \partial_{j} g 
\big)
+\hbar^{3}\big(
-\tfrac{1}{6} \cP^{ij} \partial_{n} \cP^{k{\ell}} \partial_{{\ell}} \cP^{mn} \partial_{k} \partial_{i} f \partial_{m} \partial_{j} g 
-\tfrac{1}{3} \partial_{n} \cP^{ij} \cP^{k{\ell}} \partial_{{\ell}} \cP^{mn} \partial_{k} \partial_{i} f \partial_{m} \partial_{j} g \\
-\tfrac{1}{3} \partial_{n} \cP^{ij} \cP^{k{\ell}} \cP^{mn} \partial_{k} \partial_{i} f \partial_{m} \partial_{{\ell}} \partial_{j} g 
+\tfrac{1}{3} \partial_{n} \cP^{ij} \cP^{k{\ell}} \cP^{mn} \partial_{m} \partial_{k} \partial_{i} f \partial_{{\ell}} \partial_{j} g 
+\tfrac{1}{6} \cP^{ij} \cP^{k{\ell}} \cP^{mn} \partial_{m} \partial_{k} \partial_{i} f \partial_{n} \partial_{{\ell}} \partial_{j} g 
\big)\\
+\hbar^{4}\big(
\tfrac{1}{72} \partial_{{\ell}} \cP^{ij} \partial_{j} \cP^{k{\ell}} \partial_{q} \cP^{mn} \partial_{n} \cP^{pq} \partial_{m} \partial_{i} f \partial_{p} \partial_{k} g 
-\tfrac{2}{45} \partial_{q} \cP^{ij} \partial_{j} \cP^{k{\ell}} \partial_{{\ell}} \cP^{mn} \partial_{n} \cP^{pq} \partial_{k} \partial_{i} f \partial_{p} \partial_{m} g \\
+\tfrac{11}{180} \partial_{q} \cP^{ij} \partial_{j} \cP^{k{\ell}} \partial_{{\ell}} \cP^{mn} \partial_{n} \cP^{pq} \partial_{m} \partial_{i} f \partial_{p} \partial_{k} g 
+\tfrac{1}{18} \partial_{q} \cP^{ij} \partial_{n} \cP^{k{\ell}} \partial_{{\ell}} \cP^{mn} \cP^{pq} \partial_{k} \partial_{i} f \partial_{p} \partial_{m} \partial_{j} g \\
-\tfrac{2}{45} \partial_{q} \cP^{ij} \cP^{k{\ell}} \partial_{{\ell}} \cP^{mn} \partial_{n} \cP^{pq} \partial_{k} \partial_{i} f \partial_{p} \partial_{m} \partial_{j} g 
+\tfrac{2}{15} \partial_{{\ell}} \cP^{ij} \partial_{n} \cP^{k{\ell}} \partial_{q} \cP^{mn} \cP^{pq} \partial_{k} \partial_{i} f \partial_{p} \partial_{m} \partial_{j} g \\
-\tfrac{1}{18} \partial_{{\ell}} \cP^{ij} \cP^{k{\ell}} \partial_{q} \cP^{mn} \partial_{n} \cP^{pq} \partial_{m} \partial_{k} \partial_{i} f \partial_{p} \partial_{j} g 
+\tfrac{2}{45} \partial_{{\ell}} \cP^{ij} \partial_{n} \cP^{k{\ell}} \partial_{q} \cP^{mn} \cP^{pq} \partial_{m} \partial_{k} \partial_{i} f \partial_{p} \partial_{j} g \\
-\tfrac{2}{15} \partial_{q} \cP^{ij} \cP^{k{\ell}} \partial_{{\ell}} \cP^{mn} \partial_{n} \cP^{pq} \partial_{m} \partial_{k} \partial_{i} f \partial_{p} \partial_{j} g 
-\tfrac{1}{12} \cP^{ij} \cP^{k{\ell}} \partial_{q} \cP^{mn} \partial_{n} \cP^{pq} \partial_{m} \partial_{k} \partial_{i} f \partial_{p} \partial_{{\ell}} \partial_{j} g \\
-\tfrac{1}{3} \partial_{q} \cP^{ij} \cP^{k{\ell}} \cP^{mn} \partial_{n} \cP^{pq} \partial_{m} \partial_{k} \partial_{i} f \partial_{p} \partial_{{\ell}} \partial_{j} g 
-\tfrac{1}{9} \partial_{n} \cP^{ij} \partial_{q} \cP^{k{\ell}} \cP^{mn} \cP^{pq} \partial_{m} \partial_{k} \partial_{i} f \partial_{p} \partial_{{\ell}} \partial_{j} g \\
-\tfrac{1}{6} \partial_{q} \cP^{ij} \cP^{k{\ell}} \cP^{mn} \cP^{pq} \partial_{m} \partial_{k} \partial_{i} f \partial_{p} \partial_{n} \partial_{{\ell}} \partial_{j} g 
+\tfrac{1}{6} \partial_{q} \cP^{ij} \cP^{k{\ell}} \cP^{mn} \cP^{pq} \partial_{p} \partial_{m} \partial_{k} \partial_{i} f \partial_{n} \partial_{{\ell}} \partial_{j} g \\
+\tfrac{1}{24} \cP^{ij} \cP^{k{\ell}} \cP^{mn} \cP^{pq} \partial_{p} \partial_{m} \partial_{k} \partial_{i} f \partial_{q} \partial_{n} \partial_{{\ell}} \partial_{j} g 
+\tfrac{1}{18} \partial_{n} \cP^{ij} \partial_{q} \cP^{k{\ell}} \cP^{mn} \cP^{pq} \partial_{k} \partial_{i} f \partial_{p} \partial_{m} \partial_{{\ell}} \partial_{j} g \\
+\tfrac{1}{18} \partial_{n} \cP^{ij} \partial_{q} \cP^{k{\ell}} \cP^{mn} \cP^{pq} \partial_{p} \partial_{m} \partial_{k} \partial_{i} f \partial_{{\ell}} \partial_{j} g 
+\tfrac{1}{45} \partial_{{\ell}} \cP^{ij} \partial_{n} \cP^{k{\ell}} \partial_{q} \cP^{mn} \cP^{pq} \partial_{i} f \partial_{p} \partial_{m} \partial_{k} \partial_{j} g \\
-\tfrac{1}{45} \partial_{{\ell}} \cP^{ij} \partial_{n} \cP^{k{\ell}} \partial_{q} \cP^{mn} \cP^{pq} \partial_{p} \partial_{m} \partial_{k} \partial_{i} f \partial_{j} g 
+\tfrac{1}{90} \partial_{q} \cP^{ij} \partial_{j} \cP^{k{\ell}} \partial_{{\ell}} \cP^{mn} \partial_{n} \cP^{pq} \partial_{i} f \partial_{p} \partial_{m} \partial_{k} g \\
+\tfrac{1}{90} \partial_{q} \cP^{ij} \partial_{j} \cP^{k{\ell}} \partial_{{\ell}} \cP^{mn} \partial_{n} \cP^{pq} \partial_{m} \partial_{k} \partial_{i} f \partial_{p} g 
\big)
+\hbar^{5}\big(
\tfrac{1}{60} \cP^{ij} \partial_{s} \cP^{k{\ell}} \partial_{{\ell}} \cP^{mn} \partial_{n} \cP^{pq} \partial_{q} \cP^{rs} \partial_{m} \partial_{k} \partial_{i} f \partial_{r} \partial_{p} \partial_{j} g \\
-\tfrac{2}{15} \partial_{n} \cP^{ij} \partial_{q} \cP^{k{\ell}} \partial_{r} \cP^{mn} \partial_{s} \cP^{pq} \cP^{rs} \partial_{m} \partial_{k} \partial_{i} f \partial_{p} \partial_{{\ell}} \partial_{j} g 
+\tfrac{2}{45} \partial_{s} \cP^{ij} \cP^{k{\ell}} \partial_{{\ell}} \cP^{mn} \partial_{n} \cP^{pq} \partial_{q} \cP^{rs} \partial_{m} \partial_{k} \partial_{i} f \partial_{r} \partial_{p} \partial_{j} g \\
+\tfrac{2}{45} \partial_{{\ell}} \cP^{ij} \partial_{n} \cP^{k{\ell}} \partial_{q} \cP^{mn} \partial_{s} \cP^{pq} \cP^{rs} \partial_{m} \partial_{k} \partial_{i} f \partial_{r} \partial_{p} \partial_{j} g 
+\tfrac{1}{18} \partial_{{\ell}} \cP^{ij} \partial_{s} \cP^{k{\ell}} \partial_{q} \cP^{mn} \partial_{n} \cP^{pq} \cP^{rs} \partial_{m} \partial_{k} \partial_{i} f \partial_{r} \partial_{p} \partial_{j} g \\
-\tfrac{11}{360} \partial_{s} \cP^{ij} \cP^{k{\ell}} \partial_{p} \cP^{mn} \partial_{n} \cP^{pq} \partial_{q} \cP^{rs} \partial_{m} \partial_{k} \partial_{i} f \partial_{r} \partial_{{\ell}} \partial_{j} g 
+\tfrac{1}{72} \cP^{ij} \partial_{n} \cP^{k{\ell}} \partial_{{\ell}} \cP^{mn} \partial_{s} \cP^{pq} \partial_{q} \cP^{rs} \partial_{p} \partial_{k} \partial_{i} f \partial_{r} \partial_{m} \partial_{j} g \\
-\tfrac{2}{45} \partial_{s} \cP^{ij} \cP^{k{\ell}} \cP^{mn} \partial_{n} \cP^{pq} \partial_{q} \cP^{rs} \partial_{m} \partial_{k} \partial_{i} f \partial_{r} \partial_{p} \partial_{{\ell}} \partial_{j} g 
+\tfrac{2}{15} \partial_{n} \cP^{ij} \cP^{k{\ell}} \partial_{q} \cP^{mn} \partial_{s} \cP^{pq} \cP^{rs} \partial_{m} \partial_{k} \partial_{i} f \partial_{r} \partial_{p} \partial_{{\ell}} \partial_{j} g \\
+\tfrac{1}{9} \partial_{n} \cP^{ij} \partial_{q} \cP^{k{\ell}} \partial_{s} \cP^{mn} \cP^{pq} \cP^{rs} \partial_{m} \partial_{k} \partial_{i} f \partial_{r} \partial_{p} \partial_{{\ell}} \partial_{j} g 
+\tfrac{1}{18} \partial_{s} \cP^{ij} \cP^{k{\ell}} \partial_{q} \cP^{mn} \partial_{n} \cP^{pq} \cP^{rs} \partial_{m} \partial_{k} \partial_{i} f \partial_{r} \partial_{p} \partial_{{\ell}} \partial_{j} g \\
+\tfrac{2}{45} \partial_{n} \cP^{ij} \cP^{k{\ell}} \partial_{q} \cP^{mn} \partial_{s} \cP^{pq} \cP^{rs} \partial_{p} \partial_{m} \partial_{k} \partial_{i} f \partial_{r} \partial_{{\ell}} \partial_{j} g 
-\tfrac{2}{15} \partial_{s} \cP^{ij} \cP^{k{\ell}} \cP^{mn} \partial_{n} \cP^{pq} \partial_{q} \cP^{rs} \partial_{p} \partial_{m} \partial_{k} \partial_{i} f \partial_{r} \partial_{{\ell}} \partial_{j} g \\
-\tfrac{1}{9} \partial_{n} \cP^{ij} \partial_{s} \cP^{k{\ell}} \cP^{mn} \cP^{pq} \partial_{q} \cP^{rs} \partial_{p} \partial_{m} \partial_{k} \partial_{i} f \partial_{r} \partial_{{\ell}} \partial_{j} g 
-\tfrac{1}{18} \partial_{n} \cP^{ij} \cP^{k{\ell}} \cP^{mn} \partial_{s} \cP^{pq} \partial_{q} \cP^{rs} \partial_{p} \partial_{m} \partial_{k} \partial_{i} f \partial_{r} \partial_{{\ell}} \partial_{j} g \\
-\tfrac{1}{6} \partial_{s} \cP^{ij} \cP^{k{\ell}} \cP^{mn} \cP^{pq} \partial_{q} \cP^{rs} \partial_{p} \partial_{m} \partial_{k} \partial_{i} f \partial_{r} \partial_{n} \partial_{{\ell}} \partial_{j} g 
-\tfrac{1}{9} \partial_{q} \cP^{ij} \partial_{s} \cP^{k{\ell}} \cP^{mn} \cP^{pq} \cP^{rs} \partial_{p} \partial_{m} \partial_{k} \partial_{i} f \partial_{r} \partial_{n} \partial_{{\ell}} \partial_{j} g \\
-\tfrac{1}{36} \cP^{ij} \cP^{k{\ell}} \cP^{mn} \partial_{s} \cP^{pq} \partial_{q} \cP^{rs} \partial_{p} \partial_{m} \partial_{k} \partial_{i} f \partial_{r} \partial_{n} \partial_{{\ell}} \partial_{j} g 
+\tfrac{1}{90} \cP^{ij} \partial_{s} \cP^{k{\ell}} \partial_{{\ell}} \cP^{mn} \partial_{n} \cP^{pq} \partial_{q} \cP^{rs} \partial_{k} \partial_{i} f \partial_{r} \partial_{p} \partial_{m} \partial_{j} g \\
+\tfrac{2}{45} \partial_{q} \cP^{ij} \partial_{s} \cP^{k{\ell}} \partial_{{\ell}} \cP^{mn} \partial_{n} \cP^{pq} \cP^{rs} \partial_{k} \partial_{i} f \partial_{r} \partial_{p} \partial_{m} \partial_{j} g 
-\tfrac{1}{45} \partial_{n} \cP^{ij} \partial_{q} \cP^{k{\ell}} \partial_{{\ell}} \cP^{mn} \partial_{s} \cP^{pq} \cP^{rs} \partial_{k} \partial_{i} f \partial_{r} \partial_{p} \partial_{m} \partial_{j} g \\
+\tfrac{1}{90} \cP^{ij} \partial_{s} \cP^{k{\ell}} \partial_{{\ell}} \cP^{mn} \partial_{n} \cP^{pq} \partial_{q} \cP^{rs} \partial_{p} \partial_{m} \partial_{k} \partial_{i} f \partial_{r} \partial_{j} g 
-\tfrac{1}{45} \partial_{{\ell}} \cP^{ij} \partial_{s} \cP^{k{\ell}} \cP^{mn} \partial_{n} \cP^{pq} \partial_{q} \cP^{rs} \partial_{p} \partial_{m} \partial_{k} \partial_{i} f \partial_{r} \partial_{j} g \\
+\tfrac{2}{45} \partial_{{\ell}} \cP^{ij} \partial_{n} \cP^{k{\ell}} \partial_{s} \cP^{mn} \cP^{pq} \partial_{q} \cP^{rs} \partial_{p} \partial_{m} \partial_{k} \partial_{i} f \partial_{r} \partial_{j} g 
+\tfrac{1}{120} \cP^{ij} \cP^{k{\ell}} \cP^{mn} \cP^{pq} \cP^{rs} \partial_{r} \partial_{p} \partial_{m} \partial_{k} \partial_{i} f \partial_{s} \partial_{q} \partial_{n} \partial_{{\ell}} \partial_{j} g \\
-\tfrac{1}{18} \partial_{s} \cP^{ij} \cP^{k{\ell}} \cP^{mn} \cP^{pq} \cP^{rs} \partial_{p} \partial_{m} \partial_{k} \partial_{i} f \partial_{r} \partial_{q} \partial_{n} \partial_{{\ell}} \partial_{j} g 
+\tfrac{1}{18} \partial_{s} \cP^{ij} \cP^{k{\ell}} \cP^{mn} \cP^{pq} \cP^{rs} \partial_{r} \partial_{p} \partial_{m} \partial_{k} \partial_{i} f \partial_{q} \partial_{n} \partial_{{\ell}} \partial_{j} g \\
+\tfrac{1}{18} \partial_{q} \cP^{ij} \partial_{s} \cP^{k{\ell}} \cP^{mn} \cP^{pq} \cP^{rs} \partial_{m} \partial_{k} \partial_{i} f \partial_{r} \partial_{p} \partial_{n} \partial_{{\ell}} \partial_{j} g 
+\tfrac{1}{18} \partial_{q} \cP^{ij} \partial_{s} \cP^{k{\ell}} \cP^{mn} \cP^{pq} \cP^{rs} \partial_{r} \partial_{p} \partial_{m} \partial_{k} \partial_{i} f \partial_{n} \partial_{{\ell}} \partial_{j} g \\
+\tfrac{1}{45} \partial_{n} \cP^{ij} \cP^{k{\ell}} \partial_{q} \cP^{mn} \partial_{s} \cP^{pq} \cP^{rs} \partial_{k} \partial_{i} f \partial_{r} \partial_{p} \partial_{m} \partial_{{\ell}} \partial_{j} g 
-\tfrac{1}{45} \partial_{n} \cP^{ij} \cP^{k{\ell}} \partial_{q} \cP^{mn} \partial_{s} \cP^{pq} \cP^{rs} \partial_{r} \partial_{p} \partial_{m} \partial_{k} \partial_{i} f \partial_{{\ell}} \partial_{j} g
\big)
\end{multline*}
}
{\footnotesize
\begin{gather*}
\\
+\hbar^{6}\big(
-\tfrac{2}{567} \partial_{u} \cP^{ij} \partial_{j} \cP^{k{\ell}} \partial_{{\ell}} \cP^{mn} \partial_{n} \cP^{pq} \partial_{q} \cP^{rs} \partial_{s} \cP^{tu} \partial_{m} \partial_{k} \partial_{i} f \partial_{t} \partial_{r} \partial_{p} g \\
+\tfrac{31}{3780} \partial_{u} \cP^{ij} \partial_{r} \cP^{k{\ell}} \partial_{{\ell}} \cP^{mn} \partial_{n} \cP^{pq} \partial_{q} \cP^{rs} \partial_{s} \cP^{tu} \partial_{p} \partial_{k} \partial_{i} f \partial_{t} \partial_{m} \partial_{j} g \\
-\tfrac{1}{360} \partial_{q} \cP^{ij} \partial_{j} \cP^{k{\ell}} \partial_{{\ell}} \cP^{mn} \partial_{n} \cP^{pq} \partial_{u} \cP^{rs} \partial_{s} \cP^{tu} \partial_{r} \partial_{k} \partial_{i} f \partial_{t} \partial_{p} \partial_{m} g \\
+\tfrac{7}{1620} \partial_{{\ell}} \cP^{ij} \partial_{q} \cP^{k{\ell}} \partial_{t} \cP^{mn} \partial_{n} \cP^{pq} \partial_{u} \cP^{rs} \partial_{s} \cP^{tu} \partial_{m} \partial_{k} \partial_{i} f \partial_{r} \partial_{p} \partial_{j} g \\
+\tfrac{11}{2160} \partial_{{\ell}} \cP^{ij} \partial_{t} \cP^{k{\ell}} \partial_{q} \cP^{mn} \partial_{n} \cP^{pq} \partial_{u} \cP^{rs} \partial_{s} \cP^{tu} \partial_{m} \partial_{k} \partial_{i} f \partial_{r} \partial_{p} \partial_{j} g \\
-\tfrac{1}{1296} \partial_{{\ell}} \cP^{ij} \partial_{j} \cP^{k{\ell}} \partial_{q} \cP^{mn} \partial_{n} \cP^{pq} \partial_{u} \cP^{rs} \partial_{s} \cP^{tu} \partial_{r} \partial_{m} \partial_{i} f \partial_{t} \partial_{p} \partial_{k} g \\
+\tfrac{4}{315} \partial_{u} \cP^{ij} \cP^{k{\ell}} \partial_{{\ell}} \cP^{mn} \partial_{n} \cP^{pq} \partial_{q} \cP^{rs} \partial_{s} \cP^{tu} \partial_{p} \partial_{k} \partial_{i} f \partial_{t} \partial_{r} \partial_{m} \partial_{j} g \\
+\tfrac{2}{105} \partial_{q} \cP^{ij} \partial_{t} \cP^{k{\ell}} \partial_{s} \cP^{mn} \partial_{n} \cP^{pq} \partial_{u} \cP^{rs} \cP^{tu} \partial_{m} \partial_{k} \partial_{i} f \partial_{r} \partial_{p} \partial_{{\ell}} \partial_{j} g \\
-\tfrac{2}{315} \partial_{s} \cP^{ij} \partial_{u} \cP^{k{\ell}} \partial_{{\ell}} \cP^{mn} \partial_{n} \cP^{pq} \partial_{q} \cP^{rs} \cP^{tu} \partial_{m} \partial_{k} \partial_{i} f \partial_{t} \partial_{r} \partial_{p} \partial_{j} g \\
-\tfrac{1}{180} \partial_{u} \cP^{ij} \partial_{s} \cP^{k{\ell}} \partial_{{\ell}} \cP^{mn} \partial_{n} \cP^{pq} \partial_{q} \cP^{rs} \cP^{tu} \partial_{m} \partial_{k} \partial_{i} f \partial_{t} \partial_{r} \partial_{p} \partial_{j} g \\
-\tfrac{4}{105} \partial_{q} \cP^{ij} \partial_{s} \cP^{k{\ell}} \partial_{{\ell}} \cP^{mn} \partial_{n} \cP^{pq} \partial_{u} \cP^{rs} \cP^{tu} \partial_{m} \partial_{k} \partial_{i} f \partial_{t} \partial_{r} \partial_{p} \partial_{j} g \\
+\tfrac{1}{270} \partial_{{\ell}} \cP^{ij} \cP^{k{\ell}} \partial_{u} \cP^{mn} \partial_{n} \cP^{pq} \partial_{q} \cP^{rs} \partial_{s} \cP^{tu} \partial_{m} \partial_{k} \partial_{i} f \partial_{t} \partial_{r} \partial_{p} \partial_{j} g \\
+\tfrac{2}{135} \partial_{{\ell}} \cP^{ij} \partial_{n} \cP^{k{\ell}} \partial_{q} \cP^{mn} \partial_{s} \cP^{pq} \partial_{u} \cP^{rs} \cP^{tu} \partial_{m} \partial_{k} \partial_{i} f \partial_{t} \partial_{r} \partial_{p} \partial_{j} g \\
+\tfrac{1}{45} \partial_{s} \cP^{ij} \partial_{t} \cP^{k{\ell}} \partial_{q} \cP^{mn} \partial_{n} \cP^{pq} \partial_{u} \cP^{rs} \cP^{tu} \partial_{m} \partial_{k} \partial_{i} f \partial_{r} \partial_{p} \partial_{{\ell}} \partial_{j} g \\
-\tfrac{11}{1080} \partial_{q} \cP^{ij} \partial_{r} \cP^{k{\ell}} \partial_{t} \cP^{mn} \cP^{pq} \partial_{u} \cP^{rs} \partial_{s} \cP^{tu} \partial_{m} \partial_{k} \partial_{i} f \partial_{p} \partial_{n} \partial_{{\ell}} \partial_{j} g \\
-\tfrac{2}{135} \partial_{u} \cP^{ij} \partial_{n} \cP^{k{\ell}} \partial_{{\ell}} \cP^{mn} \cP^{pq} \partial_{q} \cP^{rs} \partial_{s} \cP^{tu} \partial_{p} \partial_{k} \partial_{i} f \partial_{t} \partial_{r} \partial_{m} \partial_{j} g \\
-\tfrac{1}{216} \partial_{u} \cP^{ij} \partial_{n} \cP^{k{\ell}} \partial_{{\ell}} \cP^{mn} \partial_{s} \cP^{pq} \partial_{q} \cP^{rs} \cP^{tu} \partial_{p} \partial_{k} \partial_{i} f \partial_{t} \partial_{r} \partial_{m} \partial_{j} g \\
-\tfrac{4}{315} \partial_{{\ell}} \cP^{ij} \partial_{n} \cP^{k{\ell}} \partial_{s} \cP^{mn} \partial_{u} \cP^{pq} \partial_{q} \cP^{rs} \cP^{tu} \partial_{p} \partial_{m} \partial_{k} \partial_{i} f \partial_{t} \partial_{r} \partial_{j} g \\
+\tfrac{2}{105} \partial_{n} \cP^{ij} \partial_{t} \cP^{k{\ell}} \partial_{s} \cP^{mn} \partial_{u} \cP^{pq} \partial_{q} \cP^{rs} \cP^{tu} \partial_{p} \partial_{m} \partial_{k} \partial_{i} f \partial_{r} \partial_{{\ell}} \partial_{j} g \\
+\tfrac{2}{315} \partial_{{\ell}} \cP^{ij} \partial_{n} \cP^{k{\ell}} \partial_{u} \cP^{mn} \cP^{pq} \partial_{q} \cP^{rs} \partial_{s} \cP^{tu} \partial_{p} \partial_{m} \partial_{k} \partial_{i} f \partial_{t} \partial_{r} \partial_{j} g \\
+\tfrac{1}{180} \partial_{{\ell}} \cP^{ij} \cP^{k{\ell}} \partial_{u} \cP^{mn} \partial_{n} \cP^{pq} \partial_{q} \cP^{rs} \partial_{s} \cP^{tu} \partial_{p} \partial_{m} \partial_{k} \partial_{i} f \partial_{t} \partial_{r} \partial_{j} g \\
+\tfrac{4}{105} \partial_{{\ell}} \cP^{ij} \partial_{u} \cP^{k{\ell}} \cP^{mn} \partial_{n} \cP^{pq} \partial_{q} \cP^{rs} \partial_{s} \cP^{tu} \partial_{p} \partial_{m} \partial_{k} \partial_{i} f \partial_{t} \partial_{r} \partial_{j} g \\
-\tfrac{1}{270} \partial_{u} \cP^{ij} \partial_{s} \cP^{k{\ell}} \partial_{{\ell}} \cP^{mn} \partial_{n} \cP^{pq} \partial_{q} \cP^{rs} \cP^{tu} \partial_{p} \partial_{m} \partial_{k} \partial_{i} f \partial_{t} \partial_{r} \partial_{j} g \\
-\tfrac{2}{135} \partial_{u} \cP^{ij} \cP^{k{\ell}} \partial_{{\ell}} \cP^{mn} \partial_{n} \cP^{pq} \partial_{q} \cP^{rs} \partial_{s} \cP^{tu} \partial_{p} \partial_{m} \partial_{k} \partial_{i} f \partial_{t} \partial_{r} \partial_{j} g \\
+\tfrac{1}{45} \partial_{n} \cP^{ij} \partial_{t} \cP^{k{\ell}} \partial_{u} \cP^{mn} \partial_{s} \cP^{pq} \partial_{q} \cP^{rs} \cP^{tu} \partial_{p} \partial_{m} \partial_{k} \partial_{i} f \partial_{r} \partial_{{\ell}} \partial_{j} g \\
+\tfrac{11}{1080} \partial_{q} \cP^{ij} \partial_{r} \cP^{k{\ell}} \partial_{t} \cP^{mn} \cP^{pq} \partial_{u} \cP^{rs} \partial_{s} \cP^{tu} \partial_{p} \partial_{m} \partial_{k} \partial_{i} f \partial_{n} \partial_{{\ell}} \partial_{j} g \\
+\tfrac{2}{135} \partial_{{\ell}} \cP^{ij} \partial_{n} \cP^{k{\ell}} \partial_{u} \cP^{mn} \partial_{s} \cP^{pq} \partial_{q} \cP^{rs} \cP^{tu} \partial_{p} \partial_{m} \partial_{k} \partial_{i} f \partial_{t} \partial_{r} \partial_{j} g \\
+\tfrac{1}{216} \partial_{{\ell}} \cP^{ij} \cP^{k{\ell}} \partial_{q} \cP^{mn} \partial_{n} \cP^{pq} \partial_{u} \cP^{rs} \partial_{s} \cP^{tu} \partial_{r} \partial_{m} \partial_{k} \partial_{i} f \partial_{t} \partial_{p} \partial_{j} g \\
+\tfrac{1}{120} \cP^{ij} \cP^{k{\ell}} \partial_{u} \cP^{mn} \partial_{n} \cP^{pq} \partial_{q} \cP^{rs} \partial_{s} \cP^{tu} \partial_{p} \partial_{m} \partial_{k} \partial_{i} f \partial_{t} \partial_{r} \partial_{{\ell}} \partial_{j} g \\
-\tfrac{2}{45} \partial_{q} \cP^{ij} \partial_{s} \cP^{k{\ell}} \cP^{mn} \partial_{t} \cP^{pq} \partial_{u} \cP^{rs} \cP^{tu} \partial_{p} \partial_{m} \partial_{k} \partial_{i} f \partial_{r} \partial_{n} \partial_{{\ell}} \partial_{j} g \\
+\tfrac{2}{45} \partial_{n} \cP^{ij} \cP^{k{\ell}} \partial_{s} \cP^{mn} \partial_{u} \cP^{pq} \partial_{q} \cP^{rs} \cP^{tu} \partial_{p} \partial_{m} \partial_{k} \partial_{i} f \partial_{t} \partial_{r} \partial_{{\ell}} \partial_{j} g \\
+\tfrac{2}{45} \partial_{q} \cP^{ij} \cP^{k{\ell}} \partial_{u} \cP^{mn} \partial_{n} \cP^{pq} \cP^{rs} \partial_{s} \cP^{tu} \partial_{r} \partial_{m} \partial_{k} \partial_{i} f \partial_{t} \partial_{p} \partial_{{\ell}} \partial_{j} g \\
+\tfrac{1}{18} \partial_{n} \cP^{ij} \partial_{s} \cP^{k{\ell}} \partial_{u} \cP^{mn} \cP^{pq} \partial_{q} \cP^{rs} \cP^{tu} \partial_{p} \partial_{m} \partial_{k} \partial_{i} f \partial_{t} \partial_{r} \partial_{{\ell}} \partial_{j} g \\
-\tfrac{2}{45} \partial_{q} \cP^{ij} \partial_{s} \cP^{k{\ell}} \partial_{t} \cP^{mn} \cP^{pq} \partial_{u} \cP^{rs} \cP^{tu} \partial_{p} \partial_{m} \partial_{k} \partial_{i} f \partial_{r} \partial_{n} \partial_{{\ell}} \partial_{j} g \\
+\tfrac{2}{45} \partial_{q} \cP^{ij} \partial_{s} \cP^{k{\ell}} \partial_{t} \cP^{mn} \partial_{u} \cP^{pq} \cP^{rs} \cP^{tu} \partial_{p} \partial_{m} \partial_{k} \partial_{i} f \partial_{r} \partial_{n} \partial_{{\ell}} \partial_{j} g \\
+\tfrac{4}{135} \partial_{n} \cP^{ij} \partial_{u} \cP^{k{\ell}} \cP^{mn} \cP^{pq} \partial_{q} \cP^{rs} \partial_{s} \cP^{tu} \partial_{p} \partial_{m} \partial_{k} \partial_{i} f \partial_{t} \partial_{r} \partial_{{\ell}} \partial_{j} g \\
+\tfrac{4}{135} \partial_{n} \cP^{ij} \partial_{s} \cP^{k{\ell}} \partial_{q} \cP^{mn} \partial_{u} \cP^{pq} \cP^{rs} \cP^{tu} \partial_{p} \partial_{m} \partial_{k} \partial_{i} f \partial_{t} \partial_{r} \partial_{{\ell}} \partial_{j} g \\
+\tfrac{11}{720} \partial_{r} \cP^{ij} \partial_{t} \cP^{k{\ell}} \cP^{mn} \cP^{pq} \partial_{u} \cP^{rs} \partial_{s} \cP^{tu} \partial_{p} \partial_{m} \partial_{k} \partial_{i} f \partial_{q} \partial_{n} \partial_{{\ell}} \partial_{j} g \\
\end{gather*}
}
{\footnotesize
\begin{gather*}
+\tfrac{1}{18} \partial_{q} \cP^{ij} \cP^{k{\ell}} \cP^{mn} \partial_{n} \cP^{pq} \partial_{u} \cP^{rs} \partial_{s} \cP^{tu} \partial_{r} \partial_{m} \partial_{k} \partial_{i} f \partial_{t} \partial_{p} \partial_{{\ell}} \partial_{j} g \\
+\tfrac{1}{54} \partial_{n} \cP^{ij} \partial_{u} \cP^{k{\ell}} \cP^{mn} \partial_{s} \cP^{pq} \partial_{q} \cP^{rs} \cP^{tu} \partial_{p} \partial_{m} \partial_{k} \partial_{i} f \partial_{t} \partial_{r} \partial_{{\ell}} \partial_{j} g \\
+\tfrac{1}{144} \cP^{ij} \cP^{k{\ell}} \partial_{q} \cP^{mn} \partial_{n} \cP^{pq} \partial_{u} \cP^{rs} \partial_{s} \cP^{tu} \partial_{r} \partial_{m} \partial_{k} \partial_{i} f \partial_{t} \partial_{p} \partial_{{\ell}} \partial_{j} g \\
-\tfrac{1}{378} \partial_{u} \cP^{ij} \partial_{j} \cP^{k{\ell}} \partial_{{\ell}} \cP^{mn} \partial_{n} \cP^{pq} \partial_{q} \cP^{rs} \partial_{s} \cP^{tu} \partial_{k} \partial_{i} f \partial_{t} \partial_{r} \partial_{p} \partial_{m} g \\
+\tfrac{31}{11340} \partial_{{\ell}} \cP^{ij} \partial_{t} \cP^{k{\ell}} \partial_{u} \cP^{mn} \partial_{n} \cP^{pq} \partial_{q} \cP^{rs} \partial_{s} \cP^{tu} \partial_{k} \partial_{i} f \partial_{r} \partial_{p} \partial_{m} \partial_{j} g \\
-\tfrac{1}{540} \partial_{{\ell}} \cP^{ij} \partial_{j} \cP^{k{\ell}} \partial_{u} \cP^{mn} \partial_{n} \cP^{pq} \partial_{q} \cP^{rs} \partial_{s} \cP^{tu} \partial_{m} \partial_{i} f \partial_{t} \partial_{r} \partial_{p} \partial_{k} g \\
+\tfrac{1}{567} \partial_{n} \cP^{ij} \partial_{q} \cP^{k{\ell}} \partial_{r} \cP^{mn} \partial_{t} \cP^{pq} \partial_{u} \cP^{rs} \partial_{s} \cP^{tu} \partial_{k} \partial_{i} f \partial_{p} \partial_{m} \partial_{{\ell}} \partial_{j} g \\
-\tfrac{1}{378} \partial_{u} \cP^{ij} \partial_{j} \cP^{k{\ell}} \partial_{{\ell}} \cP^{mn} \partial_{n} \cP^{pq} \partial_{q} \cP^{rs} \partial_{s} \cP^{tu} \partial_{p} \partial_{m} \partial_{k} \partial_{i} f \partial_{t} \partial_{r} g \\
+\tfrac{31}{11340} \partial_{u} \cP^{ij} \partial_{r} \cP^{k{\ell}} \partial_{{\ell}} \cP^{mn} \partial_{n} \cP^{pq} \partial_{q} \cP^{rs} \partial_{s} \cP^{tu} \partial_{p} \partial_{m} \partial_{k} \partial_{i} f \partial_{t} \partial_{j} g \\
-\tfrac{1}{540} \partial_{q} \cP^{ij} \partial_{j} \cP^{k{\ell}} \partial_{{\ell}} \cP^{mn} \partial_{n} \cP^{pq} \partial_{u} \cP^{rs} \partial_{s} \cP^{tu} \partial_{r} \partial_{m} \partial_{k} \partial_{i} f \partial_{t} \partial_{p} g \\
+\tfrac{1}{567} \partial_{n} \cP^{ij} \partial_{q} \cP^{k{\ell}} \partial_{r} \cP^{mn} \partial_{t} \cP^{pq} \partial_{u} \cP^{rs} \partial_{s} \cP^{tu} \partial_{p} \partial_{m} \partial_{k} \partial_{i} f \partial_{{\ell}} \partial_{j} g \\
-\tfrac{1}{18} \partial_{s} \cP^{ij} \partial_{u} \cP^{k{\ell}} \cP^{mn} \cP^{pq} \cP^{rs} \cP^{tu} \partial_{r} \partial_{p} \partial_{m} \partial_{k} \partial_{i} f \partial_{t} \partial_{q} \partial_{n} \partial_{{\ell}} \partial_{j} g \\
-\tfrac{1}{144} \cP^{ij} \cP^{k{\ell}} \cP^{mn} \cP^{pq} \partial_{u} \cP^{rs} \partial_{s} \cP^{tu} \partial_{r} \partial_{p} \partial_{m} \partial_{k} \partial_{i} f \partial_{t} \partial_{q} \partial_{n} \partial_{{\ell}} \partial_{j} g \\
-\tfrac{1}{18} \partial_{s} \cP^{ij} \cP^{k{\ell}} \cP^{mn} \cP^{pq} \partial_{u} \cP^{rs} \cP^{tu} \partial_{r} \partial_{p} \partial_{m} \partial_{k} \partial_{i} f \partial_{t} \partial_{q} \partial_{n} \partial_{{\ell}} \partial_{j} g \\
-\tfrac{1}{15} \partial_{s} \cP^{ij} \partial_{t} \cP^{k{\ell}} \cP^{mn} \cP^{pq} \partial_{u} \cP^{rs} \cP^{tu} \partial_{p} \partial_{m} \partial_{k} \partial_{i} f \partial_{r} \partial_{q} \partial_{n} \partial_{{\ell}} \partial_{j} g \\
+\tfrac{2}{45} \partial_{u} \cP^{ij} \cP^{k{\ell}} \cP^{mn} \cP^{pq} \partial_{q} \cP^{rs} \partial_{s} \cP^{tu} \partial_{p} \partial_{m} \partial_{k} \partial_{i} f \partial_{t} \partial_{r} \partial_{n} \partial_{{\ell}} \partial_{j} g \\
+\tfrac{1}{9} \partial_{q} \cP^{ij} \partial_{s} \cP^{k{\ell}} \cP^{mn} \partial_{u} \cP^{pq} \cP^{rs} \cP^{tu} \partial_{p} \partial_{m} \partial_{k} \partial_{i} f \partial_{t} \partial_{r} \partial_{n} \partial_{{\ell}} \partial_{j} g \\
+\tfrac{1}{54} \partial_{q} \cP^{ij} \partial_{s} \cP^{k{\ell}} \partial_{u} \cP^{mn} \cP^{pq} \cP^{rs} \cP^{tu} \partial_{p} \partial_{m} \partial_{k} \partial_{i} f \partial_{t} \partial_{r} \partial_{n} \partial_{{\ell}} \partial_{j} g \\
+\tfrac{1}{36} \partial_{u} \cP^{ij} \cP^{k{\ell}} \cP^{mn} \partial_{s} \cP^{pq} \partial_{q} \cP^{rs} \cP^{tu} \partial_{p} \partial_{m} \partial_{k} \partial_{i} f \partial_{t} \partial_{r} \partial_{n} \partial_{{\ell}} \partial_{j} g \\
-\tfrac{1}{15} \partial_{s} \cP^{ij} \partial_{t} \cP^{k{\ell}} \cP^{mn} \cP^{pq} \partial_{u} \cP^{rs} \cP^{tu} \partial_{r} \partial_{p} \partial_{m} \partial_{k} \partial_{i} f \partial_{q} \partial_{n} \partial_{{\ell}} \partial_{j} g \\
-\tfrac{2}{45} \partial_{q} \cP^{ij} \cP^{k{\ell}} \cP^{mn} \partial_{s} \cP^{pq} \partial_{u} \cP^{rs} \cP^{tu} \partial_{r} \partial_{p} \partial_{m} \partial_{k} \partial_{i} f \partial_{t} \partial_{n} \partial_{{\ell}} \partial_{j} g \\
-\tfrac{1}{9} \partial_{q} \cP^{ij} \partial_{u} \cP^{k{\ell}} \cP^{mn} \cP^{pq} \cP^{rs} \partial_{s} \cP^{tu} \partial_{r} \partial_{p} \partial_{m} \partial_{k} \partial_{i} f \partial_{t} \partial_{n} \partial_{{\ell}} \partial_{j} g \\
-\tfrac{1}{54} \partial_{q} \cP^{ij} \partial_{s} \cP^{k{\ell}} \partial_{u} \cP^{mn} \cP^{pq} \cP^{rs} \cP^{tu} \partial_{r} \partial_{p} \partial_{m} \partial_{k} \partial_{i} f \partial_{t} \partial_{n} \partial_{{\ell}} \partial_{j} g \\
-\tfrac{1}{36} \partial_{q} \cP^{ij} \cP^{k{\ell}} \cP^{mn} \cP^{pq} \partial_{u} \cP^{rs} \partial_{s} \cP^{tu} \partial_{r} \partial_{p} \partial_{m} \partial_{k} \partial_{i} f \partial_{t} \partial_{n} \partial_{{\ell}} \partial_{j} g \\
+\tfrac{1}{180} \cP^{ij} \cP^{k{\ell}} \partial_{u} \cP^{mn} \partial_{n} \cP^{pq} \partial_{q} \cP^{rs} \partial_{s} \cP^{tu} \partial_{m} \partial_{k} \partial_{i} f \partial_{t} \partial_{r} \partial_{p} \partial_{{\ell}} \partial_{j} g \\
+\tfrac{2}{45} \partial_{q} \cP^{ij} \partial_{s} \cP^{k{\ell}} \partial_{t} \cP^{mn} \partial_{u} \cP^{pq} \cP^{rs} \cP^{tu} \partial_{m} \partial_{k} \partial_{i} f \partial_{r} \partial_{p} \partial_{n} \partial_{{\ell}} \partial_{j} g \\
-\tfrac{4}{135} \partial_{s} \cP^{ij} \partial_{u} \cP^{k{\ell}} \cP^{mn} \partial_{n} \cP^{pq} \partial_{q} \cP^{rs} \cP^{tu} \partial_{m} \partial_{k} \partial_{i} f \partial_{t} \partial_{r} \partial_{p} \partial_{{\ell}} \partial_{j} g \\
+\tfrac{1}{135} \partial_{n} \cP^{ij} \partial_{q} \cP^{k{\ell}} \cP^{mn} \partial_{s} \cP^{pq} \partial_{u} \cP^{rs} \cP^{tu} \partial_{m} \partial_{k} \partial_{i} f \partial_{t} \partial_{r} \partial_{p} \partial_{{\ell}} \partial_{j} g \\
-\tfrac{1}{108} \partial_{s} \cP^{ij} \partial_{u} \cP^{k{\ell}} \partial_{q} \cP^{mn} \partial_{n} \cP^{pq} \cP^{rs} \cP^{tu} \partial_{m} \partial_{k} \partial_{i} f \partial_{t} \partial_{r} \partial_{p} \partial_{{\ell}} \partial_{j} g \\
+\tfrac{2}{45} \partial_{s} \cP^{ij} \cP^{k{\ell}} \partial_{u} \cP^{mn} \partial_{n} \cP^{pq} \partial_{q} \cP^{rs} \cP^{tu} \partial_{m} \partial_{k} \partial_{i} f \partial_{t} \partial_{r} \partial_{p} \partial_{{\ell}} \partial_{j} g \\
-\tfrac{1}{45} \partial_{q} \cP^{ij} \cP^{k{\ell}} \partial_{s} \cP^{mn} \partial_{n} \cP^{pq} \partial_{u} \cP^{rs} \cP^{tu} \partial_{m} \partial_{k} \partial_{i} f \partial_{t} \partial_{r} \partial_{p} \partial_{{\ell}} \partial_{j} g \\
+\tfrac{1}{180} \cP^{ij} \cP^{k{\ell}} \partial_{u} \cP^{mn} \partial_{n} \cP^{pq} \partial_{q} \cP^{rs} \partial_{s} \cP^{tu} \partial_{r} \partial_{p} \partial_{m} \partial_{k} \partial_{i} f \partial_{t} \partial_{{\ell}} \partial_{j} g \\
-\tfrac{2}{45} \partial_{q} \cP^{ij} \partial_{s} \cP^{k{\ell}} \partial_{t} \cP^{mn} \partial_{u} \cP^{pq} \cP^{rs} \cP^{tu} \partial_{r} \partial_{p} \partial_{m} \partial_{k} \partial_{i} f \partial_{n} \partial_{{\ell}} \partial_{j} g \\
-\tfrac{4}{135} \partial_{n} \cP^{ij} \partial_{q} \cP^{k{\ell}} \partial_{s} \cP^{mn} \cP^{pq} \partial_{u} \cP^{rs} \cP^{tu} \partial_{r} \partial_{p} \partial_{m} \partial_{k} \partial_{i} f \partial_{t} \partial_{{\ell}} \partial_{j} g \\
+\tfrac{1}{135} \partial_{n} \cP^{ij} \partial_{u} \cP^{k{\ell}} \partial_{q} \cP^{mn} \partial_{s} \cP^{pq} \cP^{rs} \cP^{tu} \partial_{r} \partial_{p} \partial_{m} \partial_{k} \partial_{i} f \partial_{t} \partial_{{\ell}} \partial_{j} g \\
-\tfrac{1}{108} \partial_{n} \cP^{ij} \partial_{q} \cP^{k{\ell}} \cP^{mn} \cP^{pq} \partial_{u} \cP^{rs} \partial_{s} \cP^{tu} \partial_{r} \partial_{p} \partial_{m} \partial_{k} \partial_{i} f \partial_{t} \partial_{{\ell}} \partial_{j} g \\
-\tfrac{1}{45} \partial_{n} \cP^{ij} \cP^{k{\ell}} \partial_{u} \cP^{mn} \cP^{pq} \partial_{q} \cP^{rs} \partial_{s} \cP^{tu} \partial_{r} \partial_{p} \partial_{m} \partial_{k} \partial_{i} f \partial_{t} \partial_{{\ell}} \partial_{j} g \\
+\tfrac{2}{45} \partial_{n} \cP^{ij} \cP^{k{\ell}} \partial_{q} \cP^{mn} \partial_{u} \cP^{pq} \cP^{rs} \partial_{s} \cP^{tu} \partial_{r} \partial_{p} \partial_{m} \partial_{k} \partial_{i} f \partial_{t} \partial_{{\ell}} \partial_{j} g \\
-\tfrac{1}{270} \partial_{u} \cP^{ij} \partial_{s} \cP^{k{\ell}} \partial_{{\ell}} \cP^{mn} \partial_{n} \cP^{pq} \partial_{q} \cP^{rs} \cP^{tu} \partial_{k} \partial_{i} f \partial_{t} \partial_{r} \partial_{p} \partial_{m} \partial_{j} g
\end{gather*}
}
{\footnotesize
\begin{gather*}
-\tfrac{1}{270} \partial_{q} \cP^{ij} \partial_{n} \cP^{k{\ell}} \partial_{{\ell}} \cP^{mn} \partial_{s} \cP^{pq} \partial_{u} \cP^{rs} \cP^{tu} \partial_{k} \partial_{i} f \partial_{t} \partial_{r} \partial_{p} \partial_{m} \partial_{j} g \\
+\tfrac{2}{315} \partial_{n} \cP^{ij} \partial_{t} \cP^{k{\ell}} \partial_{q} \cP^{mn} \partial_{s} \cP^{pq} \partial_{u} \cP^{rs} \cP^{tu} \partial_{k} \partial_{i} f \partial_{r} \partial_{p} \partial_{m} \partial_{{\ell}} \partial_{j} g \\
-\tfrac{4}{315} \partial_{q} \cP^{ij} \partial_{s} \cP^{k{\ell}} \partial_{{\ell}} \cP^{mn} \partial_{n} \cP^{pq} \partial_{u} \cP^{rs} \cP^{tu} \partial_{k} \partial_{i} f \partial_{t} \partial_{r} \partial_{p} \partial_{m} \partial_{j} g \\
+\tfrac{1}{270} \partial_{{\ell}} \cP^{ij} \cP^{k{\ell}} \partial_{u} \cP^{mn} \partial_{n} \cP^{pq} \partial_{q} \cP^{rs} \partial_{s} \cP^{tu} \partial_{r} \partial_{p} \partial_{m} \partial_{k} \partial_{i} f \partial_{t} \partial_{j} g \\
+\tfrac{1}{270} \partial_{{\ell}} \cP^{ij} \partial_{n} \cP^{k{\ell}} \partial_{q} \cP^{mn} \cP^{pq} \partial_{u} \cP^{rs} \partial_{s} \cP^{tu} \partial_{r} \partial_{p} \partial_{m} \partial_{k} \partial_{i} f \partial_{t} \partial_{j} g \\
+\tfrac{2}{315} \partial_{n} \cP^{ij} \partial_{t} \cP^{k{\ell}} \partial_{q} \cP^{mn} \partial_{s} \cP^{pq} \partial_{u} \cP^{rs} \cP^{tu} \partial_{r} \partial_{p} \partial_{m} \partial_{k} \partial_{i} f \partial_{{\ell}} \partial_{j} g \\
+\tfrac{4}{315} \partial_{{\ell}} \cP^{ij} \partial_{n} \cP^{k{\ell}} \partial_{u} \cP^{mn} \cP^{pq} \partial_{q} \cP^{rs} \partial_{s} \cP^{tu} \partial_{r} \partial_{p} \partial_{m} \partial_{k} \partial_{i} f \partial_{t} \partial_{j} g \\
+\tfrac{1}{720} \cP^{ij} \cP^{k{\ell}} \cP^{mn} \cP^{pq} \cP^{rs} \cP^{tu} \partial_{t} \partial_{r} \partial_{p} \partial_{m} \partial_{k} \partial_{i} f \partial_{u} \partial_{s} \partial_{q} \partial_{n} \partial_{{\ell}} \partial_{j} g \\
-\tfrac{1}{72} \partial_{u} \cP^{ij} \cP^{k{\ell}} \cP^{mn} \cP^{pq} \cP^{rs} \cP^{tu} \partial_{r} \partial_{p} \partial_{m} \partial_{k} \partial_{i} f \partial_{t} \partial_{s} \partial_{q} \partial_{n} \partial_{{\ell}} \partial_{j} g \\
+\tfrac{1}{72} \partial_{u} \cP^{ij} \cP^{k{\ell}} \cP^{mn} \cP^{pq} \cP^{rs} \cP^{tu} \partial_{t} \partial_{r} \partial_{p} \partial_{m} \partial_{k} \partial_{i} f \partial_{s} \partial_{q} \partial_{n} \partial_{{\ell}} \partial_{j} g \\
+\tfrac{1}{36} \partial_{s} \cP^{ij} \partial_{u} \cP^{k{\ell}} \cP^{mn} \cP^{pq} \cP^{rs} \cP^{tu} \partial_{p} \partial_{m} \partial_{k} \partial_{i} f \partial_{t} \partial_{r} \partial_{q} \partial_{n} \partial_{{\ell}} \partial_{j} g \\
+\tfrac{1}{36} \partial_{s} \cP^{ij} \partial_{u} \cP^{k{\ell}} \cP^{mn} \cP^{pq} \cP^{rs} \cP^{tu} \partial_{t} \partial_{r} \partial_{p} \partial_{m} \partial_{k} \partial_{i} f \partial_{q} \partial_{n} \partial_{{\ell}} \partial_{j} g \\
+\tfrac{1}{90} \partial_{q} \cP^{ij} \cP^{k{\ell}} \cP^{mn} \partial_{s} \cP^{pq} \partial_{u} \cP^{rs} \cP^{tu} \partial_{m} \partial_{k} \partial_{i} f \partial_{t} \partial_{r} \partial_{p} \partial_{n} \partial_{{\ell}} \partial_{j} g \\
-\tfrac{1}{162} \partial_{q} \cP^{ij} \partial_{s} \cP^{k{\ell}} \partial_{u} \cP^{mn} \cP^{pq} \cP^{rs} \cP^{tu} \partial_{m} \partial_{k} \partial_{i} f \partial_{t} \partial_{r} \partial_{p} \partial_{n} \partial_{{\ell}} \partial_{j} g \\
-\tfrac{1}{90} \partial_{q} \cP^{ij} \cP^{k{\ell}} \cP^{mn} \partial_{s} \cP^{pq} \partial_{u} \cP^{rs} \cP^{tu} \partial_{t} \partial_{r} \partial_{p} \partial_{m} \partial_{k} \partial_{i} f \partial_{n} \partial_{{\ell}} \partial_{j} g \\
+\tfrac{1}{162} \partial_{q} \cP^{ij} \partial_{s} \cP^{k{\ell}} \partial_{u} \cP^{mn} \cP^{pq} \cP^{rs} \cP^{tu} \partial_{t} \partial_{r} \partial_{p} \partial_{m} \partial_{k} \partial_{i} f \partial_{n} \partial_{{\ell}} \partial_{j} g \\
-\tfrac{2}{945} \partial_{{\ell}} \cP^{ij} \partial_{n} \cP^{k{\ell}} \partial_{q} \cP^{mn} \partial_{s} \cP^{pq} \partial_{u} \cP^{rs} \cP^{tu} \partial_{i} f \partial_{t} \partial_{r} \partial_{p} \partial_{m} \partial_{k} \partial_{j} g \\
+\tfrac{2}{945} \partial_{{\ell}} \cP^{ij} \partial_{n} \cP^{k{\ell}} \partial_{q} \cP^{mn} \partial_{s} \cP^{pq} \partial_{u} \cP^{rs} \cP^{tu} \partial_{t} \partial_{r} \partial_{p} \partial_{m} \partial_{k} \partial_{i} f \partial_{j} g \\
-\tfrac{1}{135} \partial_{n} \cP^{ij} \partial_{q} \cP^{k{\ell}} \partial_{s} \cP^{mn} \cP^{pq} \partial_{u} \cP^{rs} \cP^{tu} \partial_{k} \partial_{i} f \partial_{t} \partial_{r} \partial_{p} \partial_{m} \partial_{{\ell}} \partial_{j} g \\
-\tfrac{1}{135} \partial_{n} \cP^{ij} \partial_{q} \cP^{k{\ell}} \partial_{s} \cP^{mn} \cP^{pq} \partial_{u} \cP^{rs} \cP^{tu} \partial_{t} \partial_{r} \partial_{p} \partial_{m} \partial_{k} \partial_{i} f \partial_{{\ell}} \partial_{j} g \\
-\tfrac{1}{945} \partial_{u} \cP^{ij} \partial_{j} \cP^{k{\ell}} \partial_{{\ell}} \cP^{mn} \partial_{n} \cP^{pq} \partial_{q} \cP^{rs} \partial_{s} \cP^{tu} \partial_{i} f \partial_{t} \partial_{r} \partial_{p} \partial_{m} \partial_{k} g \\
-\tfrac{1}{945} \partial_{u} \cP^{ij} \partial_{j} \cP^{k{\ell}} \partial_{{\ell}} \cP^{mn} \partial_{n} \cP^{pq} \partial_{q} \cP^{rs} \partial_{s} \cP^{tu} \partial_{r} \partial_{p} \partial_{m} \partial_{k} \partial_{i} f \partial_{t} g
\big)\\
+\hbar^{7}\big(
-\tfrac{1}{180} \partial_{{\ell}} \cP^{ij} \partial_{w} \cP^{k{\ell}} \partial_{u} \cP^{mn} \partial_{n} \cP^{pq} \partial_{q} \cP^{rs} \partial_{s} \cP^{tu} \cP^{vw} \partial_{p} \partial_{m} \partial_{k} \partial_{i} f \partial_{v} \partial_{t} \partial_{r} \partial_{j} g \\
+\tfrac{53}{11340} \cP^{ij} \partial_{w} \cP^{k{\ell}} \partial_{{\ell}} \cP^{mn} \partial_{n} \cP^{pq} \partial_{q} \cP^{rs} \partial_{s} \cP^{tu} \partial_{u} \cP^{vw} \partial_{p} \partial_{m} \partial_{k} \partial_{i} f \partial_{v} \partial_{t} \partial_{r} \partial_{j} g \\
-\tfrac{31}{3780} \cP^{ij} \partial_{w} \cP^{k{\ell}} \partial_{{\ell}} \cP^{mn} \partial_{n} \cP^{pq} \partial_{q} \cP^{rs} \partial_{s} \cP^{tu} \partial_{u} \cP^{vw} \partial_{r} \partial_{m} \partial_{k} \partial_{i} f \partial_{v} \partial_{t} \partial_{p} \partial_{j} g \\
+\tfrac{2}{135} \partial_{{\ell}} \cP^{ij} \partial_{n} \cP^{k{\ell}} \partial_{s} \cP^{mn} \partial_{u} \cP^{pq} \partial_{q} \cP^{rs} \partial_{w} \cP^{tu} \cP^{vw} \partial_{p} \partial_{m} \partial_{k} \partial_{i} f \partial_{v} \partial_{t} \partial_{r} \partial_{j} g \\
-\tfrac{8}{945} \partial_{{\ell}} \cP^{ij} \partial_{w} \cP^{k{\ell}} \cP^{mn} \partial_{n} \cP^{pq} \partial_{q} \cP^{rs} \partial_{s} \cP^{tu} \partial_{u} \cP^{vw} \partial_{p} \partial_{m} \partial_{k} \partial_{i} f \partial_{v} \partial_{t} \partial_{r} \partial_{j} g \\
-\tfrac{29}{945} \partial_{s} \cP^{ij} \partial_{u} \cP^{k{\ell}} \partial_{{\ell}} \cP^{mn} \partial_{n} \cP^{pq} \partial_{q} \cP^{rs} \partial_{w} \cP^{tu} \cP^{vw} \partial_{p} \partial_{m} \partial_{k} \partial_{i} f \partial_{v} \partial_{t} \partial_{r} \partial_{j} g \\
+\tfrac{46}{945} \partial_{n} \cP^{ij} \partial_{q} \cP^{k{\ell}} \partial_{s} \cP^{mn} \partial_{v} \cP^{pq} \partial_{u} \cP^{rs} \partial_{w} \cP^{tu} \cP^{vw} \partial_{p} \partial_{m} \partial_{k} \partial_{i} f \partial_{t} \partial_{r} \partial_{{\ell}} \partial_{j} g \\
+\tfrac{32}{945} \partial_{s} \cP^{ij} \partial_{u} \cP^{k{\ell}} \partial_{v} \cP^{mn} \partial_{n} \cP^{pq} \partial_{q} \cP^{rs} \partial_{w} \cP^{tu} \cP^{vw} \partial_{p} \partial_{m} \partial_{k} \partial_{i} f \partial_{t} \partial_{r} \partial_{{\ell}} \partial_{j} g \\
-\tfrac{8}{315} \partial_{n} \cP^{ij} \partial_{u} \cP^{k{\ell}} \partial_{r} \cP^{mn} \partial_{w} \cP^{pq} \partial_{q} \cP^{rs} \partial_{s} \cP^{tu} \cP^{vw} \partial_{p} \partial_{m} \partial_{k} \partial_{i} f \partial_{v} \partial_{t} \partial_{{\ell}} \partial_{j} g \\
+\tfrac{191}{22680} \partial_{n} \cP^{ij} \cP^{k{\ell}} \partial_{w} \cP^{mn} \partial_{r} \cP^{pq} \partial_{q} \cP^{rs} \partial_{s} \cP^{tu} \partial_{u} \cP^{vw} \partial_{p} \partial_{m} \partial_{k} \partial_{i} f \partial_{v} \partial_{t} \partial_{{\ell}} \partial_{j} g \\
+\tfrac{11}{1080} \partial_{n} \cP^{ij} \partial_{s} \cP^{k{\ell}} \partial_{v} \cP^{mn} \cP^{pq} \partial_{q} \cP^{rs} \partial_{w} \cP^{tu} \partial_{u} \cP^{vw} \partial_{p} \partial_{m} \partial_{k} \partial_{i} f \partial_{t} \partial_{r} \partial_{{\ell}} \partial_{j} g \\
-\tfrac{1}{360} \cP^{ij} \partial_{s} \cP^{k{\ell}} \partial_{{\ell}} \cP^{mn} \partial_{n} \cP^{pq} \partial_{q} \cP^{rs} \partial_{w} \cP^{tu} \partial_{u} \cP^{vw} \partial_{t} \partial_{m} \partial_{k} \partial_{i} f \partial_{v} \partial_{r} \partial_{p} \partial_{j} g \\
-\tfrac{1}{135} \partial_{s} \cP^{ij} \cP^{k{\ell}} \partial_{{\ell}} \cP^{mn} \partial_{n} \cP^{pq} \partial_{q} \cP^{rs} \partial_{w} \cP^{tu} \partial_{u} \cP^{vw} \partial_{t} \partial_{m} \partial_{k} \partial_{i} f \partial_{v} \partial_{r} \partial_{p} \partial_{j} g \\
-\tfrac{1}{135} \partial_{{\ell}} \cP^{ij} \partial_{n} \cP^{k{\ell}} \partial_{u} \cP^{mn} \partial_{s} \cP^{pq} \partial_{q} \cP^{rs} \partial_{w} \cP^{tu} \cP^{vw} \partial_{p} \partial_{m} \partial_{k} \partial_{i} f \partial_{v} \partial_{t} \partial_{r} \partial_{j} g
\end{gather*}
}
{\footnotesize
\begin{gather*}
+\tfrac{1}{45} \partial_{n} \cP^{ij} \partial_{u} \cP^{k{\ell}} \partial_{v} \cP^{mn} \partial_{s} \cP^{pq} \partial_{q} \cP^{rs} \partial_{w} \cP^{tu} \cP^{vw} \partial_{p} \partial_{m} \partial_{k} \partial_{i} f \partial_{t} \partial_{r} \partial_{{\ell}} \partial_{j} g \\
-\tfrac{11}{2160} \partial_{t} \cP^{ij} \partial_{v} \cP^{k{\ell}} \cP^{mn} \partial_{s} \cP^{pq} \partial_{q} \cP^{rs} \partial_{w} \cP^{tu} \partial_{u} \cP^{vw} \partial_{p} \partial_{m} \partial_{k} \partial_{i} f \partial_{r} \partial_{n} \partial_{{\ell}} \partial_{j} g \\
-\tfrac{1}{216} \partial_{n} \cP^{ij} \cP^{k{\ell}} \partial_{{\ell}} \cP^{mn} \partial_{s} \cP^{pq} \partial_{q} \cP^{rs} \partial_{w} \cP^{tu} \partial_{u} \cP^{vw} \partial_{t} \partial_{p} \partial_{k} \partial_{i} f \partial_{v} \partial_{r} \partial_{m} \partial_{j} g \\
-\tfrac{1}{1296} \cP^{ij} \partial_{n} \cP^{k{\ell}} \partial_{{\ell}} \cP^{mn} \partial_{s} \cP^{pq} \partial_{q} \cP^{rs} \partial_{w} \cP^{tu} \partial_{u} \cP^{vw} \partial_{t} \partial_{p} \partial_{k} \partial_{i} f \partial_{v} \partial_{r} \partial_{m} \partial_{j} g \\
-\tfrac{31}{22680} \partial_{r} \cP^{ij} \partial_{t} \cP^{k{\ell}} \partial_{v} \cP^{mn} \cP^{pq} \partial_{w} \cP^{rs} \partial_{s} \cP^{tu} \partial_{u} \cP^{vw} \partial_{p} \partial_{m} \partial_{k} \partial_{i} f \partial_{q} \partial_{n} \partial_{{\ell}} \partial_{j} g \\
+\tfrac{16}{945} \partial_{n} \cP^{ij} \partial_{s} \cP^{k{\ell}} \partial_{q} \cP^{mn} \partial_{v} \cP^{pq} \partial_{u} \cP^{rs} \partial_{w} \cP^{tu} \cP^{vw} \partial_{p} \partial_{m} \partial_{k} \partial_{i} f \partial_{t} \partial_{r} \partial_{{\ell}} \partial_{j} g \\
+\tfrac{16}{945} \partial_{q} \cP^{ij} \partial_{s} \cP^{k{\ell}} \partial_{t} \cP^{mn} \partial_{v} \cP^{pq} \partial_{w} \cP^{rs} \cP^{tu} \partial_{u} \cP^{vw} \partial_{p} \partial_{m} \partial_{k} \partial_{i} f \partial_{r} \partial_{n} \partial_{{\ell}} \partial_{j} g \\
+\tfrac{1}{54} \partial_{q} \cP^{ij} \partial_{w} \cP^{k{\ell}} \cP^{mn} \cP^{pq} \partial_{u} \cP^{rs} \partial_{s} \cP^{tu} \cP^{vw} \partial_{r} \partial_{p} \partial_{m} \partial_{k} \partial_{i} f \partial_{v} \partial_{t} \partial_{n} \partial_{{\ell}} \partial_{j} g \\
+\tfrac{1}{432} \cP^{ij} \cP^{k{\ell}} \cP^{mn} \partial_{s} \cP^{pq} \partial_{q} \cP^{rs} \partial_{w} \cP^{tu} \partial_{u} \cP^{vw} \partial_{t} \partial_{p} \partial_{m} \partial_{k} \partial_{i} f \partial_{v} \partial_{r} \partial_{n} \partial_{{\ell}} \partial_{j} g \\
+\tfrac{1}{360} \cP^{ij} \cP^{k{\ell}} \cP^{mn} \partial_{w} \cP^{pq} \partial_{q} \cP^{rs} \partial_{s} \cP^{tu} \partial_{u} \cP^{vw} \partial_{t} \partial_{p} \partial_{m} \partial_{k} \partial_{i} f \partial_{v} \partial_{r} \partial_{n} \partial_{{\ell}} \partial_{j} g \\
-\tfrac{1}{135} \partial_{t} \cP^{ij} \cP^{k{\ell}} \cP^{mn} \cP^{pq} \partial_{w} \cP^{rs} \partial_{s} \cP^{tu} \partial_{u} \cP^{vw} \partial_{r} \partial_{p} \partial_{m} \partial_{k} \partial_{i} f \partial_{v} \partial_{q} \partial_{n} \partial_{{\ell}} \partial_{j} g \\
+\tfrac{1}{15} \partial_{u} \cP^{ij} \cP^{k{\ell}} \cP^{mn} \partial_{w} \cP^{pq} \partial_{q} \cP^{rs} \partial_{s} \cP^{tu} \cP^{vw} \partial_{r} \partial_{p} \partial_{m} \partial_{k} \partial_{i} f \partial_{v} \partial_{t} \partial_{n} \partial_{{\ell}} \partial_{j} g \\
-\tfrac{1}{45} \partial_{q} \cP^{ij} \cP^{k{\ell}} \cP^{mn} \partial_{s} \cP^{pq} \partial_{u} \cP^{rs} \partial_{w} \cP^{tu} \cP^{vw} \partial_{r} \partial_{p} \partial_{m} \partial_{k} \partial_{i} f \partial_{v} \partial_{t} \partial_{n} \partial_{{\ell}} \partial_{j} g \\
+\tfrac{1}{18} \partial_{q} \cP^{ij} \partial_{u} \cP^{k{\ell}} \cP^{mn} \partial_{w} \cP^{pq} \cP^{rs} \partial_{s} \cP^{tu} \cP^{vw} \partial_{r} \partial_{p} \partial_{m} \partial_{k} \partial_{i} f \partial_{v} \partial_{t} \partial_{n} \partial_{{\ell}} \partial_{j} g \\
+\tfrac{4}{135} \partial_{u} \cP^{ij} \partial_{w} \cP^{k{\ell}} \cP^{mn} \cP^{pq} \partial_{q} \cP^{rs} \partial_{s} \cP^{tu} \cP^{vw} \partial_{r} \partial_{p} \partial_{m} \partial_{k} \partial_{i} f \partial_{v} \partial_{t} \partial_{n} \partial_{{\ell}} \partial_{j} g \\
+\tfrac{2}{135} \partial_{s} \cP^{ij} \partial_{u} \cP^{k{\ell}} \partial_{v} \cP^{mn} \cP^{pq} \partial_{w} \cP^{rs} \cP^{tu} \cP^{vw} \partial_{r} \partial_{p} \partial_{m} \partial_{k} \partial_{i} f \partial_{t} \partial_{q} \partial_{n} \partial_{{\ell}} \partial_{j} g \\
-\tfrac{2}{135} \partial_{s} \cP^{ij} \partial_{u} \cP^{k{\ell}} \partial_{v} \cP^{mn} \cP^{pq} \cP^{rs} \partial_{w} \cP^{tu} \cP^{vw} \partial_{r} \partial_{p} \partial_{m} \partial_{k} \partial_{i} f \partial_{t} \partial_{q} \partial_{n} \partial_{{\ell}} \partial_{j} g \\
+\tfrac{4}{135} \partial_{q} \cP^{ij} \partial_{s} \cP^{k{\ell}} \cP^{mn} \partial_{u} \cP^{pq} \cP^{rs} \partial_{w} \cP^{tu} \cP^{vw} \partial_{r} \partial_{p} \partial_{m} \partial_{k} \partial_{i} f \partial_{v} \partial_{t} \partial_{n} \partial_{{\ell}} \partial_{j} g \\
+\tfrac{1}{27} \partial_{q} \cP^{ij} \partial_{s} \cP^{k{\ell}} \partial_{u} \cP^{mn} \partial_{w} \cP^{pq} \cP^{rs} \cP^{tu} \cP^{vw} \partial_{r} \partial_{p} \partial_{m} \partial_{k} \partial_{i} f \partial_{v} \partial_{t} \partial_{n} \partial_{{\ell}} \partial_{j} g \\
+\tfrac{1}{36} \partial_{q} \cP^{ij} \cP^{k{\ell}} \cP^{mn} \partial_{w} \cP^{pq} \partial_{u} \cP^{rs} \partial_{s} \cP^{tu} \cP^{vw} \partial_{r} \partial_{p} \partial_{m} \partial_{k} \partial_{i} f \partial_{v} \partial_{t} \partial_{n} \partial_{{\ell}} \partial_{j} g \\
+\tfrac{4}{315} \partial_{w} \cP^{ij} \cP^{k{\ell}} \cP^{mn} \partial_{n} \cP^{pq} \partial_{q} \cP^{rs} \partial_{s} \cP^{tu} \partial_{u} \cP^{vw} \partial_{r} \partial_{m} \partial_{k} \partial_{i} f \partial_{v} \partial_{t} \partial_{p} \partial_{{\ell}} \partial_{j} g \\
+\tfrac{2}{105} \partial_{s} \cP^{ij} \partial_{v} \cP^{k{\ell}} \cP^{mn} \partial_{u} \cP^{pq} \partial_{q} \cP^{rs} \partial_{w} \cP^{tu} \cP^{vw} \partial_{p} \partial_{m} \partial_{k} \partial_{i} f \partial_{t} \partial_{r} \partial_{n} \partial_{{\ell}} \partial_{j} g \\
-\tfrac{2}{105} \partial_{s} \cP^{ij} \cP^{k{\ell}} \partial_{u} \cP^{mn} \partial_{n} \cP^{pq} \partial_{q} \cP^{rs} \partial_{w} \cP^{tu} \cP^{vw} \partial_{p} \partial_{m} \partial_{k} \partial_{i} f \partial_{v} \partial_{t} \partial_{r} \partial_{{\ell}} \partial_{j} g \\
+\tfrac{8}{945} \partial_{n} \cP^{ij} \cP^{k{\ell}} \partial_{q} \cP^{mn} \partial_{s} \cP^{pq} \partial_{u} \cP^{rs} \partial_{w} \cP^{tu} \cP^{vw} \partial_{p} \partial_{m} \partial_{k} \partial_{i} f \partial_{v} \partial_{t} \partial_{r} \partial_{{\ell}} \partial_{j} g \\
-\tfrac{2}{105} \partial_{n} \cP^{ij} \cP^{k{\ell}} \partial_{u} \cP^{mn} \partial_{w} \cP^{pq} \partial_{q} \cP^{rs} \partial_{s} \cP^{tu} \cP^{vw} \partial_{p} \partial_{m} \partial_{k} \partial_{i} f \partial_{v} \partial_{t} \partial_{r} \partial_{{\ell}} \partial_{j} g \\
+\tfrac{2}{45} \partial_{q} \cP^{ij} \partial_{s} \cP^{k{\ell}} \partial_{v} \cP^{mn} \partial_{u} \cP^{pq} \partial_{w} \cP^{rs} \cP^{tu} \cP^{vw} \partial_{p} \partial_{m} \partial_{k} \partial_{i} f \partial_{t} \partial_{r} \partial_{n} \partial_{{\ell}} \partial_{j} g \\
-\tfrac{4}{135} \partial_{n} \cP^{ij} \partial_{u} \cP^{k{\ell}} \partial_{w} \cP^{mn} \cP^{pq} \partial_{q} \cP^{rs} \partial_{s} \cP^{tu} \cP^{vw} \partial_{p} \partial_{m} \partial_{k} \partial_{i} f \partial_{v} \partial_{t} \partial_{r} \partial_{{\ell}} \partial_{j} g \\
+\tfrac{1}{270} \partial_{n} \cP^{ij} \cP^{k{\ell}} \cP^{mn} \partial_{w} \cP^{pq} \partial_{q} \cP^{rs} \partial_{s} \cP^{tu} \partial_{u} \cP^{vw} \partial_{p} \partial_{m} \partial_{k} \partial_{i} f \partial_{v} \partial_{t} \partial_{r} \partial_{{\ell}} \partial_{j} g \\
-\tfrac{1}{180} \partial_{w} \cP^{ij} \cP^{k{\ell}} \partial_{u} \cP^{mn} \partial_{n} \cP^{pq} \partial_{q} \cP^{rs} \partial_{s} \cP^{tu} \cP^{vw} \partial_{p} \partial_{m} \partial_{k} \partial_{i} f \partial_{v} \partial_{t} \partial_{r} \partial_{{\ell}} \partial_{j} g \\
-\tfrac{1}{135} \partial_{u} \cP^{ij} \partial_{w} \cP^{k{\ell}} \cP^{mn} \partial_{n} \cP^{pq} \partial_{q} \cP^{rs} \partial_{s} \cP^{tu} \cP^{vw} \partial_{p} \partial_{m} \partial_{k} \partial_{i} f \partial_{v} \partial_{t} \partial_{r} \partial_{{\ell}} \partial_{j} g \\
+\tfrac{1}{45} \partial_{q} \cP^{ij} \partial_{s} \cP^{k{\ell}} \partial_{u} \cP^{mn} \partial_{v} \cP^{pq} \partial_{w} \cP^{rs} \cP^{tu} \cP^{vw} \partial_{p} \partial_{m} \partial_{k} \partial_{i} f \partial_{t} \partial_{r} \partial_{n} \partial_{{\ell}} \partial_{j} g \\
-\tfrac{1}{45} \partial_{n} \cP^{ij} \partial_{u} \cP^{k{\ell}} \partial_{s} \cP^{mn} \partial_{w} \cP^{pq} \partial_{q} \cP^{rs} \cP^{tu} \cP^{vw} \partial_{p} \partial_{m} \partial_{k} \partial_{i} f \partial_{v} \partial_{t} \partial_{r} \partial_{{\ell}} \partial_{j} g \\
+\tfrac{1}{45} \partial_{u} \cP^{ij} \partial_{v} \cP^{k{\ell}} \cP^{mn} \partial_{s} \cP^{pq} \partial_{q} \cP^{rs} \partial_{w} \cP^{tu} \cP^{vw} \partial_{p} \partial_{m} \partial_{k} \partial_{i} f \partial_{t} \partial_{r} \partial_{n} \partial_{{\ell}} \partial_{j} g \\
-\tfrac{11}{1080} \partial_{s} \cP^{ij} \partial_{t} \cP^{k{\ell}} \partial_{v} \cP^{mn} \cP^{pq} \cP^{rs} \partial_{w} \cP^{tu} \partial_{u} \cP^{vw} \partial_{p} \partial_{m} \partial_{k} \partial_{i} f \partial_{r} \partial_{q} \partial_{n} \partial_{{\ell}} \partial_{j} g \\
-\tfrac{2}{135} \partial_{w} \cP^{ij} \cP^{k{\ell}} \partial_{q} \cP^{mn} \partial_{n} \cP^{pq} \cP^{rs} \partial_{s} \cP^{tu} \partial_{u} \cP^{vw} \partial_{r} \partial_{m} \partial_{k} \partial_{i} f \partial_{v} \partial_{t} \partial_{p} \partial_{{\ell}} \partial_{j} g \\
-\tfrac{1}{54} \partial_{n} \cP^{ij} \partial_{u} \cP^{k{\ell}} \partial_{w} \cP^{mn} \partial_{s} \cP^{pq} \partial_{q} \cP^{rs} \cP^{tu} \cP^{vw} \partial_{p} \partial_{m} \partial_{k} \partial_{i} f \partial_{v} \partial_{t} \partial_{r} \partial_{{\ell}} \partial_{j} g \\
-\tfrac{1}{216} \partial_{w} \cP^{ij} \cP^{k{\ell}} \partial_{q} \cP^{mn} \partial_{n} \cP^{pq} \partial_{u} \cP^{rs} \partial_{s} \cP^{tu} \cP^{vw} \partial_{r} \partial_{m} \partial_{k} \partial_{i} f \partial_{v} \partial_{t} \partial_{p} \partial_{{\ell}} \partial_{j} g
\end{gather*}
}
{\footnotesize
\begin{gather*}
+\tfrac{2}{135} \partial_{n} \cP^{ij} \partial_{u} \cP^{k{\ell}} \cP^{mn} \partial_{w} \cP^{pq} \partial_{q} \cP^{rs} \partial_{s} \cP^{tu} \cP^{vw} \partial_{p} \partial_{m} \partial_{k} \partial_{i} f \partial_{v} \partial_{t} \partial_{r} \partial_{{\ell}} \partial_{j} g \\
-\tfrac{1}{135} \partial_{n} \cP^{ij} \partial_{s} \cP^{k{\ell}} \cP^{mn} \partial_{u} \cP^{pq} \partial_{q} \cP^{rs} \partial_{w} \cP^{tu} \cP^{vw} \partial_{p} \partial_{m} \partial_{k} \partial_{i} f \partial_{v} \partial_{t} \partial_{r} \partial_{{\ell}} \partial_{j} g \\
-\tfrac{4}{315} \partial_{n} \cP^{ij} \cP^{k{\ell}} \partial_{q} \cP^{mn} \partial_{u} \cP^{pq} \partial_{w} \cP^{rs} \partial_{s} \cP^{tu} \cP^{vw} \partial_{r} \partial_{p} \partial_{m} \partial_{k} \partial_{i} f \partial_{v} \partial_{t} \partial_{{\ell}} \partial_{j} g \\
+\tfrac{2}{105} \partial_{q} \cP^{ij} \partial_{v} \cP^{k{\ell}} \cP^{mn} \partial_{u} \cP^{pq} \partial_{w} \cP^{rs} \partial_{s} \cP^{tu} \cP^{vw} \partial_{r} \partial_{p} \partial_{m} \partial_{k} \partial_{i} f \partial_{t} \partial_{n} \partial_{{\ell}} \partial_{j} g \\
+\tfrac{2}{105} \partial_{n} \cP^{ij} \cP^{k{\ell}} \partial_{w} \cP^{mn} \cP^{pq} \partial_{q} \cP^{rs} \partial_{s} \cP^{tu} \partial_{u} \cP^{vw} \partial_{r} \partial_{p} \partial_{m} \partial_{k} \partial_{i} f \partial_{v} \partial_{t} \partial_{{\ell}} \partial_{j} g \\
-\tfrac{8}{945} \partial_{w} \cP^{ij} \cP^{k{\ell}} \cP^{mn} \partial_{n} \cP^{pq} \partial_{q} \cP^{rs} \partial_{s} \cP^{tu} \partial_{u} \cP^{vw} \partial_{r} \partial_{p} \partial_{m} \partial_{k} \partial_{i} f \partial_{v} \partial_{t} \partial_{{\ell}} \partial_{j} g \\
+\tfrac{2}{105} \partial_{s} \cP^{ij} \cP^{k{\ell}} \partial_{w} \cP^{mn} \partial_{n} \cP^{pq} \partial_{q} \cP^{rs} \cP^{tu} \partial_{u} \cP^{vw} \partial_{t} \partial_{p} \partial_{m} \partial_{k} \partial_{i} f \partial_{v} \partial_{r} \partial_{{\ell}} \partial_{j} g \\
+\tfrac{2}{45} \partial_{q} \cP^{ij} \partial_{u} \cP^{k{\ell}} \partial_{v} \cP^{mn} \partial_{w} \cP^{pq} \cP^{rs} \partial_{s} \cP^{tu} \cP^{vw} \partial_{r} \partial_{p} \partial_{m} \partial_{k} \partial_{i} f \partial_{t} \partial_{n} \partial_{{\ell}} \partial_{j} g \\
+\tfrac{4}{135} \partial_{n} \cP^{ij} \partial_{u} \cP^{k{\ell}} \partial_{q} \cP^{mn} \partial_{w} \cP^{pq} \cP^{rs} \partial_{s} \cP^{tu} \cP^{vw} \partial_{r} \partial_{p} \partial_{m} \partial_{k} \partial_{i} f \partial_{v} \partial_{t} \partial_{{\ell}} \partial_{j} g \\
-\tfrac{1}{270} \partial_{w} \cP^{ij} \cP^{k{\ell}} \partial_{u} \cP^{mn} \partial_{n} \cP^{pq} \partial_{q} \cP^{rs} \partial_{s} \cP^{tu} \cP^{vw} \partial_{r} \partial_{p} \partial_{m} \partial_{k} \partial_{i} f \partial_{v} \partial_{t} \partial_{{\ell}} \partial_{j} g \\
+\tfrac{1}{180} \partial_{n} \cP^{ij} \cP^{k{\ell}} \cP^{mn} \partial_{w} \cP^{pq} \partial_{q} \cP^{rs} \partial_{s} \cP^{tu} \partial_{u} \cP^{vw} \partial_{r} \partial_{p} \partial_{m} \partial_{k} \partial_{i} f \partial_{v} \partial_{t} \partial_{{\ell}} \partial_{j} g \\
+\tfrac{1}{135} \partial_{n} \cP^{ij} \partial_{q} \cP^{k{\ell}} \partial_{s} \cP^{mn} \cP^{pq} \partial_{u} \cP^{rs} \partial_{w} \cP^{tu} \cP^{vw} \partial_{r} \partial_{p} \partial_{m} \partial_{k} \partial_{i} f \partial_{v} \partial_{t} \partial_{{\ell}} \partial_{j} g \\
-\tfrac{1}{45} \partial_{q} \cP^{ij} \partial_{s} \cP^{k{\ell}} \partial_{u} \cP^{mn} \partial_{v} \cP^{pq} \cP^{rs} \partial_{w} \cP^{tu} \cP^{vw} \partial_{r} \partial_{p} \partial_{m} \partial_{k} \partial_{i} f \partial_{t} \partial_{n} \partial_{{\ell}} \partial_{j} g \\
+\tfrac{1}{45} \partial_{n} \cP^{ij} \partial_{s} \cP^{k{\ell}} \cP^{mn} \partial_{w} \cP^{pq} \partial_{q} \cP^{rs} \cP^{tu} \partial_{u} \cP^{vw} \partial_{t} \partial_{p} \partial_{m} \partial_{k} \partial_{i} f \partial_{v} \partial_{r} \partial_{{\ell}} \partial_{j} g \\
+\tfrac{1}{45} \partial_{q} \cP^{ij} \partial_{v} \cP^{k{\ell}} \cP^{mn} \partial_{w} \cP^{pq} \partial_{u} \cP^{rs} \partial_{s} \cP^{tu} \cP^{vw} \partial_{r} \partial_{p} \partial_{m} \partial_{k} \partial_{i} f \partial_{t} \partial_{n} \partial_{{\ell}} \partial_{j} g \\
+\tfrac{11}{1080} \partial_{s} \cP^{ij} \partial_{t} \cP^{k{\ell}} \partial_{v} \cP^{mn} \cP^{pq} \cP^{rs} \partial_{w} \cP^{tu} \partial_{u} \cP^{vw} \partial_{r} \partial_{p} \partial_{m} \partial_{k} \partial_{i} f \partial_{q} \partial_{n} \partial_{{\ell}} \partial_{j} g \\
+\tfrac{2}{135} \partial_{n} \cP^{ij} \cP^{k{\ell}} \partial_{q} \cP^{mn} \partial_{w} \cP^{pq} \partial_{u} \cP^{rs} \partial_{s} \cP^{tu} \cP^{vw} \partial_{r} \partial_{p} \partial_{m} \partial_{k} \partial_{i} f \partial_{v} \partial_{t} \partial_{{\ell}} \partial_{j} g \\
+\tfrac{1}{54} \partial_{n} \cP^{ij} \partial_{s} \cP^{k{\ell}} \cP^{mn} \cP^{pq} \partial_{q} \cP^{rs} \partial_{w} \cP^{tu} \partial_{u} \cP^{vw} \partial_{t} \partial_{p} \partial_{m} \partial_{k} \partial_{i} f \partial_{v} \partial_{r} \partial_{{\ell}} \partial_{j} g \\
+\tfrac{1}{216} \partial_{n} \cP^{ij} \cP^{k{\ell}} \cP^{mn} \partial_{s} \cP^{pq} \partial_{q} \cP^{rs} \partial_{w} \cP^{tu} \partial_{u} \cP^{vw} \partial_{t} \partial_{p} \partial_{m} \partial_{k} \partial_{i} f \partial_{v} \partial_{r} \partial_{{\ell}} \partial_{j} g \\
+\tfrac{1}{135} \partial_{n} \cP^{ij} \partial_{w} \cP^{k{\ell}} \partial_{u} \cP^{mn} \cP^{pq} \partial_{q} \cP^{rs} \partial_{s} \cP^{tu} \cP^{vw} \partial_{r} \partial_{p} \partial_{m} \partial_{k} \partial_{i} f \partial_{v} \partial_{t} \partial_{{\ell}} \partial_{j} g \\
-\tfrac{2}{135} \partial_{n} \cP^{ij} \partial_{w} \cP^{k{\ell}} \partial_{q} \cP^{mn} \partial_{u} \cP^{pq} \cP^{rs} \partial_{s} \cP^{tu} \cP^{vw} \partial_{r} \partial_{p} \partial_{m} \partial_{k} \partial_{i} f \partial_{v} \partial_{t} \partial_{{\ell}} \partial_{j} g \\
-\tfrac{1}{270} \partial_{n} \cP^{ij} \cP^{k{\ell}} \partial_{{\ell}} \cP^{mn} \partial_{w} \cP^{pq} \partial_{q} \cP^{rs} \partial_{s} \cP^{tu} \partial_{u} \cP^{vw} \partial_{p} \partial_{k} \partial_{i} f \partial_{v} \partial_{t} \partial_{r} \partial_{m} \partial_{j} g \\
+\tfrac{11}{7560} \cP^{ij} \partial_{w} \cP^{k{\ell}} \partial_{{\ell}} \cP^{mn} \partial_{n} \cP^{pq} \partial_{q} \cP^{rs} \partial_{s} \cP^{tu} \partial_{u} \cP^{vw} \partial_{m} \partial_{k} \partial_{i} f \partial_{v} \partial_{t} \partial_{r} \partial_{p} \partial_{j} g \\
-\tfrac{2}{315} \partial_{u} \cP^{ij} \partial_{w} \cP^{k{\ell}} \partial_{{\ell}} \cP^{mn} \partial_{n} \cP^{pq} \partial_{q} \cP^{rs} \partial_{s} \cP^{tu} \cP^{vw} \partial_{m} \partial_{k} \partial_{i} f \partial_{v} \partial_{t} \partial_{r} \partial_{p} \partial_{j} g \\
-\tfrac{4}{315} \partial_{{\ell}} \cP^{ij} \partial_{w} \cP^{k{\ell}} \cP^{mn} \partial_{n} \cP^{pq} \partial_{q} \cP^{rs} \partial_{s} \cP^{tu} \partial_{u} \cP^{vw} \partial_{m} \partial_{k} \partial_{i} f \partial_{v} \partial_{t} \partial_{r} \partial_{p} \partial_{j} g \\
-\tfrac{4}{315} \partial_{n} \cP^{ij} \partial_{q} \cP^{k{\ell}} \partial_{s} \cP^{mn} \partial_{v} \cP^{pq} \partial_{u} \cP^{rs} \partial_{w} \cP^{tu} \cP^{vw} \partial_{m} \partial_{k} \partial_{i} f \partial_{t} \partial_{r} \partial_{p} \partial_{{\ell}} \partial_{j} g \\
-\tfrac{31}{7560} \cP^{ij} \partial_{w} \cP^{k{\ell}} \partial_{{\ell}} \cP^{mn} \partial_{n} \cP^{pq} \partial_{q} \cP^{rs} \partial_{s} \cP^{tu} \partial_{u} \cP^{vw} \partial_{r} \partial_{k} \partial_{i} f \partial_{v} \partial_{t} \partial_{p} \partial_{m} \partial_{j} g \\
+\tfrac{8}{945} \partial_{s} \cP^{ij} \partial_{u} \cP^{k{\ell}} \partial_{{\ell}} \cP^{mn} \partial_{n} \cP^{pq} \partial_{q} \cP^{rs} \partial_{w} \cP^{tu} \cP^{vw} \partial_{m} \partial_{k} \partial_{i} f \partial_{v} \partial_{t} \partial_{r} \partial_{p} \partial_{j} g \\
-\tfrac{2}{945} \partial_{{\ell}} \cP^{ij} \partial_{n} \cP^{k{\ell}} \partial_{q} \cP^{mn} \partial_{s} \cP^{pq} \partial_{u} \cP^{rs} \partial_{w} \cP^{tu} \cP^{vw} \partial_{m} \partial_{k} \partial_{i} f \partial_{v} \partial_{t} \partial_{r} \partial_{p} \partial_{j} g \\
-\tfrac{1}{540} \cP^{ij} \partial_{n} \cP^{k{\ell}} \partial_{{\ell}} \cP^{mn} \partial_{w} \cP^{pq} \partial_{q} \cP^{rs} \partial_{s} \cP^{tu} \partial_{u} \cP^{vw} \partial_{p} \partial_{k} \partial_{i} f \partial_{v} \partial_{t} \partial_{r} \partial_{m} \partial_{j} g \\
+\tfrac{37}{15120} \partial_{s} \cP^{ij} \cP^{k{\ell}} \partial_{v} \cP^{mn} \partial_{n} \cP^{pq} \partial_{q} \cP^{rs} \partial_{w} \cP^{tu} \partial_{u} \cP^{vw} \partial_{m} \partial_{k} \partial_{i} f \partial_{t} \partial_{r} \partial_{p} \partial_{{\ell}} \partial_{j} g \\
-\tfrac{1}{135} \partial_{u} \cP^{ij} \partial_{n} \cP^{k{\ell}} \partial_{{\ell}} \cP^{mn} \partial_{w} \cP^{pq} \partial_{q} \cP^{rs} \partial_{s} \cP^{tu} \cP^{vw} \partial_{p} \partial_{k} \partial_{i} f \partial_{v} \partial_{t} \partial_{r} \partial_{m} \partial_{j} g \\
+\tfrac{1}{270} \partial_{n} \cP^{ij} \partial_{u} \cP^{k{\ell}} \partial_{{\ell}} \cP^{mn} \partial_{s} \cP^{pq} \partial_{q} \cP^{rs} \partial_{w} \cP^{tu} \cP^{vw} \partial_{p} \partial_{k} \partial_{i} f \partial_{v} \partial_{t} \partial_{r} \partial_{m} \partial_{j} g \\
-\tfrac{1}{270} \partial_{{\ell}} \cP^{ij} \partial_{w} \cP^{k{\ell}} \partial_{u} \cP^{mn} \partial_{n} \cP^{pq} \partial_{q} \cP^{rs} \partial_{s} \cP^{tu} \cP^{vw} \partial_{r} \partial_{p} \partial_{m} \partial_{k} \partial_{i} f \partial_{v} \partial_{t} \partial_{j} g \\
+\tfrac{11}{7560} \cP^{ij} \partial_{w} \cP^{k{\ell}} \partial_{{\ell}} \cP^{mn} \partial_{n} \cP^{pq} \partial_{q} \cP^{rs} \partial_{s} \cP^{tu} \partial_{u} \cP^{vw} \partial_{r} \partial_{p} \partial_{m} \partial_{k} \partial_{i} f \partial_{v} \partial_{t} \partial_{j} g \\
-\tfrac{2}{315} \partial_{{\ell}} \cP^{ij} \partial_{n} \cP^{k{\ell}} \partial_{q} \cP^{mn} \partial_{w} \cP^{pq} \cP^{rs} \partial_{s} \cP^{tu} \partial_{u} \cP^{vw} \partial_{r} \partial_{p} \partial_{m} \partial_{k} \partial_{i} f \partial_{v} \partial_{t} \partial_{j} g \\
-\tfrac{4}{315} \partial_{u} \cP^{ij} \partial_{w} \cP^{k{\ell}} \partial_{{\ell}} \cP^{mn} \partial_{n} \cP^{pq} \partial_{q} \cP^{rs} \partial_{s} \cP^{tu} \cP^{vw} \partial_{r} \partial_{p} \partial_{m} \partial_{k} \partial_{i} f \partial_{v} \partial_{t} \partial_{j} g \\
+\tfrac{4}{315} \partial_{n} \cP^{ij} \partial_{u} \cP^{k{\ell}} \partial_{v} \cP^{mn} \partial_{w} \cP^{pq} \partial_{q} \cP^{rs} \partial_{s} \cP^{tu} \cP^{vw} \partial_{r} \partial_{p} \partial_{m} \partial_{k} \partial_{i} f \partial_{t} \partial_{{\ell}} \partial_{j} g
\end{gather*}
}
{\footnotesize
\begin{gather*}
-\tfrac{31}{7560} \cP^{ij} \partial_{w} \cP^{k{\ell}} \partial_{{\ell}} \cP^{mn} \partial_{n} \cP^{pq} \partial_{q} \cP^{rs} \partial_{s} \cP^{tu} \partial_{u} \cP^{vw} \partial_{t} \partial_{r} \partial_{m} \partial_{k} \partial_{i} f \partial_{v} \partial_{p} \partial_{j} g \\
+\tfrac{8}{945} \partial_{{\ell}} \cP^{ij} \partial_{n} \cP^{k{\ell}} \partial_{w} \cP^{mn} \cP^{pq} \partial_{q} \cP^{rs} \partial_{s} \cP^{tu} \partial_{u} \cP^{vw} \partial_{r} \partial_{p} \partial_{m} \partial_{k} \partial_{i} f \partial_{v} \partial_{t} \partial_{j} g \\
-\tfrac{2}{945} \partial_{w} \cP^{ij} \cP^{k{\ell}} \partial_{{\ell}} \cP^{mn} \partial_{n} \cP^{pq} \partial_{q} \cP^{rs} \partial_{s} \cP^{tu} \partial_{u} \cP^{vw} \partial_{r} \partial_{p} \partial_{m} \partial_{k} \partial_{i} f \partial_{v} \partial_{t} \partial_{j} g \\
-\tfrac{1}{540} \cP^{ij} \partial_{s} \cP^{k{\ell}} \partial_{{\ell}} \cP^{mn} \partial_{n} \cP^{pq} \partial_{q} \cP^{rs} \partial_{w} \cP^{tu} \partial_{u} \cP^{vw} \partial_{t} \partial_{p} \partial_{m} \partial_{k} \partial_{i} f \partial_{v} \partial_{r} \partial_{j} g \\
+\tfrac{37}{15120} \partial_{n} \cP^{ij} \cP^{k{\ell}} \partial_{q} \cP^{mn} \partial_{w} \cP^{pq} \partial_{t} \cP^{rs} \partial_{s} \cP^{tu} \partial_{u} \cP^{vw} \partial_{r} \partial_{p} \partial_{m} \partial_{k} \partial_{i} f \partial_{v} \partial_{{\ell}} \partial_{j} g \\
+\tfrac{1}{270} \partial_{{\ell}} \cP^{ij} \partial_{s} \cP^{k{\ell}} \cP^{mn} \partial_{n} \cP^{pq} \partial_{q} \cP^{rs} \partial_{w} \cP^{tu} \partial_{u} \cP^{vw} \partial_{t} \partial_{p} \partial_{m} \partial_{k} \partial_{i} f \partial_{v} \partial_{r} \partial_{j} g \\
-\tfrac{1}{135} \partial_{{\ell}} \cP^{ij} \partial_{n} \cP^{k{\ell}} \partial_{s} \cP^{mn} \cP^{pq} \partial_{q} \cP^{rs} \partial_{w} \cP^{tu} \partial_{u} \cP^{vw} \partial_{t} \partial_{p} \partial_{m} \partial_{k} \partial_{i} f \partial_{v} \partial_{r} \partial_{j} g \\
-\tfrac{1}{72} \partial_{w} \cP^{ij} \cP^{k{\ell}} \cP^{mn} \cP^{pq} \cP^{rs} \cP^{tu} \partial_{u} \cP^{vw} \partial_{t} \partial_{r} \partial_{p} \partial_{m} \partial_{k} \partial_{i} f \partial_{v} \partial_{s} \partial_{q} \partial_{n} \partial_{{\ell}} \partial_{j} g \\
-\tfrac{1}{54} \partial_{u} \cP^{ij} \partial_{w} \cP^{k{\ell}} \cP^{mn} \cP^{pq} \cP^{rs} \cP^{tu} \cP^{vw} \partial_{t} \partial_{r} \partial_{p} \partial_{m} \partial_{k} \partial_{i} f \partial_{v} \partial_{s} \partial_{q} \partial_{n} \partial_{{\ell}} \partial_{j} g \\
-\tfrac{1}{720} \cP^{ij} \cP^{k{\ell}} \cP^{mn} \cP^{pq} \cP^{rs} \partial_{w} \cP^{tu} \partial_{u} \cP^{vw} \partial_{t} \partial_{r} \partial_{p} \partial_{m} \partial_{k} \partial_{i} f \partial_{v} \partial_{s} \partial_{q} \partial_{n} \partial_{{\ell}} \partial_{j} g \\
-\tfrac{1}{45} \partial_{u} \cP^{ij} \partial_{v} \cP^{k{\ell}} \cP^{mn} \cP^{pq} \cP^{rs} \partial_{w} \cP^{tu} \cP^{vw} \partial_{r} \partial_{p} \partial_{m} \partial_{k} \partial_{i} f \partial_{t} \partial_{s} \partial_{q} \partial_{n} \partial_{{\ell}} \partial_{j} g \\
+\tfrac{2}{135} \partial_{w} \cP^{ij} \cP^{k{\ell}} \cP^{mn} \cP^{pq} \cP^{rs} \partial_{s} \cP^{tu} \partial_{u} \cP^{vw} \partial_{r} \partial_{p} \partial_{m} \partial_{k} \partial_{i} f \partial_{v} \partial_{t} \partial_{q} \partial_{n} \partial_{{\ell}} \partial_{j} g \\
+\tfrac{1}{18} \partial_{s} \cP^{ij} \partial_{u} \cP^{k{\ell}} \cP^{mn} \cP^{pq} \partial_{w} \cP^{rs} \cP^{tu} \cP^{vw} \partial_{r} \partial_{p} \partial_{m} \partial_{k} \partial_{i} f \partial_{v} \partial_{t} \partial_{q} \partial_{n} \partial_{{\ell}} \partial_{j} g \\
+\tfrac{1}{54} \partial_{s} \cP^{ij} \partial_{u} \cP^{k{\ell}} \partial_{w} \cP^{mn} \cP^{pq} \cP^{rs} \cP^{tu} \cP^{vw} \partial_{r} \partial_{p} \partial_{m} \partial_{k} \partial_{i} f \partial_{v} \partial_{t} \partial_{q} \partial_{n} \partial_{{\ell}} \partial_{j} g \\
+\tfrac{1}{108} \partial_{w} \cP^{ij} \cP^{k{\ell}} \cP^{mn} \cP^{pq} \partial_{u} \cP^{rs} \partial_{s} \cP^{tu} \cP^{vw} \partial_{r} \partial_{p} \partial_{m} \partial_{k} \partial_{i} f \partial_{v} \partial_{t} \partial_{q} \partial_{n} \partial_{{\ell}} \partial_{j} g \\
-\tfrac{1}{45} \partial_{u} \cP^{ij} \partial_{v} \cP^{k{\ell}} \cP^{mn} \cP^{pq} \cP^{rs} \partial_{w} \cP^{tu} \cP^{vw} \partial_{t} \partial_{r} \partial_{p} \partial_{m} \partial_{k} \partial_{i} f \partial_{s} \partial_{q} \partial_{n} \partial_{{\ell}} \partial_{j} g \\
-\tfrac{2}{135} \partial_{s} \cP^{ij} \cP^{k{\ell}} \cP^{mn} \cP^{pq} \partial_{u} \cP^{rs} \partial_{w} \cP^{tu} \cP^{vw} \partial_{t} \partial_{r} \partial_{p} \partial_{m} \partial_{k} \partial_{i} f \partial_{v} \partial_{q} \partial_{n} \partial_{{\ell}} \partial_{j} g \\
-\tfrac{1}{18} \partial_{s} \cP^{ij} \partial_{w} \cP^{k{\ell}} \cP^{mn} \cP^{pq} \cP^{rs} \cP^{tu} \partial_{u} \cP^{vw} \partial_{t} \partial_{r} \partial_{p} \partial_{m} \partial_{k} \partial_{i} f \partial_{v} \partial_{q} \partial_{n} \partial_{{\ell}} \partial_{j} g \\
-\tfrac{1}{54} \partial_{s} \cP^{ij} \partial_{u} \cP^{k{\ell}} \partial_{w} \cP^{mn} \cP^{pq} \cP^{rs} \cP^{tu} \cP^{vw} \partial_{t} \partial_{r} \partial_{p} \partial_{m} \partial_{k} \partial_{i} f \partial_{v} \partial_{q} \partial_{n} \partial_{{\ell}} \partial_{j} g \\
-\tfrac{1}{108} \partial_{s} \cP^{ij} \cP^{k{\ell}} \cP^{mn} \cP^{pq} \cP^{rs} \partial_{w} \cP^{tu} \partial_{u} \cP^{vw} \partial_{t} \partial_{r} \partial_{p} \partial_{m} \partial_{k} \partial_{i} f \partial_{v} \partial_{q} \partial_{n} \partial_{{\ell}} \partial_{j} g \\
+\tfrac{1}{540} \cP^{ij} \cP^{k{\ell}} \cP^{mn} \partial_{w} \cP^{pq} \partial_{q} \cP^{rs} \partial_{s} \cP^{tu} \partial_{u} \cP^{vw} \partial_{p} \partial_{m} \partial_{k} \partial_{i} f \partial_{v} \partial_{t} \partial_{r} \partial_{n} \partial_{{\ell}} \partial_{j} g \\
+\tfrac{2}{45} \partial_{s} \cP^{ij} \partial_{u} \cP^{k{\ell}} \partial_{v} \cP^{mn} \cP^{pq} \partial_{w} \cP^{rs} \cP^{tu} \cP^{vw} \partial_{p} \partial_{m} \partial_{k} \partial_{i} f \partial_{t} \partial_{r} \partial_{q} \partial_{n} \partial_{{\ell}} \partial_{j} g \\
-\tfrac{4}{135} \partial_{u} \cP^{ij} \partial_{w} \cP^{k{\ell}} \cP^{mn} \cP^{pq} \partial_{q} \cP^{rs} \partial_{s} \cP^{tu} \cP^{vw} \partial_{p} \partial_{m} \partial_{k} \partial_{i} f \partial_{v} \partial_{t} \partial_{r} \partial_{n} \partial_{{\ell}} \partial_{j} g \\
+\tfrac{1}{135} \partial_{q} \cP^{ij} \partial_{s} \cP^{k{\ell}} \cP^{mn} \cP^{pq} \partial_{u} \cP^{rs} \partial_{w} \cP^{tu} \cP^{vw} \partial_{p} \partial_{m} \partial_{k} \partial_{i} f \partial_{v} \partial_{t} \partial_{r} \partial_{n} \partial_{{\ell}} \partial_{j} g \\
-\tfrac{1}{54} \partial_{q} \cP^{ij} \partial_{s} \cP^{k{\ell}} \partial_{u} \cP^{mn} \partial_{w} \cP^{pq} \cP^{rs} \cP^{tu} \cP^{vw} \partial_{p} \partial_{m} \partial_{k} \partial_{i} f \partial_{v} \partial_{t} \partial_{r} \partial_{n} \partial_{{\ell}} \partial_{j} g \\
-\tfrac{1}{108} \partial_{u} \cP^{ij} \partial_{w} \cP^{k{\ell}} \cP^{mn} \partial_{s} \cP^{pq} \partial_{q} \cP^{rs} \cP^{tu} \cP^{vw} \partial_{p} \partial_{m} \partial_{k} \partial_{i} f \partial_{v} \partial_{t} \partial_{r} \partial_{n} \partial_{{\ell}} \partial_{j} g \\
+\tfrac{1}{45} \partial_{u} \cP^{ij} \cP^{k{\ell}} \cP^{mn} \partial_{w} \cP^{pq} \partial_{q} \cP^{rs} \partial_{s} \cP^{tu} \cP^{vw} \partial_{p} \partial_{m} \partial_{k} \partial_{i} f \partial_{v} \partial_{t} \partial_{r} \partial_{n} \partial_{{\ell}} \partial_{j} g \\
-\tfrac{1}{90} \partial_{s} \cP^{ij} \cP^{k{\ell}} \cP^{mn} \partial_{u} \cP^{pq} \partial_{q} \cP^{rs} \partial_{w} \cP^{tu} \cP^{vw} \partial_{p} \partial_{m} \partial_{k} \partial_{i} f \partial_{v} \partial_{t} \partial_{r} \partial_{n} \partial_{{\ell}} \partial_{j} g \\
+\tfrac{1}{540} \cP^{ij} \cP^{k{\ell}} \cP^{mn} \partial_{w} \cP^{pq} \partial_{q} \cP^{rs} \partial_{s} \cP^{tu} \partial_{u} \cP^{vw} \partial_{t} \partial_{r} \partial_{p} \partial_{m} \partial_{k} \partial_{i} f \partial_{v} \partial_{n} \partial_{{\ell}} \partial_{j} g \\
-\tfrac{2}{45} \partial_{s} \cP^{ij} \partial_{u} \cP^{k{\ell}} \partial_{v} \cP^{mn} \cP^{pq} \partial_{w} \cP^{rs} \cP^{tu} \cP^{vw} \partial_{t} \partial_{r} \partial_{p} \partial_{m} \partial_{k} \partial_{i} f \partial_{q} \partial_{n} \partial_{{\ell}} \partial_{j} g \\
-\tfrac{4}{135} \partial_{q} \cP^{ij} \partial_{s} \cP^{k{\ell}} \cP^{mn} \partial_{u} \cP^{pq} \cP^{rs} \partial_{w} \cP^{tu} \cP^{vw} \partial_{t} \partial_{r} \partial_{p} \partial_{m} \partial_{k} \partial_{i} f \partial_{v} \partial_{n} \partial_{{\ell}} \partial_{j} g \\
+\tfrac{1}{135} \partial_{q} \cP^{ij} \partial_{w} \cP^{k{\ell}} \cP^{mn} \partial_{s} \cP^{pq} \partial_{u} \cP^{rs} \cP^{tu} \cP^{vw} \partial_{t} \partial_{r} \partial_{p} \partial_{m} \partial_{k} \partial_{i} f \partial_{v} \partial_{n} \partial_{{\ell}} \partial_{j} g \\
-\tfrac{1}{54} \partial_{q} \cP^{ij} \partial_{s} \cP^{k{\ell}} \partial_{w} \cP^{mn} \cP^{pq} \cP^{rs} \cP^{tu} \partial_{u} \cP^{vw} \partial_{t} \partial_{r} \partial_{p} \partial_{m} \partial_{k} \partial_{i} f \partial_{v} \partial_{n} \partial_{{\ell}} \partial_{j} g \\
-\tfrac{1}{108} \partial_{q} \cP^{ij} \partial_{s} \cP^{k{\ell}} \cP^{mn} \cP^{pq} \cP^{rs} \partial_{w} \cP^{tu} \partial_{u} \cP^{vw} \partial_{t} \partial_{r} \partial_{p} \partial_{m} \partial_{k} \partial_{i} f \partial_{v} \partial_{n} \partial_{{\ell}} \partial_{j} g \\
-\tfrac{1}{90} \partial_{q} \cP^{ij} \cP^{k{\ell}} \cP^{mn} \partial_{w} \cP^{pq} \cP^{rs} \partial_{s} \cP^{tu} \partial_{u} \cP^{vw} \partial_{t} \partial_{r} \partial_{p} \partial_{m} \partial_{k} \partial_{i} f \partial_{v} \partial_{n} \partial_{{\ell}} \partial_{j} g \\
+\tfrac{1}{45} \partial_{q} \cP^{ij} \cP^{k{\ell}} \cP^{mn} \partial_{s} \cP^{pq} \partial_{w} \cP^{rs} \cP^{tu} \partial_{u} \cP^{vw} \partial_{t} \partial_{r} \partial_{p} \partial_{m} \partial_{k} \partial_{i} f \partial_{v} \partial_{n} \partial_{{\ell}} \partial_{j} g \\
+\tfrac{2}{315} \partial_{w} \cP^{ij} \cP^{k{\ell}} \cP^{mn} \partial_{n} \cP^{pq} \partial_{q} \cP^{rs} \partial_{s} \cP^{tu} \partial_{u} \cP^{vw} \partial_{m} \partial_{k} \partial_{i} f \partial_{v} \partial_{t} \partial_{r} \partial_{p} \partial_{{\ell}} \partial_{j} g \\
-\tfrac{1}{270} \partial_{w} \cP^{ij} \cP^{k{\ell}} \partial_{u} \cP^{mn} \partial_{n} \cP^{pq} \partial_{q} \cP^{rs} \partial_{s} \cP^{tu} \cP^{vw} \partial_{m} \partial_{k} \partial_{i} f \partial_{v} \partial_{t} \partial_{r} \partial_{p} \partial_{{\ell}} \partial_{j} g
\end{gather*}
}
{\enlargethispage{1.3\baselineskip}
\footnotesize
\begin{gather*}
-\tfrac{1}{135} \partial_{n} \cP^{ij} \partial_{q} \cP^{k{\ell}} \partial_{s} \cP^{mn} \partial_{u} \cP^{pq} \cP^{rs} \partial_{w} \cP^{tu} \cP^{vw} \partial_{m} \partial_{k} \partial_{i} f \partial_{v} \partial_{t} \partial_{r} \partial_{p} \partial_{{\ell}} \partial_{j} g \\
-\tfrac{2}{105} \partial_{s} \cP^{ij} \cP^{k{\ell}} \partial_{u} \cP^{mn} \partial_{n} \cP^{pq} \partial_{q} \cP^{rs} \partial_{w} \cP^{tu} \cP^{vw} \partial_{m} \partial_{k} \partial_{i} f \partial_{v} \partial_{t} \partial_{r} \partial_{p} \partial_{{\ell}} \partial_{j} g \\
-\tfrac{1}{270} \partial_{s} \cP^{ij} \cP^{k{\ell}} \partial_{q} \cP^{mn} \partial_{n} \cP^{pq} \partial_{u} \cP^{rs} \partial_{w} \cP^{tu} \cP^{vw} \partial_{m} \partial_{k} \partial_{i} f \partial_{v} \partial_{t} \partial_{r} \partial_{p} \partial_{{\ell}} \partial_{j} g \\
-\tfrac{2}{135} \partial_{s} \cP^{ij} \partial_{u} \cP^{k{\ell}} \partial_{w} \cP^{mn} \partial_{n} \cP^{pq} \partial_{q} \cP^{rs} \cP^{tu} \cP^{vw} \partial_{m} \partial_{k} \partial_{i} f \partial_{v} \partial_{t} \partial_{r} \partial_{p} \partial_{{\ell}} \partial_{j} g \\
+\tfrac{1}{135} \partial_{q} \cP^{ij} \partial_{s} \cP^{k{\ell}} \partial_{u} \cP^{mn} \partial_{n} \cP^{pq} \cP^{rs} \partial_{w} \cP^{tu} \cP^{vw} \partial_{m} \partial_{k} \partial_{i} f \partial_{v} \partial_{t} \partial_{r} \partial_{p} \partial_{{\ell}} \partial_{j} g \\
-\tfrac{2}{315} \partial_{q} \cP^{ij} \partial_{v} \cP^{k{\ell}} \cP^{mn} \partial_{s} \cP^{pq} \partial_{w} \cP^{rs} \cP^{tu} \partial_{u} \cP^{vw} \partial_{m} \partial_{k} \partial_{i} f \partial_{t} \partial_{r} \partial_{p} \partial_{n} \partial_{{\ell}} \partial_{j} g \\
-\tfrac{2}{315} \partial_{n} \cP^{ij} \cP^{k{\ell}} \partial_{q} \cP^{mn} \partial_{s} \cP^{pq} \partial_{u} \cP^{rs} \partial_{w} \cP^{tu} \cP^{vw} \partial_{t} \partial_{r} \partial_{p} \partial_{m} \partial_{k} \partial_{i} f \partial_{v} \partial_{{\ell}} \partial_{j} g \\
+\tfrac{1}{270} \partial_{n} \cP^{ij} \cP^{k{\ell}} \cP^{mn} \partial_{w} \cP^{pq} \partial_{q} \cP^{rs} \partial_{s} \cP^{tu} \partial_{u} \cP^{vw} \partial_{t} \partial_{r} \partial_{p} \partial_{m} \partial_{k} \partial_{i} f \partial_{v} \partial_{{\ell}} \partial_{j} g \\
+\tfrac{1}{135} \partial_{n} \cP^{ij} \partial_{w} \cP^{k{\ell}} \partial_{q} \cP^{mn} \partial_{s} \cP^{pq} \cP^{rs} \cP^{tu} \partial_{u} \cP^{vw} \partial_{t} \partial_{r} \partial_{p} \partial_{m} \partial_{k} \partial_{i} f \partial_{v} \partial_{{\ell}} \partial_{j} g \\
+\tfrac{2}{105} \partial_{n} \cP^{ij} \cP^{k{\ell}} \partial_{q} \cP^{mn} \partial_{w} \cP^{pq} \cP^{rs} \partial_{s} \cP^{tu} \partial_{u} \cP^{vw} \partial_{t} \partial_{r} \partial_{p} \partial_{m} \partial_{k} \partial_{i} f \partial_{v} \partial_{{\ell}} \partial_{j} g \\
+\tfrac{1}{270} \partial_{n} \cP^{ij} \cP^{k{\ell}} \partial_{q} \cP^{mn} \partial_{s} \cP^{pq} \cP^{rs} \partial_{w} \cP^{tu} \partial_{u} \cP^{vw} \partial_{t} \partial_{r} \partial_{p} \partial_{m} \partial_{k} \partial_{i} f \partial_{v} \partial_{{\ell}} \partial_{j} g \\
-\tfrac{1}{135} \partial_{n} \cP^{ij} \partial_{q} \cP^{k{\ell}} \partial_{w} \cP^{mn} \cP^{pq} \cP^{rs} \partial_{s} \cP^{tu} \partial_{u} \cP^{vw} \partial_{t} \partial_{r} \partial_{p} \partial_{m} \partial_{k} \partial_{i} f \partial_{v} \partial_{{\ell}} \partial_{j} g \\
+\tfrac{2}{135} \partial_{n} \cP^{ij} \partial_{q} \cP^{k{\ell}} \partial_{s} \cP^{mn} \cP^{pq} \partial_{w} \cP^{rs} \cP^{tu} \partial_{u} \cP^{vw} \partial_{t} \partial_{r} \partial_{p} \partial_{m} \partial_{k} \partial_{i} f \partial_{v} \partial_{{\ell}} \partial_{j} g \\
-\tfrac{2}{315} \partial_{q} \cP^{ij} \partial_{v} \cP^{k{\ell}} \cP^{mn} \partial_{s} \cP^{pq} \partial_{w} \cP^{rs} \cP^{tu} \partial_{u} \cP^{vw} \partial_{t} \partial_{r} \partial_{p} \partial_{m} \partial_{k} \partial_{i} f \partial_{n} \partial_{{\ell}} \partial_{j} g \\
-\tfrac{1}{945} \cP^{ij} \partial_{w} \cP^{k{\ell}} \partial_{{\ell}} \cP^{mn} \partial_{n} \cP^{pq} \partial_{q} \cP^{rs} \partial_{s} \cP^{tu} \partial_{u} \cP^{vw} \partial_{k} \partial_{i} f \partial_{v} \partial_{t} \partial_{r} \partial_{p} \partial_{m} \partial_{j} g \\
-\tfrac{4}{945} \partial_{u} \cP^{ij} \partial_{w} \cP^{k{\ell}} \partial_{{\ell}} \cP^{mn} \partial_{n} \cP^{pq} \partial_{q} \cP^{rs} \partial_{s} \cP^{tu} \cP^{vw} \partial_{k} \partial_{i} f \partial_{v} \partial_{t} \partial_{r} \partial_{p} \partial_{m} \partial_{j} g \\
+\tfrac{2}{945} \partial_{n} \cP^{ij} \partial_{q} \cP^{k{\ell}} \partial_{s} \cP^{mn} \partial_{v} \cP^{pq} \partial_{u} \cP^{rs} \partial_{w} \cP^{tu} \cP^{vw} \partial_{k} \partial_{i} f \partial_{t} \partial_{r} \partial_{p} \partial_{m} \partial_{{\ell}} \partial_{j} g \\
+\tfrac{4}{945} \partial_{q} \cP^{ij} \partial_{s} \cP^{k{\ell}} \partial_{{\ell}} \cP^{mn} \partial_{n} \cP^{pq} \partial_{u} \cP^{rs} \partial_{w} \cP^{tu} \cP^{vw} \partial_{k} \partial_{i} f \partial_{v} \partial_{t} \partial_{r} \partial_{p} \partial_{m} \partial_{j} g \\
-\tfrac{2}{945} \partial_{n} \cP^{ij} \partial_{q} \cP^{k{\ell}} \partial_{{\ell}} \cP^{mn} \partial_{s} \cP^{pq} \partial_{u} \cP^{rs} \partial_{w} \cP^{tu} \cP^{vw} \partial_{k} \partial_{i} f \partial_{v} \partial_{t} \partial_{r} \partial_{p} \partial_{m} \partial_{j} g \\
-\tfrac{1}{945} \cP^{ij} \partial_{w} \cP^{k{\ell}} \partial_{{\ell}} \cP^{mn} \partial_{n} \cP^{pq} \partial_{q} \cP^{rs} \partial_{s} \cP^{tu} \partial_{u} \cP^{vw} \partial_{t} \partial_{r} \partial_{p} \partial_{m} \partial_{k} \partial_{i} f \partial_{v} \partial_{j} g \\
-\tfrac{4}{945} \partial_{{\ell}} \cP^{ij} \partial_{n} \cP^{k{\ell}} \partial_{q} \cP^{mn} \partial_{s} \cP^{pq} \partial_{w} \cP^{rs} \cP^{tu} \partial_{u} \cP^{vw} \partial_{t} \partial_{r} \partial_{p} \partial_{m} \partial_{k} \partial_{i} f \partial_{v} \partial_{j} g \\
-\tfrac{2}{945} \partial_{n} \cP^{ij} \partial_{q} \cP^{k{\ell}} \partial_{s} \cP^{mn} \partial_{v} \cP^{pq} \partial_{u} \cP^{rs} \partial_{w} \cP^{tu} \cP^{vw} \partial_{t} \partial_{r} \partial_{p} \partial_{m} \partial_{k} \partial_{i} f \partial_{{\ell}} \partial_{j} g \\
-\tfrac{2}{945} \partial_{{\ell}} \cP^{ij} \partial_{w} \cP^{k{\ell}} \cP^{mn} \partial_{n} \cP^{pq} \partial_{q} \cP^{rs} \partial_{s} \cP^{tu} \partial_{u} \cP^{vw} \partial_{t} \partial_{r} \partial_{p} \partial_{m} \partial_{k} \partial_{i} f \partial_{v} \partial_{j} g \\
+\tfrac{4}{945} \partial_{{\ell}} \cP^{ij} \partial_{n} \cP^{k{\ell}} \partial_{w} \cP^{mn} \cP^{pq} \partial_{q} \cP^{rs} \partial_{s} \cP^{tu} \partial_{u} \cP^{vw} \partial_{t} \partial_{r} \partial_{p} \partial_{m} \partial_{k} \partial_{i} f \partial_{v} \partial_{j} g \\
+\tfrac{1}{5040} \cP^{ij} \cP^{k{\ell}} \cP^{mn} \cP^{pq} \cP^{rs} \cP^{tu} \cP^{vw} \partial_{v} \partial_{t} \partial_{r} \partial_{p} \partial_{m} \partial_{k} \partial_{i} f \partial_{w} \partial_{u} \partial_{s} \partial_{q} \partial_{n} \partial_{{\ell}} \partial_{j} g \\
-\tfrac{1}{360} \partial_{w} \cP^{ij} \cP^{k{\ell}} \cP^{mn} \cP^{pq} \cP^{rs} \cP^{tu} \cP^{vw} \partial_{t} \partial_{r} \partial_{p} \partial_{m} \partial_{k} \partial_{i} f \partial_{v} \partial_{u} \partial_{s} \partial_{q} \partial_{n} \partial_{{\ell}} \partial_{j} g \\
+\tfrac{1}{360} \partial_{w} \cP^{ij} \cP^{k{\ell}} \cP^{mn} \cP^{pq} \cP^{rs} \cP^{tu} \cP^{vw} \partial_{v} \partial_{t} \partial_{r} \partial_{p} \partial_{m} \partial_{k} \partial_{i} f \partial_{u} \partial_{s} \partial_{q} \partial_{n} \partial_{{\ell}} \partial_{j} g \\
+\tfrac{1}{108} \partial_{u} \cP^{ij} \partial_{w} \cP^{k{\ell}} \cP^{mn} \cP^{pq} \cP^{rs} \cP^{tu} \cP^{vw} \partial_{r} \partial_{p} \partial_{m} \partial_{k} \partial_{i} f \partial_{v} \partial_{t} \partial_{s} \partial_{q} \partial_{n} \partial_{{\ell}} \partial_{j} g \\
+\tfrac{1}{108} \partial_{u} \cP^{ij} \partial_{w} \cP^{k{\ell}} \cP^{mn} \cP^{pq} \cP^{rs} \cP^{tu} \cP^{vw} \partial_{v} \partial_{t} \partial_{r} \partial_{p} \partial_{m} \partial_{k} \partial_{i} f \partial_{s} \partial_{q} \partial_{n} \partial_{{\ell}} \partial_{j} g \\
+\tfrac{1}{270} \partial_{s} \cP^{ij} \cP^{k{\ell}} \cP^{mn} \cP^{pq} \partial_{u} \cP^{rs} \partial_{w} \cP^{tu} \cP^{vw} \partial_{p} \partial_{m} \partial_{k} \partial_{i} f \partial_{v} \partial_{t} \partial_{r} \partial_{q} \partial_{n} \partial_{{\ell}} \partial_{j} g \\
-\tfrac{1}{162} \partial_{s} \cP^{ij} \partial_{u} \cP^{k{\ell}} \partial_{w} \cP^{mn} \cP^{pq} \cP^{rs} \cP^{tu} \cP^{vw} \partial_{p} \partial_{m} \partial_{k} \partial_{i} f \partial_{v} \partial_{t} \partial_{r} \partial_{q} \partial_{n} \partial_{{\ell}} \partial_{j} g \\
-\tfrac{1}{270} \partial_{s} \cP^{ij} \cP^{k{\ell}} \cP^{mn} \cP^{pq} \partial_{u} \cP^{rs} \partial_{w} \cP^{tu} \cP^{vw} \partial_{v} \partial_{t} \partial_{r} \partial_{p} \partial_{m} \partial_{k} \partial_{i} f \partial_{q} \partial_{n} \partial_{{\ell}} \partial_{j} g \\
+\tfrac{1}{162} \partial_{s} \cP^{ij} \partial_{u} \cP^{k{\ell}} \partial_{w} \cP^{mn} \cP^{pq} \cP^{rs} \cP^{tu} \cP^{vw} \partial_{v} \partial_{t} \partial_{r} \partial_{p} \partial_{m} \partial_{k} \partial_{i} f \partial_{q} \partial_{n} \partial_{{\ell}} \partial_{j} g \\
-\tfrac{1}{135} \partial_{q} \cP^{ij} \partial_{s} \cP^{k{\ell}} \cP^{mn} \partial_{u} \cP^{pq} \cP^{rs} \partial_{w} \cP^{tu} \cP^{vw} \partial_{m} \partial_{k} \partial_{i} f \partial_{v} \partial_{t} \partial_{r} \partial_{p} \partial_{n} \partial_{{\ell}} \partial_{j} g \\
-\tfrac{1}{135} \partial_{q} \cP^{ij} \partial_{s} \cP^{k{\ell}} \cP^{mn} \partial_{u} \cP^{pq} \cP^{rs} \partial_{w} \cP^{tu} \cP^{vw} \partial_{v} \partial_{t} \partial_{r} \partial_{p} \partial_{m} \partial_{k} \partial_{i} f \partial_{n} \partial_{{\ell}} \partial_{j} g \\
-\tfrac{2}{945} \partial_{n} \cP^{ij} \cP^{k{\ell}} \partial_{q} \cP^{mn} \partial_{s} \cP^{pq} \partial_{u} \cP^{rs} \partial_{w} \cP^{tu} \cP^{vw} \partial_{k} \partial_{i} f \partial_{v} \partial_{t} \partial_{r} \partial_{p} \partial_{m} \partial_{{\ell}} \partial_{j} g \\
+\tfrac{2}{945} \partial_{n} \cP^{ij} \cP^{k{\ell}} \partial_{q} \cP^{mn} \partial_{s} \cP^{pq} \partial_{u} \cP^{rs} \partial_{w} \cP^{tu} \cP^{vw} \partial_{v} \partial_{t} \partial_{r} \partial_{p} \partial_{m} \partial_{k} \partial_{i} f \partial_{{\ell}} \partial_{j} g
\big) \lefteqn{ {}+ \bar{o}(\hbar^7).}
\end{gather*}

}\normalsize%
\label{lastpage}

\end{document}